\documentclass[11pt,a4paper]{article}

\usepackage{amssymb,amsmath,amscd,enumerate,hyperref,verbatim}
\usepackage{physics}
\usepackage[utf8]{inputenc}
\usepackage[mathscr]{euscript}
\usepackage[left=1in,right=1in,top=1in,bottom=1in]{geometry}
\usepackage{enumitem}
\usepackage{extpfeil}
\usepackage{thmtools}
\usepackage{amsthm}
\usepackage{thm-restate}
\usepackage{hyperref}
\usepackage{cleveref}
\usepackage[titles]{tocloft}
\usepackage{graphicx} 
\usepackage{verbatim}
\usepackage{sectsty}
\sectionfont{\fontsize{12}{15}\selectfont}

\newcommand{\indentEq}{\hspace{.25in}}
\newcommand{\Sub}{\scriptscriptstyle}

\hypersetup{
     colorlinks   = true,
}

\newtheorem{thm}{Theorem}[section]
\newtheorem{lem}[thm]{Lemma}
\newtheorem{prop}[thm]{Proposition}
\newtheorem{cor}[thm]{Corollary}
\newtheorem{defn}[thm]{Definition}
\newtheorem{conj}[thm]{Conjecture}
\newtheorem*{eg*}{Example}
\newtheorem{rem}[thm]{Remark}

\numberwithin{equation}{section}

\begin{document}

\title{Geodesic Coordinates for the Pressure Metric at the Fuchsian Locus}

 \author{Xian Dai}

\date{}
\maketitle 
\hypersetup{linkcolor=green}

    \hypersetup{linkcolor=red}

    \hypersetup{linkcolor=blue}

\begin{abstract}
We prove that the Hitchin parametrization provides geodesic coordinates at the Fuchsian locus for the pressure metric in the Hitchin component $\mathcal{H}_{3}(S)$ of surface group representations into $\mathrm{PSL}(3,\mathbb{R})$.  

The proof consists of the following elements: we compute first derivatives of the pressure metric using the thermodynamic formalism. We invoke a gauge-theoretic formula to compute first and second variations of reparametrization functions by studying flat connections from Hitchin's equations and their parallel transports. We then extend these expressions of integrals over closed geodesics to integrals over the two-dimensional surface. Symmetries of the Liouville measure then provide cancellations, which show that the first derivatives of the pressure metric tensors vanish at the Fuchsian locus.
\end{abstract}

\setcounter{tocdepth}{1}
\tableofcontents

\section{Introduction}

The Weil-Petersson metric on Teichm{\"u}ller space is a central object in classical Teichm{\"u}ller theory. Quite a bit is known about it: It is a negatively curved real analytic K{\"a}hler metric with isometry group induced from the extended mapping class group (Ahlfors \cite{Some_remarks_on_TeichSpace}, Tromba \cite{Tromba_Teich}, Masur-Wolf \cite{MasurWolf_WP}). Although it is not complete (Wolpert \cite{Wolpert_Noncompleteness}, Chu \cite{Chu_NonCompleteness}), it resembles a complete negative curved metric and shares many similar nice properties (Wolpert \cite{Wolpert_Noncompleteness},\cite{Wolpert_NielsenProblem}).

In recent years, considerable attention has focused on higher rank Teichm{\"u}ller spaces (\cite{Goldman1990_Projective}, \cite{HITCHIN_LieGrp_and_TeichmullerSpace}, \cite{Labourie_AnosovFlows}). It is natural to seek metric structures on these spaces with the hope that structure will reflect important properties of the spaces. To that end, Bridgeman, Canary, Labourie and Sambarino in \cite{PressureMetric-MainPaper} have extended the Weil-Petersson metric from Teichm{\"u}ller space to an analytic Riemannian metric by techniques from thermodynamic formalism, called the pressure metric on Hitchin components. The Hitchin component $\mathcal{H}_{n}(S)$, defined by Hitchin in \cite{HITCHIN_LieGrp_and_TeichmullerSpace} is a special component of the representation space of the fundamental group of a closed surface $S$ of genus $g \geq 2$ into $\mathrm{PSL}(n,\mathbb{R})$. In particular, the Teichm{\"u}ller space $\mathcal{T}(S)$, identified as representations into $\mathrm{PSL}(2,\mathbb{R})$, embeds in this component and is called the Fuchsian locus. To define the pressure metric, we associate a geodesic flow to each Hitchin representation and describe these reparametrized geodesic flows by some H{\"o}lder functions called reparametrization functions. Our pressure metric is defined on the tangent space of a Hitchin component by taking the variance of the first variation of reparametrization functions that record the infinitesimal change of the representations.

Bridgeman, Canary, Labourie and Sambarino have proved that the pressure metric in fact restricts to a multiple of the Weil-Petersson metric on the Fuchsian locus and is invariant under the action of the mapping class group. Despite this nice coincidence, very little is presently known about the pressure metric. Some $C^{0}$ properties of the pressure metric have recently been identifed by Labourie and Wentworth in \cite{Variation_along_FuchsianLocus}. In particular, they show, when restricted to the Fuchsian locus, the pressure metric is proportional to a Petersson-type pairing for variation given by holomorphic differentials. Building upon their work, our goal in this paper is to investigate some variational $C^{1}$ properties of the pressure metric using tools from thermodynamic formalism.

One may be curious to what extent that the pressure metric in Hitchin components resembles  Weil-Petersson geometry. Inspired by Ahlfors' work in \cite{Some_remarks_on_TeichSpace} that the Bers coordinates are geodesic for Weil-Petersson metric, we will show that for one particular case of Hitchin component, similar coordinates are geodesic for the pressure metric near the Fuchsian locus. The Hitchin component we consider is $\mathcal{H}_{3}(S)$ which coincides with the space of convex real projective structures \cite{CG_convex}. It is a prototypical example of higher rank Teichm{\"u}ller spaces.
We expect similar results will hold for general cases of Hitchin components $\mathcal{H}_{n}(S)$.

Inspired by the methods in Labourie and Wentworth's work \cite{Variation_along_FuchsianLocus} for the $C^{0}$ properties of the pressure metric, we will find and evaluate expressions for the derivatives of the pressure metric at the Fuchsian locus for the case of $\mathrm{PSL}(3,\mathbb{R})$ and its Hitchin component $\mathcal{H}_{3}(S)$.

The coordinates we choose are very natural in the setting of Hitchin components from a Higgs bundle perspective. Picking $(q_{1}, \cdots, q_{6g-6})$ to be a basis for $H^{0}(X,K^{2})$ over $\mathbb{R}$ and $(q_{6g-5}, \cdots, q_{16g-16})$ to be a basis for $H^{0}(X,K^{3})$ over $\mathbb{R}$, every element of $\mathcal{H}_{3}(S)$ corresponds to some
\begin{equation*}
    m(\xi)=\xi_{1}q_{1}+ \cdots \xi_{l}q_{l}
\end{equation*}
with $\xi=(\xi_{1}, \cdots, \xi_{l})\in \mathbb{R}^{l}$ and $l=16g-16$. 

The $\xi_i$ are coordinate functions and the coordinate system is realized by the Hitchin parametrization $\mathcal{H}_{3}(S) \cong H^{0}(X,K^2)$ $\bigoplus H^{0}(X,K^3)$. The Hitchin parametrization is given by the Hitchin section of the Hitchin fibration which are defined by Hitchin in \cite{HITCHIN_LieGrp_and_TeichmullerSpace} and will be explained in the next section. 

We will show
\begin{restatable}{thm}{main}
Let $S$ be a closed oriented surface with genus $g\geq 2$. For any point $\sigma\in \mathcal{T}(S) \subset \mathcal{H}_{3}(S)$, let $X$ be the Riemann surface corresponding to $\sigma$. Then the Hitchin parametrization $H^{0}(X,K^2)\bigoplus H^{0}(X,K^3)$ provides geodesic coordinates for the pressure metric at $\sigma$.
\label{thm main}
\end{restatable}

More explicitly, if we denote components of the pressure metric at $\sigma$ as $g_{ij}(\sigma)$ with respect to the coordinates given by Hitchin parametrization, then $\partial_{k}g_{ij}(\sigma)=0$ for all possible $i,j,k$ ranging from $1$ to $16g-16$.   

The proof will be a combination of techniques from the theory of thermodynamic formalism and the theory of Higgs bundles. On the one hand, we will use thermodynamic formalism to study the pressure metric and investigate its $C^{1}$ properties. On the other hand, reparametrization functions and their variations need to be understood via their Higgs bundle invariants. We now outline some important ingredients of our computations and proofs.

Since there are two types of tangential directions in $\mathcal{H}_{3}(S)$, directions given by quadratic differentials and directions given by cubic differentials (corresponding to directions along the Fuchsian locus and transverse to it respectively), the derivatives of the metric tensor will be divided into different cases according to this distinction.
\begin{itemize}
    \item 
     The vanishing of a few types of first derivatives of the metric tensor follows easily from the geometric facts that the Fuchsian locus is a totally geodesic embedding into the Hitchin component and that the Bers coordinates on Teichm{\"u}ller space are geodesic.
     
    \item 
    On the other hand, to compute the bulk of the components, we need to invoke thermodynamic formalism to obtain an explicit formula for first derivatives of the pressure metric. We find the formula of the first variation of the pressure metric by computing third derivatives of pressure functions using the theory of the Ruelle operator. This expression involves first and second variations of the reparametrization functions.
    
    \item
   We start from studying first and second variations of reparametrization functions on closed geodesics. Because vectors tangent to periodic geodesics are dense in tangent bundles of hyperbolic surfaces, the computation of first and second variations of the reparametrization functions on closed geodesics can be extended to the unit tangent bundle after an argument that the natural extensions are H{\"o}lder functions.
   
    \item 
       To study the first variation of reparametrization functions on closed geodesics, we recall a gauge theoretic formula from \cite{Variation_along_FuchsianLocus}. 
       We then interpret the resulting formula as defining a system of homogeneous ordinary differential equations which we proceed to solve.
   
   \item 
    Finding the second variation of the reparametrization functions is equivalent to understanding the first variation of our gauge theoretic formula from the previous paragraph. The difficulty here is describing how projections onto the eigenvectors for the holonomy map vary when we have a family of representations in the Hitchin component. Indeed, it turns out that we need to understand the variation of all of the eigenvectors of our holonomy map. We interpret this problem in terms of solving a system of non-homogeneous ordinary differential equations with suitable boundary conditions which we then proceed to solve.
   
   \item
   For some types of metric tensors that involve both the tangential directions and transverse directions to the Fuchsian locus, analyzing flat connections associated to these directions require understanding the corresponding harmonic metrics that are solutions of Hitchin's equations. The harmonic metrics become no longer diagonalizible when leaving the Fuchsian locus along these mixed directions. We break up the infinitesimal version of Hitchin's equation system and obtain nine scalar equations. We analyze them by maximum principles and Bochner techniques to compute second variations of reparametrization functions. 
   
   \item
   The evaluation of first derivatives of the pressure metric can be lifted to the Poincar\'e disk following an idea from \cite{Variation_along_FuchsianLocus}. Here is where it becomes important that we are taking first derivatives of the pressure metric rather than zero derivatives of the pressure metric. In particular, we find formulas involving iterated integrals of these holomorphic differentials. Specifying a point on the unit tangent bundle, we can identify the Poincar\'e disk as our coordinate chart and write down the analytic expansions of our holomorphic differentials on this chart. Using geodesic flow invariance and rotational invariance of the Liouville measure, we find that no nonzero coefficients of our analytic expansions remain after integration.

\end{itemize} 

There are more cases of tangential directions along Fuchsian locus in $\mathcal{H}_{n}(S)$ for $n \geq 4$ where the harmonic metrics are not known to be diagonalizable. Despite the fact that this makes the analysis difficult, the $n=3$ case suggests the following conjecture.

\begin{conj}
Let $S$ be a closed oriented surface with genus $g\geq 2$ and $n\geq 4$. For any point $\sigma\in \mathcal{T}(S) \subset \mathcal{H}_{n}(S)$, let $X$ be the Riemann surface corresponding to $\sigma$, the Hitchin parametrization $\bigoplus\limits_{i=2}^{i=n} H^{0}(X,K^i)$ provides geodesic coordinates for the pressure metric at $\sigma$.
\end{conj}

Recently, a Riemannian metric in $\mathcal{H}_{n}(S)$ associated to periods given by the first simple root length $L_{\alpha_1}(\rho(\gamma))=\log\frac{\lambda_1(\rho(\gamma))}{\lambda_2(\rho(\gamma))}$ has been defined by Bridgeman, Canary, Labourie and Sambarino in \cite{BCLS_SimpleRootFlows},  where $\lambda_1(\rho(\gamma))$ and $\lambda_2(\rho(\gamma))$ are largest and second largest modulus of eigenvalues of $\rho(\gamma)$. This Riemannian metric is called the Liouville pressure quadratic form in \cite{BCLS_SimpleRootFlows}. Our methods of computing first derivative of metric tensors can be applied to the Liouville pressure quadratic form. We expect similar geodesic coordinate results to hold in that setting as well.

\paragraph{Structure of the article:}The article is organized as follow. In section 2, we recall some fundamental results from the theory of thermodynamic formalism and reparametrizations of geodesic flows. We define the pressure metric. We also introduce Higgs bundles and Hitchin deformation for defining our coordinates in Hitchin components. Section 3 is devoted to preliminary proofs by thermodynamic formalism machinery.  We compute the formula for third derivatives of the pressure function. In section 4, we start the proof of the main theorem and divide the components of first derivatives of metric tensors into several types. We also include a gauge-theoretic formula given by Labourie and Wentworth in \cite{Variation_along_FuchsianLocus} here. Then in section 5, we derive second variations of reparametrization functions by studying infinitesimal variation of parallel transport equations. In section 6, we evaluate the first derivatives of the pressure metric and show they are zero following the steps explained above. We finally generalize the arguments to all types of metric tensors in the last section 7.

\paragraph{Acknowledgement}
The author would like to thank her advisor, Michael Wolf, for his help and kind support. The weekly meetings were an important source of encouragement and guidance. The author also wants to thank Martin Bridgeman for his warm introduction to the pressure metric. The author also wants to express her appreciation to  Siqi He, Qiongling Li, Andrea Tamburelli and Siran Li for helpful conversations with them. Finally the author acknowledges support from U.S. National Science Foundation (NSF) grant DMS-1564374 as well as Geometric structures and Representation Varieties (the GEAR network). The paper would not have been possible without these supports. Finally we would like to thank the reviewers for careful reading and comments.

\section{Background and notation}

In this section, we develop the notation and background material that we will need. We begin in section \ref{subsection repa} with a discussion of reparametrization of geodesic flows. Then in section \ref{subsection thermo} we recall the elements of thermodynamic formalism that we will need and finally in section \ref{subsection Hitchin}, we conclude with some notation from the theory of Higgs bundles which arises in our arguments.

Let $S$ be a closed oriented surface with genus $g\geq 2$. We will define all the concepts for introducing the pressure metric in the context of Hitchin components $\mathcal{H}_{n}(S)$. The reader can find more general version in \cite{PressureMetric-MainPaper}. The Hitchin components $\mathcal{H}_{n}(S)$ will be briefly introduced in section \ref{subsection Hitchin}.

Equip $S$ with a complex structure $J$ so that $X=(S,J)$ is a Remain surface and thus a point in Teichm{\"u}ller Space. Let $\sigma$ be the hyperbolic metric in the conformal class of $X$. We denote the unit tangent bundle of $X$ with respect to $\sigma$ by $UX$ and $\Phi$ the geodesic flow on $(X,\sigma)$.

\subsection{Reparametrization function} 
\label{subsection repa}
In this subsection, we introduce how we reparametrize the geodesic flow $\Phi$ by reparametrization functions. In particular, we introduce Liv\v{s}ic's theorem and geodesic flows for Hitchin representations.

Suppose $f:UX \to \mathbb{R}$ is a positive H{\"o}lder function and $a$ a closed orbit. We will reparametrize the flow $\Phi$ by the function $f$ so that for the new flow ${\Phi}^{f}$, the flow's direction remains the same everywhere but the speed of the flow changes. In particular, for a $\Phi$-periodic orbit $a$, denoting its period with respect to $\Phi$ by $l(a)$, we want the period of $a$ for the new flow ${\Phi}^{f}$ to be the following:

\begin{equation*}
    l_{f}(a)= \int_{0}^{l(a)}f(\Phi_{s}(x))\mathrm{d}s,
\end{equation*}
where $x$ is any point on $a$.

This leads to the following definition of reparametrization.
\begin{defn}
Let $f: UX \to \mathbb{R}$ be a positive H{\"o}lder continuous function. We define the reparametrization of $\Phi$ by $f$ to be the flow $\Phi^{f}$ on $UX$ such that for any $(x,t)\in UX \times \mathbb{R}$,

\begin{equation*}
    \Phi_{t}^{f}(x)={\Phi_{\Sub \alpha_{\Sub f}(x,t)}}(x), 
\end{equation*}

where $\kappa_{f}(x,t)=\int_{0}^{t}f(\Phi_{s}(x)))\mathrm{d}s$ and $\alpha_{f}: UX \times \mathbb{R} \to \mathbb{R}$ satisfies
\begin{equation*}
    \alpha_{f}(x,\kappa_{f}(x,t))=t.
\end{equation*}

\end{defn}

\begin{rem}
Suppose $O$ is the set of periodic orbits of $\Phi$. If $a \in O$, then its period as a $\Phi_{t}^f$ periodic orbit is $l_{f}(a)$ because 
\begin{equation*}
\Phi_{l_{\Sub f}(a)}^{f}(x)=\Phi_{\alpha_{\Sub f}(x,l_f(a))}(x)=\Phi_{\Sub l(a)}(x)=x.
\end{equation*}
\end{rem}

We introduce Liv\v{s}ic cohomology classes, originally established by Liv\v{s}ic  \cite{Livshits-cohomologous}.  Liv\v{s}ic cohomologous H{\"o}lder functions turn out to reparametrize a flow in ``equivalent" ways.

Let $C^{h}(UX)$ denote the set of real-valued H{\"o}lder functions on $UX$. 
 \begin{defn} 
 \label{defn Livsic}
    For $f,g \in C^{h}(UX)$, we say they are Liv\v{s}ic cohomologous if there exists a H{\"o}lder continuous function $V:UX\to \mathbb{R}$ that is differentiable in the flow's direction such that 
  \begin{equation*}
      f(x)-g(x)= \frac{\partial}{\partial t}\bigg|_{t=0}V(\Phi_{t}(x)).
  \end{equation*}
 If $f$ is Liv\v{s}ic cohomologous to $g$, then we will denote it as $f\sim g$. 
 \end{defn}

We have the following important properties of Liv\v{s}ic cohomologous functions:
 \begin{enumerate}
     \item 
      (Liv\v{s}ic's Theorem, \cite{Liv_ic_Cohomology}) Two H\"older continuous function $f$ and $g$ are Liv\v{s}ic cohomologous if and only if $l_{f}(a)=l_g(a)$ for every $a\in O$.
     
     \item
     If $f$ and $g$ are Liv\v{s}ic cohomologous then they have the same integral over any $\Phi$-invariant measure.
     
     This is because $\int_{UX} V(\Phi_{t}(x)) \mathrm{d}m= Const$ for any $\Phi$-invariant measure $m$ and any $t\in \mathbb{R}$.
     
     \item If $f$ and $g$ are positive and Liv\v{s}ic  cohomologous, then the reparametrized flows $\Phi^f$ and $\Phi^g$ are H{\"o}lder conjugate, i.e. there exists a H{\"o}lder homeomorphism $h:UX\to UX$ such that, for all $x\in UX$ and $t\in \mathbb{R}$,
     \begin{equation*}
         h(\Phi_{t}^f(x))=\Phi_{t}^{g}(h(x)).
     \end{equation*}
     See (\cite[Proposition.19.2.8]{Katok_dynamics} ).
 \end{enumerate}
 
 The procedure of reparametrizing geodesic flows can be applied to Hitchin components $\mathcal{H}_{n}(S)$ and provides reparametrization functions as codings for representations. This idea was first introduced by Sambarino to study counting problems associated to Anosov representations \cite{Sambarino_Lemma}. It has also been elaborated later in \cite{Sambarino_OrbitCounting}, \cite{Sambarino_Entropy} and Sambarino's other papers as well. In the setting we are working on, similar ideas lead to a construction of a geodesic flow $\Phi^{\rho}$ associated to each (conjugacy class of) Hitchin representation $\rho \in \mathcal{H}_{n}(S)$. We refer the reader to \cite{PressureMetric-MainPaper} for the explicit construction. In particular, this flow relates $\mathcal{H}_{n}(S)$ to thermodynamic formalism. We will describe here some of the important properties of $\Phi^{\rho}$:
\begin{itemize}
    \item $\Phi^{\rho}$ is an Anosov flow.
     \item There exists a H{\"o}lder function $f_{\rho}: UX \longrightarrow \mathbb{R^{+}}$, called the reparametrization function of $\rho$, such that the reparametrized flow $\Phi^{f_{\rho}}$ of $\Phi$ is H{\"o}lder conjugate to $\Phi^{\rho}$ (\cite{Sambarino_Lemma}). 
    \item The period of the orbit associated to $[\gamma]\in \pi_{1}(S)$ is $\log \Lambda_{\gamma}(\rho)$, where $\Lambda_{\gamma}(\rho)$ is the spectral radius of $\rho(\gamma)$, i,e, the largest modulus of the eigenvalues of $\rho(\gamma)$.
\end{itemize}
 
\begin{rem}
One can also reperametrizes the geodesic flow by a H{\"o}lder function with periods given by simple root lengths $L_{\alpha_1}(\rho(\gamma))=\log\frac{\lambda_1(\rho(\gamma))}{\lambda_2(\rho(\gamma))}$ where $\lambda_1(\rho(\gamma))$ and $\lambda_2(\rho(\gamma))$ are largest and second largest modulus of eigenvalues of $\rho(\gamma)$. This will lead to the Liouville pressure quadratic form which also gives rise to a Riemannian metric in $\mathcal{H}_{n}(S)$(see \cite[Theorem 1.6]{BCLS_SimpleRootFlows}). However we will mainly focus on spectrum radius length $\Lambda_{\gamma}(\rho)$ and its associated pressure metric in this paper.  
\end{rem}
 
\subsection{Thermodynamic formalism}
 \label{subsection thermo}
Next we will introduce some concepts arising from the thermodynamic formalism needed for our proofs. The introduction of most of the material here can also be found in \cite{PressureMetric-MainPaper}. After the introduction, we will define the pressure metric on Hitchin components.

As usual, we let $\Phi$ denote the geodesic flow on a hyperbolic surface $(X, \sigma)$. We denote by $\mathcal{M}^{\Phi}$ the set of $\Phi$-invariant probability measures on $UX$. Recall $l(a)$ denotes the period of the periodic point $a$ with respect to $\Phi$. Let
\begin{equation*}
    R_T=\{a  \text{ closed orbit of $\Phi$ } | l(a)\leq T\}.
\end{equation*}

 \begin{defn}
The topological entropy of $\Phi$ is defined as:
\begin{equation*}
    h(\Phi)=\limsup\limits_{T\longrightarrow \infty} \frac{\log \# R_T}{T}.
\end{equation*}
\end{defn}

Recall for a  H{\"o}lder function $f:UX \to \mathbb{R}$, we denote
\begin{equation*}
    l_{f}(a)= \int_{0}^{l(a)}f(\Phi_{s}(x))\mathrm{d}s.
\end{equation*}
\begin{defn}
The topological pressure (or simply pressure) of a continuous function $f: UX \to \mathbb{R}$ with respect to $\Phi$ is defined by
\begin{equation*}
 \mathbf{P}(\Phi,f)=\limsup\limits_{T\longrightarrow \infty} \frac{1}{T}\log(\sum\limits_{a\in R_T} e^{l_f(a)}).
\end{equation*}
\end{defn}

\begin{rem}
From this definition, we see the pressure of a function f only depends on the periods of $f$, i.e. the collection of numbers $\{ l_{f}(a)\}$ for any $a \in O$. From Liv\v{s}ic's Theorem, we conclude the pressure of a function only depends on Liv\v{s}ic cohomologous class.
\end{rem}

In statistical mechanics, suppose we are given a physical system with different possible states $i=1, \cdots,n$ and the energies of these states are $E_1, E_2 \cdots, E_n$ with probability $p_j$ that state $j$ occurs. When energy is fixed, the principle ``nature maximizes entropy $h$" says that the entropy $h(p_1, \cdots, p_n)=\sum_{i=1}^{n}-p_i \log p_i$ of the distribution will be maximized with right choices of $p_i$. However when the physical system is put in contact with a much larger ``heat source" which is at a fixed temperature $T$ and energy is allowed to pass between the original system and the heat source, ``nature minimizes the free energy" will instead apply by reaching the ``Gibbs distribution". The free energy is $E-kTh$, where $k$ is a physical constant, $h$ is the entropy and $E=\sum_{i=1}^{n} p_i E_i$ is the average of energy. In the thermodynamic formalism, energy potentials $E_i$ of different states are encoded by continuous functions and ``Gibbs distributions" for discrete probability spaces are generalized to equilibrium states. The principle ``nature minimizes free energy" motivates the following.

\begin{prop}
\label{prop defn_pressure}
(Variational principle.)

Denoting the measure-theoretic entropy of $\Phi$ with respect to a measure $m\in\mathcal{M}^{\Phi}$ as $h(\Phi,m)$, the (topological) pressure of a continuous function $f:UX\to \mathbb{R}$ satisfies
\begin{equation*}
\mathbf{P}(\Phi,f)= \sup\limits_{m\in \mathcal{M}^{\Phi}}(h(\Phi, m)+ \int_{\Sub UX} f \mathrm{d}m ).   
\end{equation*}
In particular, the topological entropy is the supremum of all measure-theoretic entropies,
\begin{equation*}
\mathbf{P}(\Phi,0)=\sup\limits_{m\in \mathcal{M}^{\Phi}}(h(\Phi, m))=h(\Phi).
\end{equation*}
\end{prop}

\begin{rem}
One can also take Proposition \ref{prop defn_pressure} as definitions of pressure and topological entropies. 
\end{rem}

We shall omit the background geodesic flow $\Phi$ in the notation of pressure and simply write 
\begin{align*}
    \mathbf{P}(\cdot)=\mathbf{P}(\Phi, \cdot)
\end{align*}
 
 \begin{defn}
  A measure $m\in \mathcal{M}^{\Phi}$ on $UX$ such that 
  \begin{equation*}
\mathbf{P}(f)= h(\Phi, m)+ \int_{\Sub UX} f \mathrm{d}m   
\end{equation*}
is called an equilibrium state of $f$. 
 \end{defn}

\begin{prop}
(Bowen-Ruelle \cite{Bowen-Ruelle})
For any H{\"older} function $f:UX \to \mathbb{R}$, with respect to the geodesic flow $\Phi$, there exists a unique equilibrium state for $f$, denoted as $m_{f}$. Moreover, $m_{f}$ is ergodic.
\end{prop}

\begin{rem}
We observe from the definition of equilibrium states that if $f-g$ is Liv\v{s}ic cohomologous to a constant, then $f$ and $g$ have the same equilibrium states.
\label{rem SameEquilState}
\end{rem}

\begin{defn}
The equilibrium state $m_0$ for $f=0$ is called a probability measure of maximal entropy. It is also called the Bowen-Margulis measure of $\Phi$. We also denote it as $m_{\Phi}$. It satisfies
\begin{align*}
     \mathbf{P}(0)= \mathbf{P}(\Phi, 0)=h(\Phi, m_{\Phi})=h(\Phi).
\end{align*}
\end{defn}

\begin{rem}
\label{rem Liouville}
    
The Liouville measure $m_{L}$, the normalized Riemannian measure on $UX$, is a probability measure of maximal entropy for geodesic flows of closed hyperbolic manifolds (see \cite{katok_Entropy-and-closed-geodesies} section 2). Thus when considering the geodesic flow $\Phi$ of a hyperbolic surface $(X,\sigma)$, we have $m_{L}=m_{\Phi}$.
\end{rem}

Given $f$ a positive H{\"o}lder continuous function on $UX$, denoting $h(f)=h(\Phi^{f})$ to be the topological entropy of the reparametrized flow $\Phi^{f}$, we have the following lemma that allows us to ``normalize" a H{\"o}lder function to have pressure zero.

\begin{lem}
(Sambarino \cite{Sambarino_Lemma}, Bowen-Ruelle \cite{Bowen-Ruelle})
The pressure satisfies
$$\mathbf{P}(-hf)=0$$
if and only if $h=h(f)=h(\Phi^{f})$.
\end{lem}

Potrie and Sambarino show, in the Hitchin component $\mathcal{H}_{n}(S)$, the topological entropy is maximized only along the Fuchsian locus. In particular, it is a constant on the Fuchsian locus.

\begin{thm}(Potrie-Sambarino \cite{Sambarino_Entropy})
\label{thm Entropy}
If $\rho \in \mathcal{H}_{n}(S)$, then $h(\rho)\leq \frac{2}{n-1}$. Moreover, if $h(\rho)=\frac{2}{n-1}$, then $\rho$ lies in the Fuchsian locus.
\end{thm}

We start to define variance and covariance which will be important. The convergence of them for mean zero functions is classical. 
\begin{defn}
For $g$ a H{\"o}lder continuous function on $UX$ with mean zero with respect to $m_{f}$(i.e. $\int_{\Sub UX} g \mathrm{d}m_f=0$), the variance of $g$ with respect to $f$ is defined as:
\begin{equation}
    \mathrm{Var}(g,m_f)= \lim_{T\to\infty} \frac{1}{T} \int_{\Sub UX}  \left( \int_{0}^{T} g(\Phi_{s}(x))\mathrm{d}s \right) ^2 \mathrm{d}m_{f}(x).
    \label{eq:Var}
\end{equation}
\end{defn}

\begin{defn}
For $g_{1}, g_{2}$  H{\"o}lder continuous functions on $UX$ with mean zero with respect to $m_{f}$ (i.e. $\int_{\Sub UX} g_1 \mathrm{d}m_f = \int_{\Sub UX} g_2 \mathrm{d}m_f =0$), the covariance of $g_1, g_2$ with respect to f is defined as:
\begin{equation}
    Cov(g_1,g_2,m_{f})= \lim_{T\to\infty} \frac{1}{T} \int_{\Sub UX}  \left( \int_{0}^{T} g_1(\Phi_{s}(x))\mathrm{d}s \right) \left( \int_{0}^{T} g_2(\Phi_{s}(x))\mathrm{d}s \right)   \mathrm{d}m_{f}(x).
    \label{eq:Cov}
\end{equation}
\end{defn}

Note these expressions are finite: 
\begin{prop}
For $g_{1}, g_{2}$ H{\"o}lder continuous function on $UX$ with mean zero with respect to $m_{f}$, the covariance of $g_1$ and $g_2$ is finite:   
$$Cov(g_1,g_2,m_{f}) < \infty .$$
\end{prop}
The convergence is guaranteed by decay of correlations (see \cite{McMullen2008}).

\begin{defn}
\label{defn projection}
We define an operator $P_{m}: C^{h}(UX)\to C^{h}(UX)$ associated to a probability measure $m$ on $UX$ to be:
\begin{equation*}
    P_{m}(g)(x)=g(x)-m(g),
\end{equation*}
\end{defn}

where we use the notation $m(g)=\int_{\Sub UX} g \mathrm{d}m$ for a probability measure $m$. 

The following corollary will be useful: 
\begin{cor}
\label{cor:CovNotMeanZero}

 It suffices to have $m_{f}(g_{1})=0$ and $m_{f}(g_{2})<\infty$ to guarantee the convergence of covariance and 
\begin{equation}
 Cov(g_{1},g_{2},m_{f})=Cov(g_{1}, P_{m_{f}}(g_{2}),m_{f})< \infty.   
\end{equation}

\begin{proof}[Proof of Corollary \ref{cor:CovNotMeanZero}]\hfill

Because
\begin{align*}
   & \frac{1}{T} \int_{\Sub UX}  \left( \int_{0}^{T} g_1(\Phi_{s}(x))\mathrm{d}s \right) \left( \int_{0}^{T} g_{2}(\Phi_{s}(x))-P_{m_f}(g_{2}(\Phi_{s}(x)))\mathrm{d}s \right)   \mathrm{d}m_{f}(x)\\
   =& \frac{1}{T} \int_{\Sub UX}  \left( \int_{0}^{T} g_1(\Phi_{s}(x))\mathrm{d}s \right) \left( \int_{0}^{T} m_{f}(g_2)\mathrm{d}s \right)   \mathrm{d}m_{f}(x)  \\
   =&m_{f}(g_2)\int_{\Sub UX} \int_{0}^{T}g_{1}(\Phi_{s}(x))\mathrm{d}s \mathrm{d}m_{f}(x)
   \hspace{.5in}\text{As $m_{f}(g_2)$ is a constant}\\
   =&m_{f}(g_2)\int_{0}^{T}\int_{\Sub UX}g_{1}(\Phi_{s}(x))\mathrm{d}m_{f}(x)ds  \hspace{.5in}\text{By Fubini theorem}\\
   =&m_{f}(g_2)\int_{0}^{T}\int_{\Sub UX}g_1(x)\mathrm{d}m_{f}(x)ds  \hspace{.5in}\text{$m_{f}$ is $\Phi$-invariant}\\
   =&0
\end{align*}
 Let $T\to \infty$, we obtain the desired result. 
The same applies to the case $m_{f}(h_{2})=0, m_{f}(h_{1})<\infty$.
\end{proof}
\end{cor}

We will also need the following characterization of covariance for later use.

\begin{defn}(Pollicott \cite{pollicott_DerivativeOfEntropy})
\label{defn Cov}
 For $g_{1}, g_{2}$  H{\"o}lder continuous function with mean zero with respect to $m_{f}$ (i.e. $\int_{\Sub UX} g_1 \mathrm{d}m_f = \int_{\Sub UX} g_2 \mathrm{d}m_f =0$), the covariance of $g_1, g_2$ may also be written as:
 \begin{equation*}
    Cov(g_1,g_2,m_{f})= \lim_{T\to\infty}\int_{\Sub UX} g_2(x) \left( \int_{-\frac{T}{2}}^{\frac{T}{2}} g_1(\Phi_{s}(x))\mathrm{d}s \right)   \mathrm{d}m_{f}(x).
\end{equation*}
\end{defn}

\begin{proof}
This proof is from \cite{pollicott_DerivativeOfEntropy}.
\begin{align*}
    Cov(g_1,g_2,m_{f})&=\lim_{T\to\infty} \frac{1}{T} \int_{\Sub UX}  \left( \int_{0}^{T} g_1(\Phi_{s}(x))\mathrm{d}s \right) \left( \int_{0}^{T} g_2(\Phi_{s}(x))\mathrm{d}s \right)   \mathrm{d}m_{f}(x) \\
    &=\lim_{T\to\infty} \frac{1}{T} \int_{\Sub UX}  \left( \int_{-\frac{T}{2}}^{\frac{T}{2}} g_1(\Phi_{s}(x))\mathrm{d}s \right) \left( \int_{-\frac{T}{2}}^{\frac{T}{2}} g_2(\Phi_{s}(x))\mathrm{d}s \right)   \mathrm{d}m_{f}(x)\hspace{.5in}\text{$m_{f}$ is $\Phi$-invariant}\\
    &=\lim_{T\to\infty}  \int_{-\frac{T}{2}}^{\frac{T}{2}} \int_{\Sub UX}  g_1(\Phi_{t}(x)) \frac{1}{T} \left( \int_{-\frac{T}{2}}^{\frac{T}{2}} g_2(\Phi_{s}(x))\mathrm{d}s \right)   \mathrm{d}m_{f}(x) \mathrm{d}t
\end{align*}
Because $m\in \mathcal{M}^{\Phi}$, the following does not vary with $s$.
\begin{align*}
   Const.&=\lim_{T\to\infty}\int_{-\frac{T}{2}}^{\frac{T}{2}} \int_{\Sub UX}  g_1(\Phi_{t}(x))  g_2(\Phi_{s}(x))  \mathrm{d}m_{f}(x) \mathrm{d}t \hspace{.5in}\text{$\forall s\in \mathbb{R}$} \\
   &=\lim_{T\to\infty}  \int_{-\frac{T}{2}}^{\frac{T}{2}} \int_{\Sub UX}  g_1(\Phi_{t}(x)) \frac{1}{S} \left( \int_{-\frac{S}{2}}^{\frac{S}{2}} g_2(\Phi_{s}(x))\mathrm{d}s \right)   \mathrm{d}m_{f}(x) \mathrm{d}t \hspace{.5in}\text{Average over $s \in \left[-\frac{S}{2}, \frac{S}{2}\right]$}\\
    &=\lim_{S\to\infty} \lim_{T\to\infty}\int_{-\frac{T}{2}}^{\frac{T}{2}} \int_{\Sub UX}  g_1(\Phi_{t}(x)) \frac{1}{S} \left( \int_{-\frac{S}{2}}^{\frac{S}{2}} g_2(\Phi_{s}(x))\mathrm{d}s \right)   \mathrm{d}m_{f}(x) \mathrm{d}t\\
     &=\lim_{T\to\infty}  \int_{-\frac{T}{2}}^{\frac{T}{2}} \int_{\Sub UX}  g_1(\Phi_{t}(x)) \frac{1}{T} \left( \int_{-\frac{T}{2}}^{\frac{T}{2}} g_2(\Phi_{s}(x))\mathrm{d}s \right)   \mathrm{d}m_{f}(x) \mathrm{d}t\\
   &=Cov(g_1,g_2,m_{f})
\end{align*}
In particular, setting $s=0$ gives
\begin{align*}
  Cov(g_1,g_2,m_{f})=\lim_{T\to\infty}\int_{-\frac{T}{2}}^{\frac{T}{2}} \int_{\Sub UX}  g_1(\Phi_{t}(x))  g_2((x))  \mathrm{d}m_{f}(x) \mathrm{d}t
\end{align*}
Rearranging the integrals gives the desired result.
\end{proof}

Higher correlation and higher covariance are introduced for Anosov diffemorphism in \cite{kotani_sunada_HigherCorrelations}. For geodesic flows, we define 

\begin{defn}
\label{defn higherCov}
 For $g_{1}, g_{2}, g_{3}$  H{\"o}lder continuous functions with mean zero with respect to $m_f$, we define the higher covariance as follows,
 \begin{equation*}
    Cov(g_{1}, g_{2}, g_{3}, m_f)= \lim_{T\to \infty}\frac{1}{T}
  \int_{\Sub UX} \int_{0}^{T}g_{1}(\Phi_{t}(x))\mathrm{d}t \int_{0}^{T}g_{2}(\Phi_{t}(x))\mathrm{d}t \int_{0}^{T}g_{3}(\Phi_{t}(x))\mathrm{d}t\mathrm{d}m_f(x).
\end{equation*}
equivalently,
 \begin{equation*}
    Cov(g_1,g_2,g_{3}, m_f)= \lim_{T\to\infty}\int_{\Sub UX} g_1(x) \left( \int_{-\frac{T}{2}}^{\frac{T}{2}} g_2(\Phi_{s}(x))\mathrm{d}s \right) \left( \int_{-\frac{T}{2}}^{\frac{T}{2}} g_3(\Phi_{s}(x))\mathrm{d}s \right)  \mathrm{d}m_f(x).
\end{equation*}
\end{defn}

This equivalence is clear from the proof of equivalent Definition \ref{defn Cov}.  The convergence of $Cov(h_{1}, h_{2}, h_{3}, m)$ is guaranteed by ``exponential multiple mixing" for geodesic flow on negatively curved compact surfaces (see Pollicott's note \cite{Pollicott_MultipleMixing}). These definitions will be used later when we introduce first derivatives of the pressure metric.

We use the general notation in the sequel:
 \begin{align}
      \partial_{s}{f}(0)=\frac{d}{ds}\bigg|_{s=0}f(s) , \text{  } \partial^{2}_{s}{f}(0)=\frac{d^2}{ds^2}\bigg|_{s=0}f(s)
      \label{eq notation1}
 \end{align}
If there are more than one parameter, e.g. $f(s_1,s_2,\cdots, s_{k})$ and $k\geq 2$, then we specify the indexes that we are taking derivatives of:
\begin{align}
\partial_{s_{i_1}\cdots,s_{i_j}}f(0)=\frac{\partial^{j} f(s_1, s_2, \cdots, s_k)}{\partial s_{i_1}\cdots \partial{s_{i_j}}}\bigg|_{s_1=s_2=\cdots=0}
\label{eq notation2}
\end{align}

\begin{thm}(Parry-Pollicott \cite{Pollicot_Zeta}, McMullen \cite{McMullen2008})
Let $f_s$ be a smooth family of functions in $C^{h}(UX)$, then we have

\begin{enumerate}
\item The first derivative of $\mathbf{P}(f_s)$ at $s=0$ is given by
\begin{equation}
    \frac{d \mathbf{P}(f_s)}{ds}\bigg|_{s=0}= \int_{\Sub UX} \partial_{s}{f}_{0} \mathrm{d}m_{f_{0}},
    \label{eq: First dev of P}
\end{equation}
\item
If the first derivative is zero, then 

\begin{equation}
    \frac{d^2 \mathbf{P}(f_s)}{ds^2}\bigg|_{s=0}= Var(\partial_{s}{f}_{0}, m_{f_{0}})+\int_{\Sub UX} \partial^{2}_{s}{f}_{0} \mathrm{d}m_{f_{0}},
    \label{eq:Second s-dev of P}
\end{equation}

\item
If the first derivative is zero, then 

$Var(\partial_{s}{f_{0}}, m_{f_{0}})=0$ if and only if $\partial_{s}f_{0}$ is Liv\v{s}ic cohomologous to zero.

\end{enumerate}

\end{thm}

 \begin{rem}
 If $f(s,t)$ is a smooth two parameter family in $C^h(UX)$, then 
 \begin{equation}
     \frac{\partial \mathbf{P}(f(s,t))}{\partial t \partial s}\bigg|_{s=t=0}=Cov(P_{\Sub m_{f(0)}}(\partial_{s}f(0)),P_{\Sub m_{\Sub f(0)}}(\partial_{t}f(0)), m_{f(0)})+\int_{\Sub UX} \partial_{st}f(0)\mathrm{d}m_{f(0)}.
     \label{eq: Second dev of P}
 \end{equation}
 \end{rem}

 Define $\mathcal{P}(UX)$ to be the set of pressure zero H{\"o}lder functions on $UX$, i.e.
\begin{equation*}
\mathcal{P}(UX)=\{f\in C^h(UX): \mathbf{P}(f)=0 \}.
\end{equation*}
The tangent space of $\mathcal{P}(UX)$ at $f$ is the set
\begin{align*}
    T_{f} \mathcal{P}(UX)=\text{ker } d_{f}\mathbf{P}= \{h\in C^{h}(UX) | \int_{UX} h  \mathrm{d}m_f =0\}
\end{align*}

We define a pressure semi-norm on the tangent space of $\mathcal{P}(UX)$ at $f$, by letting

\begin{defn}
The pressure semi-norm of $g\in T_{f} \mathcal{P}(UX)$ is defined as:
\begin{equation*}
    {\langle g,g \rangle}_{\Sub P}=-\frac{Var(g,m_f)}{\int_{UX}f dm_f}.
\end{equation*}
\end{defn}

One notices for $g\in T_{f} \mathcal{P}(UX)$, the variance $Var(g,m_f)=0$ if and only if $g$ is Liv\v{s}ic cohomologous to $0$, i.e. $g\sim 0$. 
  
\subsection{Higgs bundles and Hitchin deformation}  
\label{subsection Hitchin}

In this subsection we introduce all the notation from the theory of Higgs bundles that will arise in our arguments. We also introduce a coordinate system on the Hitchin component at the end of the section.

Recall $S$ is a closed oriented surface with genus $g\geq 2$ and $X=(S,J)$ is a Riemann surface. 
\begin{defn}
A rank $n$ Higgs bundle over $X$ is a pair $(E,\Phi)$ where $E$ is a holomorphic vector bundle of rank $n$ and $\Phi \in H^{0}(X,End(E)\otimes K)$ is called a Higgs field. A $\mathrm{SL}(n,\mathbb{C})$-Higgs bundle is a Higgs bundle $(E,\Phi)$ satisfying $\det E = \mathcal{O}$ and $\Tr\Phi=0$.
\end{defn}

\begin{defn}{\ }
\begin{enumerate}
    \item A Higgs bundle $(E,\Phi)$ is (semi)stable if every proper $\Phi$-invariant holomorphic subbundle $F$ of $E$ satisfies
    \begin{align*}
        \frac{\deg(F)}{\rank(F)}(\leq) < \frac{\deg(E)}{\rank(E)}. 
    \end{align*}
    \item A semi-stable Higgs bundle $(E,\Phi)$ is polystable if it decomposes as a direct sum of stable Higgs bundles. 
\end{enumerate}
\end{defn}

\begin{thm}
\label{thmChernConnection}
It is classical that for a holomorphic vector bundle $E$ with holomorphic structure $\bar{\partial}_{E}$ and a Hermitian metric $H$, there exists a unique connection $\nabla_{{\bar{\partial}}_{E},H}$, called the Chern connection, such that 
\begin{enumerate}
    \item  $\nabla_{{\bar{\partial}}_{E},H}^{0,1}=\bar{\partial}_{E}$.
    \item $\nabla_{{\bar{\partial}}_{E},H}$ is unitary.
\end{enumerate}
\end{thm}

We will from now on restrict our interest to degree zero Higgs bundles.

\begin{thm}
(Hitchin\cite{Hitchin87theself-duality}, Simpson\cite{Simpson-Variation-of-Hodge})
Let $(E,\Phi)$ be a rank $n$, degree zero Higgs bundle on $X$. Then $E$ admits a Hermitian metric $H$ satisfying Hitchin's equation if and only if $(E,\Phi)$ is polystable. Hitchin's equation is 
\begin{align}
    F_{\bar{\partial},H}+[\Phi,{\Phi}^{*H}]=0,
    \label{eq Hitchin'sEquation}
\end{align}
where $F_{\bar{\partial},H}$ is the curvature of the Chern connection $\nabla_{{\bar{\partial}}_{E},H}$, and ${\Phi}^{*H}$ is the Hermitian adjoint of $\Phi$.
\end{thm}

\begin{rem}
Define a connection $D_{H}$ on $(E, \Phi, H)$ as 
\begin{align}
    D_{H}=\nabla_{{\bar{\partial}}_{E},H}^{0,1}+ \Phi+ {\Phi}^{*H}.
    \label{eq flatD}
\end{align}
$D_{H}$ is flat if and only if the Hitchin's equation is satisfied.
\end{rem}

We define Higgs bundles moduli space and de Rham moduli space as 
\begin{defn}
\begin{itemize}
    \item The space of gauge equivalence classes of polystable $\mathrm{SL}(n, \mathbb{C})$ Higgs bundles is called the moduli space of $\mathrm{SL}(n, \mathbb{C})$-Higgs bundles and is denoted by $\mathcal{M}_{Higgs}(\mathrm{SL}(n,\mathbb{C}))$. 
    \item The space of gauge equivalence classes of reductive flat $\mathrm{SL}(n, \mathbb{C})$ connections is called the de Rham moduli space and is denoted by $\mathcal{M}_{deRham}(\mathrm{SL}(n,\mathbb{C}))$. 
\end{itemize}
\end{defn}

\begin{rem}
The Hitchin-Simpson Theorem gives a 1-1 correspondence between $\mathcal{M}_{Higgs}(\mathrm{SL}(n,\mathbb{C}))$ and $\mathcal{M}_{deRham}(\mathrm{SL}(n,\mathbb{C}))$ from the above remark. It is also called Hitchin-Kobayashi correspondence.
\end{rem}

We will introduce Hitchin fibration and Hitchin section following Baraglia's work \cite{Baraglia_thesis}. We refer the reader to section 2 of \cite{Baraglia_thesis} for a more comprehensive exposition.

Given a principal 3-dimensional subalgebra $\mathfrak{s}=\text{span}\{x, e, \tilde{e}\}$ of $\mathfrak{sl}(n,\mathbb{C})$ consisting of a semisimple element $x$, regular nilpotent elements $e$ and $\tilde{e}$ with commutation relations:
$$[x,e]=e, \text{  } [x,\tilde{e}]=-\tilde{e}, \text{  } [e,\tilde{e}]=x,$$
the Lie algebra 
$\mathfrak{sl}(n,\mathbb{C})$ decomposes into a direct sum of irreducible subspaces under the adjoint representation of $\mathfrak{s}$:
$$\mathfrak{sl}(n,\mathbb{C})= \bigoplus\limits_{i=1}^{n-1} V_i.$$
We take $e_1,\cdots, e_{n-1}$ as highest weight elements of $V_1, \cdots, V_{n-1}$ where $e_1=e$. With these defined, there exists basis of $\mathrm{SL}(n,\mathbb{C})$-invariant homogeneous polynomials $p_{i}$ of degree $i$ on $\mathfrak{sl}(n,\mathbb{C})$, where $2\leq i\leq n$, such that for all elements $f\in\mathfrak{sl}(n,\mathbb{C})$ of the form
$$f= \tilde{e}+\alpha_2 e_1+ \cdots + \alpha_{n} e_{n-1},$$
wee have $p_i(f)=\alpha_{i}$.
\begin{defn}
The Hitchin fibration is a map from the moduli space of $\mathrm{SL}(n,\mathbb{C})$-Higgs bundles over $X$ to the direct sum of holomorphic differentials given by
\begin{align*}
    p:\mathcal{M}_{Higgs}(\mathrm{SL}(n,\mathbb{C})) &\longrightarrow \bigoplus\limits_{i=2}^{i=n} H^{0}(X,K^i),\\
    (E,\Phi) &\longmapsto  (p_{2}(\Phi),\cdots, p_{n}(\Phi)).
\end{align*}
Where $p_i$ are homogeneous invariant polynomials defined above.
\end{defn}

\begin{defn}
A Hitchin section $s$ of the Hitchin fibration is a map from $\bigoplus\limits_{i=2}^{i=n} H^{0}(X,K^i)$ back to $\mathcal{M}_{Higgs}(\mathrm{SL}(n,\mathbb{C}))$. For $q=(q_{2},q_{3},\cdots, q_{n})\in \bigoplus\limits_{i=2}^{i=n} H^{0}(X,K^i)$, we define $s(q)$ to be a Higgs bundle $E=K^{\frac{n-1}{2}}\bigoplus K^{\frac{n-3}{2}}\cdots \bigoplus K^{\frac{1-n}{2}}$
with its Higgs field given as follows,
$$\Phi(q)=\tilde{e}+q_{2}e_{1}+q_{3}e_{2}+ \cdots q_{n}e_{n-1}.$$
More explicitly, we have
 \[
  \Phi(q)=
  \begin{bmatrix}
   0 & & & r_1q_{2} &  & r_1 r_2 q_{3} & r_1 r_2 r_3 q_4 &\cdots & \prod\limits_{i=1}^{n-2} r_{i}q_{n-1} & \prod\limits_{i=1}^{n-1} r_{i}q_{n}\\
   1 & & & 0 &  & r_2q_{2} &  r_2 r_3 q_3 & \cdots & \cdots& \prod\limits_{i=2}^{n-1} r_{i}q_{n-1} \\
    0 & & & 1 &  &0 &  r_3 q_2 & r_3 r_4 q_3 & \cdots& \vdots \\[1.0em]
     \vdots & & & \vdots &  & \ddots &  \ddots & \ddots & \ddots& \vdots \\[0.8em]
       \vdots& & & \vdots &  &  \ddots & \ddots & \ddots& \vdots& \vdots \\[1.0em]
       0 & & & 0 &  &\cdots & 0  & 1 & 0 & r_{n-1}q_2 \\[1.5em]
       0  & & & 0 &  &\cdots & \cdots & 0 & 1 & 0 \\
  \end{bmatrix}
  : E \to E \otimes K
\] 
where $r_{i}=\frac{i(n-i)}{2}$ and $K^{\frac{1}{2}}$ is a holomorphic line bundle with its square to be the canonical line bundle $K$. The notation for $e_{i}$ we use here can be found in \cite{Baraglia_thesis} and \cite{Qiongling_Harmonic}.
\end{defn}

\begin{rem}
There exists an involutive automorphism $\sigma$ on $\mathfrak{sl}(n,\mathbb{C})$ such that
$$\sigma(e_i)=-e_i, \text{   } \sigma(\tilde{e})=-\tilde{e}$$
Composing with the compact real form $\rho$ on $\mathfrak{sl}(n,\mathbb{C})$ given by $\rho(X)=-X^*$, we can obtain the split real involution given by $\lambda= \rho\circ \sigma$. The fixed points set of $\lambda$ is the real split form $\mathfrak{sl}(n,\mathbb{R})$. A detailed exposition for this can be found in \cite{Baraglia_thesis}.

From the fact that $\lambda(\Phi(q))=\Phi(q)^{*}$, one can see the flat connection (\ref{eq flatD}) has holonomy in the split real form of $\mathfrak{sl}(n,\mathbb{C})$.  Hitchin therefore shows that the Higgs bundles in the image of the Hitchin section have holonomy in $\mathrm{SL}(n,\mathbb{R})$ (see \cite{HITCHIN_LieGrp_and_TeichmullerSpace}). The representation space of these Higgs bundles up to conjugacy equivalence form a connected component of the representation variety $Rep(\pi_1(S),\mathrm{SL}(n,\mathbb{R}))$, called the Hitchin component $\mathcal{H}_{n}(S)$. Here we recall that the representation variety $Rep(\pi_1(S),\mathrm{SL}(n,\mathbb{R}))$ is the space of conjugacy class of reductive representations from $\pi_1(S)$ to $\mathrm{SL}(n,\mathbb{R})$.
\end{rem}

\begin{rem}
The isomorphism between $\mathcal{H}_{n}(S)$ and $\bigoplus\limits_{i=2}^{i=n} H^{0}(X,K^i)$ yields a parametrization of the Hitchin component $\mathcal{H}_{n}(S)$. We call
$\bigoplus\limits_{i=2}^{i=n} H^{0}(X,K^i)$ the Hitchin base. In particular, the tangent space at Fuchsian point $X$ is identified with the Hitchin base.
\label{rem HitchinPara}
\end{rem}

Fixing $E=K^{\frac{n-1}{2}}\bigoplus K^{\frac{n-3}{2}}\cdots \bigoplus K^{\frac{1-n}{2}}$, we consider the following map  as an infinitesimal change of a family of Higgs fields $\Phi_{\epsilon}$ associated to $q$. 

\begin{align*}
\chi: \bigoplus\limits_{i=2}^{i=n} H^{0}(X,K^i) &\to \Omega^{1,0}(X, \mathfrak{sl(n,\mathbb{R})})\\
    \chi(q)&=\sum\limits_{i=2}^{n}q_{i}\otimes e_{i-1}.
\end{align*}

In particular, the infinitesimal change of a family of flat connections (\ref{eq flatD}) in $\mathcal{M}_{deRham}(\mathrm{SL}(n,\mathbb{C}))$ associated to $q$ defines an isomorphism of $\bigoplus\limits_{i=2}^{i=n} H^{0}(X,K^i)$ with the tangent space of the Hitchin component $T_{X}\mathcal{H}_{n}(S)$. Associated to $\chi(q)$, the deformation of flat connections which is the infinitesimal version of equation (\ref{eq flatD}) is

\begin{defn}
\label{defn HitchinDeformation}
At the Fuchsian point $X$, we define our Hitchin deformation associated to $q$ to be 
\begin{align*}
\varphi(q):= \chi(q)+\lambda(\chi(q)) ,   
\end{align*}
where $\lambda$ is the antilinear involution for the split real form of $\mathfrak{sl(n,\mathbb{C})}$ defined above.
\end{defn}

 This type of deformation will be the tangential objects we consider for the pressure metric.

\begin{rem}
\label{rem coordinate}
The Hitchin parametrization in Remark \ref{rem HitchinPara} gives a coordinate system for $\mathcal{H}_{n}(S)$ based at $X$. More explicitly, given a basis $\{q_{i}\}_{i=1}^{i=l}$ of $\bigoplus\limits_{i=2}^{i=n}H^{0}(X,K^i)$ with $l=2(n^2-1)(g-1)$, the coordinate system is given by
\begin{equation*}
    m(\xi)=\xi_{1}q_{1}+ \cdots \xi_{l}q_{l},
\end{equation*}
where $\xi=(\xi_{1}, \cdots, \xi_{l})\in \mathbb{R}^{l}$. Because of the isomorphism between $\mathcal{H}_{n}(S)$ and $\bigoplus\limits_{i=2}^{i=n} H^{0}(X,K^i)$, the vector $\xi=(\xi_{1}, \cdots, \xi_{l})$ provides local parameters on $\mathcal{H}_{n}(S)$ and the function $\xi_{i}: \mathcal{H}_{n}(S)\to \mathbb{R}$ is a coordinate function for $1\leq i \leq l$.
\end{rem}

\subsection{The pressure metric on Hitchin components}
We define the pressure metric for Hitchin components $\mathcal{H}_{n}(S)$ in this subsection and state some known results about it.

Recall $\mathcal{H}(UX)$ is  the space of pressure zero H{\"o}lder functions modulo Liv\v{s}ic coboundaries.
We relates $\mathcal{H}(UX)$ to the Hitchin component $\mathcal{H}_{n}(S)$ by the following thermodynamic mapping. 
\begin{defn}
The thermodynamic mapping $\Psi: \mathcal{H}_{n}(S) \longrightarrow \mathcal{H} (UX)$ from a Hitchin component $\mathcal{H}_{n}(S)$ to the space $\mathcal{H} (UX)$ of Liv\v{s}ic cohomology classes of pressure zero H{\"o}lder functions on $UX$ is defined as 
\begin{equation*}
     \Psi(\rho)=[-h(\rho) f_{\rho}],
\end{equation*}
where $h(\rho)=h(f_{\rho})=h(\Phi^{f_{\rho}})$ is the topological entropy of the reparametrized flow $\Phi^{f_{\rho}}$. 
\end{defn}
The mapping $\Psi$ admits local analytic lifts to the space $\mathcal{P}(UX)$ of pressure zero H{\"o}lder functions. In particular, the map $\tilde{\Psi}: \mathcal{H}_{n}(S) \longrightarrow \mathcal{P}(UX)$ given by $\tilde{\Psi}(\rho)=-h(\rho) f_{\rho}$ is an analytic local lift of $\Psi$. This enables us to pull back the pressure form on $\mathcal{P}(UX)$ to obtain a pressure form on $\mathcal{H}_{n}(S)$.

We will from now on denote $f_{\rho}^N=-h(\rho)f_{\rho}$ to be the normalized reparametrization function. 

Given  an analytic family $\{\rho_s\}_{s\in(-1,1)}$ of (conjugacy class of) representations in Hitchin component $\mathcal{H}_n(S)$, we denote $\dot{\rho}_{0}=\partial_{s}\rho_{0}=\partial_{s}\rho_{s}|_{s=0}$. Let $\{f_{\rho_{s}}\}_{s\in(-1,1)}$ be associated reparametrization functions, we pull back the pressure form on $\mathcal{P}(UX)$ as follow:
\begin{align*}
\langle\dot{\rho}_{0},\dot{\rho}_{0}\rangle_{\Sub P}&= 
\langle d\tilde{\Psi}(\dot{\rho}_{0}),d\tilde{\Psi}(\dot{\rho}_{0})\rangle_{\Sub P}\\
&=\langle\frac{\partial{(-h(\rho_s)f_{\rho_s}})}{\partial{s}}\bigg|_{s=0},\frac{\partial{(-h(\rho_s)f_{\rho_s})}}{\partial{s}}\bigg|_{s=0}\rangle_{\Sub P}\\
 &=\langle \partial_{s}(f_{\rho_s}^N)|_{s=0},\partial_{s}(f_{\rho_s}^N)|_{s=0}\rangle_{\Sub P} \\
    &=-\frac{Var(\partial_{s}(f_{\rho_s}^N)|_{s=0}, m_{f_{\rho_0}^N})}{\int_{\Sub UX}f_{\rho_{0}}^N dm_{f_{\rho_0}^N}}
\end{align*}
It is proved in \cite{PressureMetric-MainPaper} that the pull back pressure form is nondegenerate and thus defines a Riemanninan metric on $\mathcal{H}_{n}(S)$:

\begin{defn}
Suppose $\{\rho_s\}_{s\in(-1,1)}$ and $\{\eta_t\}_{t\in(-1,1)}$ are two analytic families of (conjugacy classes of) representations in Hitchin component $\mathcal{H}_n(S)$ such that $\rho_{0}=\eta_{0}$,  the pressure metric for $\dot{\rho_{0}},\dot{\eta_{0}}  \in T_{\rho_{0}} \mathcal{H}_{n}(S)$ is defined as: 
\begin{equation*}
    {\langle \dot{\rho_{0}},\dot{\eta_{0}} \rangle}_{\Sub P}=-\frac{Cov(\partial_{s}(f_{\rho_s}^N)|_{s=0}, \partial_{s}(f_{\eta_s}^N)|_{s=0}, m_{f_{\rho_0}^N})}{\int_{\Sub UX}f_{\rho_{0}}^N dm_{f_{\rho_0}^N}}
\end{equation*}
\end{defn}

For simiplicity, later we will also denote $\partial_{s}(f_{\rho_s}^N)|_{s=0}=\partial_{s}f_{\rho_0}^N$ and $\partial_{s}(f_{\eta_s}^N)|_{s=0}=\partial_{s}f_{\eta_0}^N$. The principle is, we always first normalize a family of reparametrization functions to be pressure zero and then take derivatives.
 
Because of the identification of $\bigoplus\limits_{i=2}^{i=n} H^{0}(X,K^i)$ with the tangent space of the Hitchin component $T_{X}\mathcal{H}_{n}(S)$, our Hitchin deformation $\varphi(q)$ introduced in Definition \ref{defn HitchinDeformation} can be thought as tangent vectors in $T_{X}\mathcal{H}_{n}(S)$. With this understood, we introduce the following important results from \cite{Variation_along_FuchsianLocus} by Labourie and Wentworth:

Let $q_i$ be a holomorphic differential of degree $k$ on $X$ and let $\varphi(q_i)$ be the associated Hitchin deformation. Labourie and Wentworth in \cite{Variation_along_FuchsianLocus} show the pressure metric satisfies:

\begin{equation*}
     \langle \varphi(q_i), \varphi(q_i)\rangle_{\Sub P}=C(n,k){\langle q_i,q_i\rangle}_{X},
\end{equation*}
where $C(n,k)>0$ is a constant that does not depend on $\sigma$ and $\langle q_{i},q_{i}\rangle_{\Sub X}$ is the Petersson pairing:
\begin{align*}
    \langle q_{i},q_{i}\rangle_{\Sub X}=\int_{X} q_{i}\bar{q_{i}} \sigma^{-k}(z)\mathrm{d}A_\sigma
\end{align*}
with $\mathrm{d}A_\sigma=\sigma(z)\mathrm{d}x\wedge \mathrm{d}y$ denoting the area form for the hyperbolic metric $\sigma$.

If $q_i, q_j$ are holomorphic differentials of the same degree, then
\begin{align*}
     \langle \varphi(q_i), \varphi(q_j)\rangle_{\Sub P}&=
    \frac{1}{4}[ \langle \varphi(q_i+q_j), \varphi(q_i+q_j) \rangle_{\Sub P} - \langle\varphi(q_i-q_j), \varphi(q_i-q_j) \rangle_{\Sub P}]\\
    &=\frac{C(n,k)}{4}{\langle q_i+q_j,q_i+q_j \rangle}_{\Sub X}- \frac{C(n,k)}{4}{\langle  q_i-q_j,q_i-q_j \rangle}_{\Sub X}\\
    &=C(n,k){\langle q_i,q_j \rangle}_{\Sub X}
\end{align*}
If $q_i, q_j$ are holomorphic differentials of different degrees on $X$, Labourie and Wentworth show in \cite{Variation_along_FuchsianLocus} that
\begin{align}
    \langle \varphi(q_i), \varphi(q_j)\rangle_{\Sub P}=0
\end{align}

We denote the pressure metric components with respect to the coordinates introduced in Remark \ref{rem coordinate} as $g_{ij}$. Equivalently, the metric tensor $g_{ij}(\xi)$ means that the pressure metric $ \langle \cdot , \cdot \rangle_{\Sub P}$ is evaluated at $\xi$ with tangential vectors parallel to $q_{i}$-axis and $q_{j}$-axis. In particular, at the point $X$, we have $g_{ij}(0)=g_{ij}(\sigma)= \langle \varphi(q_i), \varphi(q_j)) \rangle_{\Sub P}$. It is always possible to choose an orthonormal basis $\{q_i\}$ with respect to our pressure metric from the vector space $\bigoplus\limits_{i=2}^{i=n} H^{0}(X,K^i)$ so that $ g_{ij}(\delta)=\delta_{ij}$.

\section{More Thermodynamic Formalism}

Bowen and Ruelle's work (\cite{Bowen-Markov-Partition}, \cite{Bowen_symbolicHyperFlow}, \cite{Bowen-Ruelle}) guarantee that many of the results in the thermodynamic formalism proved for subshifts of finite type by Ruelle operator still hold for Axiom A diffeomorphisms and Axiom A flows. We adopt this idea of simplifying the rather complicated object "flow" by discretizing it and studying a relative simple object "shift" given by symbolic coding. We will compute the formula of third derivatives of pressure functions using subshifts of finite type. The reader can find an introduction for modelling hyperbolic diffeomorphisms by subshifts of finite type and modelling hyperbolic flows by suspension flows through  Markov partition and symbolic dynamics in section 3, section 4 of \cite{Bowen-LectureNotes} and Appendix III of \cite{Pollicot_Zeta}.

The subsection \ref{subsection thermoAgain} is devoted to Ruelle operator and Ruelle-Perron-Frobenius Theorem. These are important tools for studying subshifts of finite types. Then in subsection \ref{subsection thirdDev},  we will compute third derivatives of pressure functions in Lemma \ref{lem 3rd}. It will be important for the proof of the main theorem in the next section.

\subsection{Ruelle operator and others}

\label{subsection thermoAgain}

We start with a cursory introduction to the elements of thermodynamic formalism for subshifts of finite types. A complete description is in \cite{McMullen2008} and \cite{Pollicot_Zeta}.

\begin{defn}
Let $A$ be a $k\times k$ matrix of zeros and ones, we define the associated two-sided shift of finite type $(\Sigma, \sigma_{A})$ where $\Sigma$ is the set of sequences
\begin{align*}
    \Sigma=\{x=(x_{n})_{n=-\infty}^{\infty}:x_{n}\in\{1,\cdots,k\},n\in \mathbb{Z}, A(x_n,x_{n+1})=1\}
\end{align*}
and $\sigma_{A}: \Sigma \to \Sigma$ is defined by $\sigma_{A}(x)=y$, where $y_n=x_{n+1}$.
\end{defn}

If instead, we consider $x=(x_{n})_{n=0}^{\infty}$ with the same restriction given by matrix $A$ and $\sigma(x)=y$, i.e. $y_n=x_{n+1}$ for $n\geq 0$, then we obtain a one-sided shift of finite type.

 The set $\{1, \cdots, k\}$ is equipped with the discrete topology and the two-sided (one-sided) shift space $\Sigma_{A}$ is equipped with the associated product topology.
 
 Given $\alpha\in(0,1)$, we can metrize the topology on the two-sided shift space $\Sigma$ by defining a metric $d_{\alpha}(x,y)=\alpha^N$ where $N$ is the largest non-negative integer such that $x_i=y_i$ for $|i|<N$.
 Similarly, we have a metric $d_{\alpha}$ defined for one-sided shift space. 
 
 We let $C(\Sigma)$ be the space of real-valued continuous functions on $\Sigma$ and $C^{\alpha}(\Sigma)$ be the space of real-valued H{\"o}lder functions on $\Sigma$ with H{\"o}lder exponent $\alpha$ with respect to $d_{\alpha}$.

The two-sided (one-sided) shift of finite type $(\Sigma,\sigma_{A})$ is called a subshift of finite type if $\sigma_{A}$ is topological transitive.

We define the pullback operator on $C^{\alpha}(\Sigma)$ by $(\sigma_{A}^{*}f)(y)=f(\sigma_{A}(y))$.
Similarly to Definition \ref{defn Livsic}, we define 

\begin{defn}
$f_{1}$ and $f_{2}$ in $C^{\alpha}(\Sigma)$ are (Liv\v{s}ic) cohomologous if 
\begin{align*}
    f_{1}-f_{2}=f_{3}-\sigma_{A}^{*}f_{3}.
\end{align*}
for some $f_{3}\in C^{\alpha}(\Sigma)$.
\end{defn}

From now on, we assume our subshift of finite type $(\Sigma, \sigma_{A})$ to be one-sided unless otherwise specified.

\begin{defn}
Given $w \in C^{\alpha}(\Sigma)$, the Ruelle operator (or transfer operator) on $f \in C^{\alpha}(\Sigma)$ is defined by 
\begin{align*}
    \mathcal{L}_{w}(f)(x)=\sum\limits_{\sigma_{A}(y)=x}e^{w(y)}f(y).
\end{align*}
\end{defn}

We have that the following holds for Ruelle operator $ \mathcal{L}_{w}$.

\begin{thm}(Ruelle-Perron-Frobenius)\hfill
\label{thm Ruelle-Frobenius}

Suppose $(\Sigma, \sigma_{A})$ is topologically mixing (i.e. $A_{i,j}^{M}>0\text{ } \forall i,j$, for some $M>0$, also called irreducible and aperiodic) and $w \in C^{\alpha}(\Sigma)$, then 
\begin{enumerate}
    \item  There is a simple maximal positive eigenvalue $\rho(\mathcal{L}_{w})$ of $\mathcal{L}_{w}: C^{\alpha}(\Sigma) \to C^{\alpha}(\Sigma)$ with a corresponding strictly positive eigenfunction $e^{\psi}$:
    \begin{align*}
        \mathcal{L}_{w}(e^{\psi})=\rho(\mathcal{L}_{w}) e^{\psi}.
    \end{align*}
     \item The remainder of the spectrum of $\mathcal{L}_{w}$ (excluding $\rho(\mathcal{L}_{w})$) is contained in a disc of radius strictly smaller than $\rho(w)$. 
    \item There is a unique probability measure  $\mu_{w}$ on $\Sigma$ so that 
    \begin{align*}
       {\mathcal{L}_{w}}^{*}\mu_{w}=e^{\psi} \mu_{w}.
    \end{align*}
\end{enumerate}
\end{thm}

The pressure $\mathbf{P}(w)$ of $w$, which can be defined in an analogous way as the pressure of functions on $UX$ by variational principle \ref{prop defn_pressure}, turns out to be related to the spectral radius of the Ruelle operator: $\mathbf{P}(w)=\log \rho(\mathcal{L}_{w})$
(see \cite[Theorem.1.22]{Bowen-LectureNotes}).

Associated to $\mu_w$ is another measure $m_w=e^{\psi}{\mu_w}$. It is called the equilibrium measure of $w$. It is an  $\sigma_{A}$-invariant and ergodic probability measure and satisfied ${\mathcal{L}_{w}}^{*}m_{w}= m_{w}$.

We will from now on assume $\mathbf{P}(w)=0$. As pressure functions and equilibrium measures depend only on cohomology class, we can modify $w$ by a coboundary so that $\mathcal{L}_{w}(1)=1$ and $\mu_{w}=m_{w}$. One notices this implies $\mathcal{L}_{w}(\sigma_{A}^{*}f)=f$.

Fixing $m_{w}$, we denote an inner product $<f_{1},f_{2}>:= \int_{\Sigma}f_1 f_2 \mathrm{d}m_{w}$ on the Banach space $C^{\alpha}(\Sigma)$.

For the convenience of notation, we also denote $S_{n}(f,x)= \sum\limits_{i=0}^{n-1}f(\sigma_{A}^{i}x)$.

The following two lemmas are application of Ruelle operators and will be useful in the next subsection.
\begin{lem}(McMullen, \cite[Theorem.3.2, Theorem.3.3]{McMullen2008})\hfill

For any $g\in C(\Sigma)$ and $f\in C^{\alpha}(\Sigma)$ with  $\int_{\Sigma} f \mathrm{d}m_{w}=0$, we have
\label{thm:McMullen}
\begin{equation*}
    \lim\limits_{n\to\infty}\langle g, S_{n}(f)^{2}/n \rangle = Var(f,m_{w}) \int_{\Sigma}g \mathrm{d}m_{w}=0
\end{equation*}
where $Var(f,m_{w})= \lim\limits_{n\to \infty} \frac{1}{n} \langle S_{n}(f), S_{n}(f) \rangle $. 
\end{lem}

\begin{lem}
\label{thm:ConvergenceOf3}
For any $f\in C^{\alpha}(\Sigma)$ with  $\int_{\Sigma} f \mathrm{d}m_{w}=0$, 
\begin{equation*}
    \lim\limits_{n\to\infty}\frac{1}{n} \int_{\Sigma} (S_{n}(f))^{3} \mathrm{d}m_{w} < \infty.
\end{equation*}
\end{lem}

\begin{proof}
This proof is similar to Theorem.3.3 of \cite{McMullen2008}.
\begin{align*}
  \frac{1}{n} \int_{\Sigma}(S_{n}(f))^{3}\mathrm{d}m=&\frac{1}{n} \sum\limits_{i=0}^{n-1}\sum\limits_{j=0}^{n-1}\sum\limits_{k=0}^{n-1}\langle f\circ \sigma_{A}^{i}\cdot f\circ \sigma_{A}^{j}, f\circ \sigma_{A}^{k} \rangle
\end{align*}
When $k>j>i$,
\begin{align*}
    &\langle f\circ \sigma_{A}^{i} \cdot f\circ \sigma_{A}^{j}, f\circ \sigma_{A}^{k}\rangle\\
    =&\langle \sigma_{A}^{*i}(f\cdot f\circ \sigma_{A}^{j-i}),\sigma_{A}^{*i}(f\circ \sigma_{A}^{k-i})\rangle \\
    =&\langle f\cdot f\circ \sigma_{A}^{j-i}, f\circ \sigma_{A}^{k-i}\rangle  \hspace{.5in}\text{$\sigma_{A}$-invariance of $m_{w}$}\\
    =&\langle f, f\circ \sigma_{A}^{j-i} \cdot f\circ \sigma_{A}^{k-i} \rangle\\
    =& \langle f,  \sigma_{A}^{*(j-i)}(f \cdot f\circ \sigma_{A}^{k-j}) \rangle\\
    =&\langle \mathcal{L}_{w}^{j-i}(f), f\cdot f\circ {\sigma_{A}}^{k-j}\rangle  \hspace{.5in}\text{$\mathcal{L}_{w}(\sigma_{A}^{*}f)=f$ and ${\mathcal{L}_{w}}^{*}m_{w}= m_{w}$}\\
     =&\langle f\cdot\mathcal{L}_{w}^{j-i}(f), f\circ {\sigma_{A}}^{k-j} \rangle
\end{align*}
We define a projection operator on $C^{\alpha}(\Sigma)$ by $P_{m_{w}}(h)(x)=h(x)- \int_{\Sigma}h\mathrm{d}m_{w}$. Because $P_{m_{w}}(h)$ has mean zero with respect to $m_{w}$. The spectrum of the operator $T_{w}=\mathcal{L}_{w}\circ P_{m_{w}}$ lies in a disk of radius $r<1$ by Ruelle-Perron-Frobenius Theorem.

One has 
\begin{equation}
    \langle h_{1}, h_{2}\circ \sigma \rangle = \langle T_{w}(h_{1}), h_{2} \rangle
    \label{eq tw}
\end{equation}
whenever $h_{1}$ or $h_{2}$ has mean zero. 

Because $f$ is mean zero with respect to $m_{w}$. $T_{w} (f)= \mathcal{L}_{w}(f)$. Moreover, we have
\begin{align*}
    &\langle f\cdot\mathcal{L}_{w}^{j-i}(f), f\circ {\sigma_{A}}^{k-j} \rangle \\
    =&\langle f\cdot T_{w}^{j-i}(f), f\circ {\sigma_{A}}^{k-j} \rangle\\
    =&\langle T_{w}^{k-j}(f\cdot T_{w}^{j-i}(f)), f \rangle   \hspace{.5in} \text{by equation (\ref{eq tw})}\\
    \leq & \norm{T^{k-j}}\norm{T^{j-i}}\norm{f}^{3}\\
    \leq & C r^{k-i} \hspace{.5in} \text{for some $C > 0$}
\end{align*}
where the norm for $T$ is the operator norm.

Thus 
\begin{align*}
 &\frac{1}{n} \sum\limits_{0\leq i<j<k\leq n-1}\langle f\circ \sigma_{A}^{i}\cdot f\circ \sigma_A^{j}, f\circ \sigma_A^{k} \rangle \\
 \leq& \frac{C}{n} \sum\limits_{k=0}^{n-1}\sum\limits_{i=0}^{k}(k-i)r^{k-i} \hspace{.5in} \text{by the estimate above}\\
 =&\frac{C}{n} \sum\limits_{k=0}^{n-1}\sum\limits_{s=0}^{k}sr^{s}\\
 =& \frac{Cr}{1-r}(1-\frac{1}{n}\sum\limits_{k=1}^{n}r^{k}) < \infty  \hspace{.5in} \text{when $n\to \infty$}
\end{align*}
This shows  $\lim\limits_{n\to\infty}\frac{1}{n} \int_{\Sigma} (S_{n}(f))^{3} \mathrm{d}m_{w} < \infty$.

\end{proof}

\subsection{Third derivatives of pressure functions}
\label{subsection thirdDev}
Our goal in this subsection is to compute third derivatives of pressure functions in Lemma \ref{lem 3rd}. For this, we first need to compute third derivatives of pressure functions for subshifts of finite types by the method of the Ruelle operator and generalize it to our setting of suspension flows.  

We start from introducing suspension flows. We will also recall Bowen's celebrated results applying to our setting that suspension flows efficiently model the geodesic flow on $UX$.

\begin{defn}
Suppose $(\Sigma,\sigma_A)$ is a two-sided shift of finite type. Given a roof function $r: \Sigma \to\mathbb{R}^{+}$, the suspension flow of $(\Sigma,\sigma_A)$ under $r$ is the quotient space
\begin{align*}
    {\Sigma}_{r}=\{(x,t)\in  \Sigma \times \mathbb{R}: 0\leq t\leq r(x), x\in \Sigma\}/(x,r(x))\sim (\sigma_A(x),0) 
\end{align*}
equipped with the natural flow $\sigma^{r}_{A,s}(x,t)=(x,t+s)$
\end{defn}
Any $\sigma_A$-invariant probability measure $m$ on $\Sigma$ induces a natural $\sigma^{r}_{A,s}$-invariant probability measure on ${\Sigma}_{r}$
\begin{equation}
\mathrm{d}m_{r}=\frac{\mathrm{d}m\mathrm{d}t}{\int_{\Sigma}r\mathrm{d}m}. \label{eq measureCor}
\end{equation}
This correspondence gives a bijection between $\sigma_A$-invariant probability measures and $\sigma^{r}_{A,s}$-invariant probability measures. 

Bowen shows in \cite{Bowen-Markov-Partition} the construction of Markov partitions for Axiom A diffeomorphism. He then shows how to model Axiom A flows via the Markov partition and symbolic dynamics in \cite{Bowen_symbolicHyperFlow}. We illustrate the version of this celebrated result in our context (see also \cite{Ratner1973}): the geodesic flow $\Phi$ admits a Markov coding $({\Sigma}_{A}, \pi, r)$ where $({\Sigma}_{A},\sigma_{A})$ is a topological mixing two-sided shift of finite type, the roof function $r: \Sigma_{A} \to \mathbb{R}^{+}$ is a H{\"o}lder continuous and the map $\pi: \Sigma_{A} \to UX$ is also H{\"o}lder continuous. The suspension flow   $\sigma^{r}_{A,t}$ models $\Phi_{t}$ effectively in the following sense:
\begin{itemize}
    \item $\pi$ is surjective;
    \item $\pi$ is one-to-one on a set of full measure (for any ergodic measure of full support) and on a residual set;
    \item $\pi$ is finite-to-one;
    \item $\pi\sigma^{r}_{A,t}=\Phi_{t}\pi$  \text{for all $t\in\mathbb{R}$}.
\end{itemize}

Now we are able to state and prove the major proposition in this subsection.

\begin{lem}
\label{lem 3rd}
Let $F_{s}$ be a smooth family in $C^h(UX)$ such that $\mathbf{P}(F_{0})=0$ and $\partial_{s}\mathbf{P}(F_{s})|_{s=0}$. Then

\begin{align}
\begin{split}
    \frac{\mathrm{d}^{3}\mathbf{P}(F_{s})}{\mathrm{d} s^{3}}\bigg|_{s=0}
    &=
    \int_{\Sub UX} \partial_{s}^{3} F_{0}(x)\mathrm{d}m_{F_{0}}(x)\\
    &+ \lim_{r\to \infty} \frac{1}{r} \Bigg( 3 \int_{\Sub UX} \int_{0}^{r} \partial_{s}{F_{0}}(\Phi_{t}(x)) \mathrm{d}t \int_{0}^{r} \partial_{s}^{2}{F_{0}}(\Phi_{t}(x)) \mathrm{d}t  \mathrm{d}m_{F_{0}}(x)\\
    &+ \int_{\Sub UX} \left(\int_{0}^{r}\partial_{s}{F_{0}}(\Phi_{t}(x))\mathrm{d}t \right)^3\mathrm{d}m_{\Sub F_{0}}(x)\Bigg).
    \label{eq:third dev of PII}
    \end{split}
\end{align}

In particular, if $F(u,v,w)$ is a smooth three-parameter family of H{\"o}lder functions on $UX$ such that $\mathbf{P}(F(0,0,0))=0$ and all of the first variations of $\mathbf{P}(F(u,v,w))$ are zero, then

\begin{align}
\begin{split}
&\frac{\partial^{3}\mathbf{P}(F(u,v,w))}{\partial u \partial v \partial w }\bigg|_{u=v=w=0}
=
  \int_{\Sub UX} \partial_{u}\partial_{v} \partial_{w} F(0)(x)\mathrm{d}m_{F(0)}(x)
  \\
  &\indentEq+\lim_{r\to \infty}\frac{1}{r}\Bigg(
  \int_{\Sub UX} \left(\int_{0}^{r}\partial_{u}F(0)(\Phi_{t}(x))\mathrm{d}t \right) \left(\int_{0}^{r}\partial_{v}F(0)(\Phi_{t}(x))\mathrm{d}t \right) \left(\int_{0}^{r}\partial_{w}F(0)(\Phi_{t}(x))\mathrm{d}t \right)\mathrm{d}m_{\Sub F(0)}(x)
  \\
&\indentEq+ \int_{\Sub UX} \left(\int_{0}^{r}\partial_{u}F(0)(\Phi_{t}(x))\mathrm{d}t \right) \left(\int_{0}^{r}\partial_{vw}F(0)(\Phi_{t}(x))\mathrm{d}t \right)\mathrm{d}m_{\Sub F(0)}(x)
\\
&\indentEq+\int_{\Sub UX} \left(\int_{0}^{r}\partial_{v}F(0)(\Phi_{t}(x))\mathrm{d}t \right) \left(\int_{0}^{r}\partial_{uw}F(0)(\Phi_{t}(x))\mathrm{d}t \right)\mathrm{d}m_{\Sub F(0)}(x)
\\
&\indentEq+\int_{\Sub UX} \left(\int_{0}^{r}\partial_{w}F(0)(\Phi_{t}(x))\mathrm{d}t \right) \left(\int_{0}^{r}\partial_{uv}F(0)(\Phi_{t}(x))\mathrm{d}t \right)\mathrm{d}m_{\Sub F(0)}(x)\Bigg).
\label{eq:third dev of P}
\end{split}
\end{align}

\end{lem}

\begin{proof}\hfill

The proof proceeds in two steps. In the first step, we find a formula for third derivatives of pressure functions for topological mixing shifts of finite type. In the second step, we show how the computation can be carried to geodesic flows through symbolic coding and suspension flows.
\begin{itemize}
    \item Step 1.
    
 The computation of first and second derivatives of pressure functions for aperiodic shifts of finite type are shown in Parry and Pollicott's book \cite{Pollicot_Zeta} by Ruelle operator. We will give a computation of the third derivative by the same method and then generalize it to our flow case.
 
Let $(\Sigma_{A},\sigma_{A})$ be a (either one-sided or two-sided) shift of finite type that is topological mixing. We assume $f_{s}$ is a smooth family of functions on $C^{\alpha}(\Sigma_{A})$ such that $ \mathbf{P}(f_{0})=0$ and $\partial_{s}\mathbf{P}(f_{s})|_{s=0}$. We will prove 
\begin{align}
    \partial_{s}^{3}\mathbf{P}(f_{s})\bigg|_{s=0}=\lim_{n\to\infty}\frac{1}{n}\int_{\Sub X}(S_{n}(\partial_{s}f_{0}))^{3}\mathrm{d}m_{f_{0}}+\lim_{n\to\infty}\frac{3}{n}\int_{\Sub X}S_{n}(\partial_{s}f_{0})S_{n}(\partial_{s}^{2}f_{0})\mathrm{d}m_{f_{0}}+\int_{\Sub X} \partial_{s}^{3} f_{0} \mathrm{d} m_{f_{0}}
    \label{eq partial3}.
\end{align}

 Any H{\"o}lder function on a two-sided shift space is colomologous to a H{\"o}lder function depending only on the corresponding one-sided shift space (see \cite[Proposition 1.2]{Pollicot_Zeta}). It suffices to prove equation \eqref{eq partial3} for one-sided shifts of finite type. We assume $(\Sigma_A, \sigma_A)$ is one-sided and $f_{s}$ is a smooth family of H{\"o}lder function (with possibly a different H{\"o}lder exponent from $\alpha$) on $\Sigma_A$.

We change $f_{0}$ in its cohomology class so that $\mathcal{L}_{f_{0}}(1)=1$. 

Following the method in \cite{Pollicot_Zeta}, let $Q(s)$ be a projection-valued function which is analytic in $s$ and satisfies
\begin{equation*}
    \mathcal{L}_{f_s}Q(s)=Q(s)\mathcal{L}_{f_{s}}.
\end{equation*}

Let $w(s): \Sigma_{A}\to \mathbb{R}$ be $w(s)(x):=Q(s)\cdot 1$. So
\begin{equation}
    \mathcal{L}_{f_s}w(s)=e^{\mathbf{P}(f_s)}w(s)
    \label{eq:Ruelle}
\end{equation}
and $w(0)(x)=Q(0)\cdot 1= 1$.

Iterate equation \text{\eqref{eq:Ruelle}} n-times and take 3rd s-derivatives of both sides at $s=0$.
\begin{equation}
  \partial_{s}^{3}\left (\sum\limits_{\sigma_A y=x}e^{S_{n}(f_s)(y)} w(s)(y)\right )\bigg|_{s=0} =\partial_{s}^{3}(e^{n\mathbf{P}(f_s)}w(s))|_{s=0}.
  \label{eq:Ruelle-3dev}
\end{equation}

Notice $\mathbf{P}(f_0)=0$, $\partial_{s}\mathbf{P}(f_{s})|_{s=0}=0$ and $\int_{\Sub UX} \partial_{s}f_{0} \mathrm{d}m_{\Sub f_{0}}=0$. Integrating both sides of equation \eqref{eq:Ruelle-3dev} with respect to $m_{f_{0}}$ yields,
\begin{align*}
    &3n\partial_{s}^{2}\mathbf{P}(f_{s})\bigg|_{s=0}\int_{\Sub X} \partial_{s}w(0)\mathrm{d}m_{f_{0}}+n\partial_{s}^{3}\mathbf{P}(f_{s})\bigg|_{s=0}\\
    =&\int_{\Sub X}S_{n}(\partial_{s}^{3} f_{0})\mathrm{d}m_{f_{0}}+
    3\int_{\Sub X} (S_{n}(\partial_{s}f_{0})^{2}+S_{n}(\partial_{s}^{2}f_{0}))\partial_{s}w(0)\mathrm{d}m_{f_{0}}\\
    +&3\int_{\Sub X} S_{n}(\partial_{s} f_{0}) \partial_{s}^{2}w(0)\mathrm{d}m_{f_{0}}+3 \int_{\Sub X} S_{n}(\partial_{s}f_{0}) S_{n}(\partial_{s}^{2}f_{0})\mathrm{d}m_{f_{0}}+\int_{\Sub X}S_{n}(\partial_{s}f_{0})^{3}\mathrm{d}m_{f_{0}}.
\end{align*}
Divide by $n$ and take $n \to \infty$. From ergodicity of $m_{f_{0}}$, we may evaluate two of the resulting terms:
\begin{align*}
    &\lim_{n\to\infty}\frac{1}{n}\int_{\Sub X} S_{n}(\partial_{s}f_{0})\partial_{s}^{2}w(0)\mathrm{d}m_{f_{0}}= \int_{\Sub X} \partial_{s}f_{0}\mathrm{d}m_{f_0} \int_{\Sub X}\partial_{s}^{2}w(0)\mathrm{d}m_{f_{0}}=0.\\
     &\lim_{n\to\infty}\frac{1}{n}\int_{\Sub X} S_{n}(\partial_{s}^{2}f_{0})\partial_{s}w(0)\mathrm{d}m_{f_{0}}= \int_{\Sub X} \partial_{s}^{2}f_{0}\mathrm{d}m_{f_0}  \int_{\Sub X}\partial_{s}w(0)\mathrm{d}m_{f_{0}}.
\end{align*}
We also notice that by applying Lemma \ref{thm:McMullen} and the formula for second derivatives of pressure functions, we have the following equality
\begin{equation*}
    \partial_{s}^{2}\mathbf{P}(f_{s})\bigg|_{s=0}\int_{\Sub X}\partial_{s}w(0)\mathrm{d}m_{f_{0}}=\lim_{n\to\infty}\frac{1}{n}\int_{\Sub X}S_n(\partial_{s}f_{0})^{2}\partial_{s}w(0)\mathrm{d}m_{f_{0}}+\lim_{n\to\infty}\frac{1}{n}\int_{\Sub X} S_{n}(\partial_{s}^{2}f_{0})\partial_{s}w(0)\mathrm{d}m_{f_{0}}
\end{equation*}

Therefore we obtain a formal expression
\begin{align*}
    \partial_{s}^{3}\mathbf{P}(f_{s})\bigg|_{s=0}=\lim_{n\to\infty}\frac{1}{n}\int_{\Sub X}(S_{n}(\partial_{s}f_{0}))^{3}\mathrm{d}m_{f_{0}}+\lim_{n\to\infty}\frac{3}{n}\int_{\Sub X}S_{n}(\partial_{s}f_{0})S_{n}(\partial_{s}^{2}f_{0})\mathrm{d}m_{f_{0}}+\int_{\Sub X} \partial_{s}^{3} f_{0} \mathrm{d} m_{f_{0}}
\end{align*}

We observe each term of the right-hand side converges.
$\lim\limits_{n\to\infty}\frac{3}{n}\int_{\Sub X}S_{n}(\partial_{s}f_{0})S_{n}(\partial_{s}^{2}f_{0})\mathrm{d}m_{f_{0}} < \infty$ is guaranteed by Corollary \ref{cor:CovNotMeanZero} and $\lim\limits_{n\to\infty}\frac{1}{n}\int_{\Sub X}(S_{n}(\partial_{s}f_{0}))^{3}\mathrm{d}m_{f_{0}} < \infty$ has been shown in Lemma \ref{thm:ConvergenceOf3}.
\vspace{5mm}

\item Step 2.
 We now explain how we obtain the flow version of the above formula.

Suppose $F_{s}$ is a smooth family of functions in $C^h(UX)$ such that $\mathbf{P}(F_{s})=0$. We have a topologically mixing Markov coding $(\Sigma_{A}, \pi, r)$ for $UX$. Because of the conjugacy $\pi\sigma^{r}_{A,t}=\Phi_{t}\pi$ between geodesic flow and the suspension flow of $(\Sigma_{A}, \pi,r)$, it suffices to prove equation (\ref{eq:third dev of PII}) for $F_{s}\circ \pi: \Sigma_{A,r} \to \mathbb{R}$ on suspension space with pull back measure $\pi^{*}m_{F_{0}}$. For simplicity, we still denote $F_{s}\circ \pi$ as $F_{s}$ and $\pi^{*}m_{F_{0}}$ as $m_{F_{0}}$. 

We then want to reduce the problem of proving equation (\ref{eq:third dev of PII}) for suspension flows to proving it for subshifts of finite type.
We construct a function $\hat{F}_{s}:\Sigma_{A} \to \mathbb{R}$ from the function $F_{s}$ on the suspension space as:

\begin{equation}
\hat{F}_{s}(x)=\int_{0}^{r(x)} F_{s}(x,t)\mathrm{d}t.
\label{eq funCor}
\end{equation}

As $F_{s}$ and $r$ are H{\"o}lder on $\Sigma_{A,r}$ and $\Sigma_{A}$ respectively, the function $\hat{F}_{s}$ is clearly H{\"o}lder. Denoting the set of $\sigma_{A}^{r}$-invariant probability measures as $\mathcal{M}^{\sigma_{A}^{r}}$ and the set of $\sigma_{A}$-invariant probability measures as $\mathcal{M}^{\sigma_{A}}$, we have: 
\begin{align*}
     \mathbf{P}(\sigma^{r}_{A,t}, F_{s})=&\sup_{m_r\in \mathcal{M}^{\sigma_{A}^{r}}}\left(h(\sigma^{r}_{A,1},m_r)+ \int_{\Sigma_{A,r}}F_{s} \mathrm{d}m_r\right)\\
     =&\sup_{m \in \mathcal{M}^{\sigma_{A}}}\frac{h(\sigma_{A},m)+ \int_{\Sigma_{A}}F_{s} \mathrm{d}m}{\int_{\Sigma_{A}}r \mathrm{d}m}.
\end{align*}

Let $c_{s}=\mathbf{P}(\sigma^{r}_{A,t}, F_{s})$, we have the following relation between the pressure function of $F_{s}$ and the pressure function of $\hat{F}_{s}$ (also see \cite{Bowen-Ruelle}) 
\begin{align}
    \mathbf{P}(\sigma_{A}, \hat{F}_{s}-c_{s}r)=0
    \label{eq RelationForPressures}
\end{align}
 
 Denote $\partial_{s}c_{0}=\partial_{s}(c_s)|_{s=0}$ and $\partial_{ss}c_{0}=\partial_{s}^{2}(c_s)|_{s=0}$.
 
We have the assumption $\partial_{s}c_{0}=0$. Without loss of generality, we can also assume $\partial_{s}^{2}c_{0}=0$. Otherwise we consider the family of functions $\tilde{F}_{s}:=F_{s}-\frac{1}{2}s^{2}\partial_{s}^{2}c_{0}$. Clearly $\partial_{s}\mathbf{P}(\tilde{F}_{s})|_{s=0}= \partial_{s}^{2}\mathbf{P}(\tilde{F}_{s})|_{s=0}=0$ and $\partial_{s}^{3}\mathbf{P}(\tilde{F}_{s})|_{s=0}=\partial_{s}^{3}\mathbf{P}(F_{s})|_{s=0}$.

Now let's take the third $s$-derivative of equation \text{\eqref{eq RelationForPressures}} with the assumptions $\partial_{s}c_{0}=\partial_{s}^{2}c_{0}=0$. By equation \text{\eqref{eq partial3}}, 
 
\begin{align*}
    0=&\partial_{s}^{3}\mathbf{P}(\hat{F}_{s}-c_{s}r)\bigg|_{s=0}\\
    =&\lim_{n\to\infty}\frac{1}{n}\int_{\Sub \Sigma_{A}}(S_{n}(\partial_{s}\hat{F}_{0}))^{3}\mathrm{d}m_{\hat{F}_{0}}+\lim_{n\to\infty}\frac{3}{n}\int_{\Sub \Sigma_{A}}S_{n}(\partial_{s}\hat{F}_{0})S_{n}(\partial_{s}^{2}\hat{F}_{0})\mathrm{d}m_{\hat{F}_{0}}\\
    +&\int_{\Sub \Sigma_{A}} (\partial_{s}^{3} \hat{F}_{0}-\partial_{s}^{3}c_{0}r)\mathrm{d} m_{\hat{F}_{0}}.
\end{align*}

This yields 
\begin{align*}
    \partial_{s}^{3}c_{0}=\partial_{s}^{3}\mathbf{P}(\sigma^{r}_{A,t}, F_{s})\bigg|_{s=0}=\left(\int_{\Sigma_{A}}r\mathrm{d}m_{\hat{F}_{0}}\right)^{-1}\partial^{3}_{s}\mathbf{P}(\sigma_{A}, \hat{F}_{s})\bigg|_{s=0}.
\end{align*}

Therefore proving equation (\ref{eq:third dev of PII}) for $F_s$ is equivalent to proving the following
\begin{align*}
&\lim_{r\to\infty}\frac{1}{r}\int_{\Sub \Sigma_{A,r}}\left(\int_{0}^{r}\partial_{s}F_{0} \mathrm{d}t\right)^{3}\mathrm{d}m_{F_{0}}+\lim_{r\to\infty}\frac{3}{r}\int_{\Sub \Sigma_{A,r}}\int_{0}^{r}\partial_{s}F_{0}\mathrm{d}t\int_{0}^{r}\partial_{ss}F_{0}\mathrm{d}t\mathrm{d}m_{F_{0}}+\int_{\Sub \Sigma_{A,r}} \partial_{s}^{3} F_{0} \mathrm{d} m_{F_{0}}\\
=&\left(\int_{\Sigma_{A}}r\mathrm{d}m_{\hat{F}_{0}}\right)^{-1}\bigg(\lim_{n\to\infty}\frac{1}{n}\int_{\Sub \Sigma_{A}}(S_{n}(\partial_{s}\hat{F}_{0}))^{3}\mathrm{d}m_{\hat{F}_{0}}+\lim_{n\to\infty}\frac{3}{n}\int_{\Sub \Sigma_{A}}S_{n}(\partial_{s}\hat{F}_{0})S_{n}(\partial_{s}^{2}\hat{F}_{0})\mathrm{d}m_{\hat{F}_{0}}\bigg)\\
+&\left(\int_{\Sigma_{A}}r\mathrm{d}m_{\hat{F}_{0}}\right)^{-1}\int_{\Sub \Sigma_{A}} \partial_{s}^{3} \hat{F}_{0} \mathrm{d} m_{\hat{F}_{0}}.
\end{align*}

Each term on the left is actually equal to the corresponding term on the right. We show here how to obtain
\begin{equation}
    \lim_{r\to\infty}\frac{1}{r}\int_{\Sub \Sigma_{A,r}}\left(\int_{0}^{r}\partial_{s}F_{0}(\sigma_{t}^{r}(y)) \mathrm{d}t\right)^{3}\mathrm{d}m_{F_{0}}(y) =\left(\int_{\Sigma_{A}}r\mathrm{d}m_{\hat{F}_{0}}\right)^{-1}\lim_{n\to\infty}\frac{1}{n}\int_{\Sub \Sigma_{A}}(S_{n}(\partial_{s}\hat{F}_{0})(x))^{3}\mathrm{d}m_{\hat{F}_{0}}(x) \label{eq ShiftToFlow}\\
\end{equation}
The other two terms follow a similar analysis.

To see equation \eqref{eq ShiftToFlow}, we begin by noting the following identity (\cite{Pollicott_OntheRateOfMixing}) where we denote $y=(x,u)$,
\begin{align*}
    \partial_{s}F_{0}(\sigma_{A,t}^{r}(x,u))=\sum\limits_{n\in\mathbb{Z}}\left(\int_{0}^{r(\sigma_A^{n}x)}\partial_{s}F_{0}(\sigma_A^{n}x,v)\delta(u+t-v-r^{n}(x))\mathrm{d}v\right),
\end{align*}
where $r^{n}(x)=r(x)+r(\sigma_A x) +\cdots + r(\sigma_A^{n-1}x)$ for $n> 0$ and $r^{0}(x)=0$ and $r^{-n}(x)=-(r({\sigma_A}^{-1} x) +\cdots + r(\sigma_A^{-n}x))$ for $n\geq 1$.

One has from Definition \ref{defn Cov}, measure correspondence (\ref{eq measureCor}) and formula (\ref{eq funCor}) that
\begin{align*}
 & \lim_{r\to\infty}\frac{1}{r}\int_{\Sub \Sigma_{A,r}}\left(\int_{0}^{r}\partial_{s}F_{0}(\sigma_{A,t}^{r}(y)) \mathrm{d}t\right)^{3}\mathrm{d}m_{F_{0}}(y)\\
 =&\int_{-\infty}^{\infty}\int_{-\infty}^{\infty}\int_{\Sub \Sigma_{A,r}}\partial_{s}F_{0}(y)\partial_{s}F_{0}(\sigma_{A,t}^{r}(y))\partial_{s}F_{0}(\sigma_{A,v}^{r}(y)) \mathrm{d}m_{F_{0}}(y)\mathrm{d}t\mathrm{d}v\\
 =&\left(\int_{\Sigma_{A}}r\mathrm{d}m_{\hat{F}_{0}}\right)^{-1}\int_{-\infty}^{\infty}\int_{-\infty}^{\infty}\int_{\Sub \Sigma_{A}}\int_{0}^{r(x)}\partial_{s}F_{0}(x,u)\partial_{s}F_{0}(\sigma_{A,t}^{r}(x,u))\partial_{s}F_{0}(\sigma_{A,v}^{r}(x,u))\mathrm{d}u  \mathrm{d}m_{\hat{F}_{0}}(x)\mathrm{d}t\mathrm{d}v\\
 =&(\int_{\Sigma_{A}}r\mathrm{d}m_{\hat{F}_{0}})^{-1}\sum\limits_{m,n\in\mathbb{Z}}\int_{\Sub \Sigma_{A}}\mathrm{d}m_{\hat{F}_{0}}(x)\int_{0}^{r(x)}\partial_{s}F_{0}(x,u)\mathrm{d}u\int_{0}^{r(\sigma_A^{n}x)}\partial_{s}F_{0}(\sigma_A^{n}x,v)\mathrm{d}v\int_{0}^{r(\sigma_A^{m}x)}\partial_{s}F_{0}(\sigma_A^{m}x,v)\mathrm{d}v\\
 =&\left(\int_{\Sigma_{A}}r\mathrm{d}m_{\hat{F}_{0}}\right)^{-1}\sum\limits_{m,n\in\mathbb{Z}}\int_{\Sub \Sigma_{A}}\partial_{s}\hat{F}_{0}(x)\partial_{s}\hat{F}_{0}(\sigma_A^{n}x)\partial_{s}\hat{F}_{0}(\sigma_A^{m}x)\mathrm{d}m_{\hat{F}_{0}}(x)\\
 =&\left(\int_{\Sigma_{A}}r\mathrm{d}m_{\hat{F}_{0}}\right)^{-1}\lim_{n\to\infty}\frac{1}{n}\int_{\Sub \Sigma_{A}}(S_{n}(\partial_{s}\hat{F}_{0})(x))^{3}\mathrm{d}m_{\hat{F}_{0}}(x).
\end{align*}

We therefore obtain a suspension flow version of equation \text{\eqref{eq partial3}} for $F_{s}$.
\end{itemize}

The arguments for three-parameters families are the same as the one-parameter case. In fact, since the operator $\partial_{u}\partial_{v}\partial_{w}$ is a symmetric multi-linear map in $u,v,w$ that is completely characterized by its values on the diagonal, one can deduce equation \ref{eq:third dev of P} for multi-variable cases directly from equation \ref{eq:third dev of PII} for one-parameter family.
\end{proof}

Next we introduce a formula for taking derivatives of integrals over varying measures by tools of thermodynamic formalism. This formula will be very useful in later proofs.

\begin{lem}
Suppose $\{f_{s}\}_{s\in(-1,1)}$ is a smooth family of pressure zero H{\"o}lder functions over $UX$ and suppose $\{m_{\Sub f_{s}}\}_{s\in(-1,1)}$ is the associated family of equilibrium states. Suppose furthermore that $\{w_{s}\}_{s\in(-1,1)}$ is another smooth family of H{\"o}lder functions over $UX$. Then

\begin{equation}
    \partial_{s}\left(\int_{\Sub UX} w_{s} \mathrm{d}m_{\Sub f_{s}} 
    \right)\bigg|_{s=0}=Cov(w_{0},\partial_{s}f_{0}, m_{\Sub f_{0}})+ \int_{\Sub UX} \partial_{s}w_{0} \mathrm{d}m_{\Sub f_{0}}
    \label{eq: VS}
\end{equation}
\end{lem}

\begin{proof}
\begin{align*}
    \partial_{s}\left(\int_{\Sub UX} w_{s} \mathrm{d}m_{\Sub f_{s}}
    \right)\bigg|_{s=0}&=\partial_{s}\left( {\frac{\partial\mathbf{P}(f_{s}+tw_{s})}{\partial t} \bigg|_{t=0}} \right)\bigg|_{s=0}
    \hspace{.5in}\text{by formula \eqref{eq: First dev of P}}\\
    &= \frac{\partial^{2}\mathbf{P}(f_{s}+tw_{s})}{\partial s \partial t} \bigg|_{s=t=0} \\
    &=Cov(P_{m_{f_{0}}}(w_{0}), P_{m_{f_{0}}}(\partial_{s}f_{0}), m_{\Sub f_{0}})+ \int_{\Sub UX} \partial_{s}w_{0} \mathrm{d}m_{\Sub f_{0}}\text{ \hspace{.5in} by equation \eqref{eq: Second dev of P}}\\     
    &=Cov(P_{m_{f_{0}}}(w_{0}),\partial_{s}f_{0} , m_{\Sub f_{0}})+ \int_{\Sub UX} \partial_{s}w_{0} \mathrm{d}m_{\Sub f_{0}}\\
     &=Cov(w_{0},\partial_{s}f_{0} , m_{\Sub f_{0}})+ \int_{\Sub UX} \partial_{s}w_{0} \mathrm{d}m_{\Sub f_{0}}   \hspace{.5in}\text{by Corollary  \ref{cor:CovNotMeanZero}}
\end{align*}
\end{proof}

\section{Proof of the main theorem: initial steps}

We first restate our main theorem. Recall
\cref{thm main}
\main*
We want to show  $\partial_{k}g_{ij}(\sigma)=0$ for the pressure metric components $g_{ij}$ with respect to the coordinates introduced before in Remark \ref{rem coordinate} for all possible $i,j,k$.   

\subsection{Some geometrical observation}

In this subsection, we conclude some derivatives of metric tensors vanish by some geometric observation. Starting from the next section, we will develop a general method to compute first derivatives of the pressure metric by the thermodynamic formalism.

From now on, we restrict ourselves to the Hitchin component $\mathcal{H}_{3}(S)$. Suppose $\{q_{i}\}$ is a basis of holomorphic differentials in $H^{0}(X,K^2)\bigoplus H^{0}(X,K^3)$ and suppose $\{\varphi(q_{i})\}$ is the associated Hitchin deformation given in Definition \ref{defn HitchinDeformation}. Recall we use the notation $g_{ij}(\sigma)= \langle \varphi(q_i), \varphi(q_j)) \rangle_{\Sub P}$ to emphasize the metric tensor is evaluated at $\sigma \in \mathcal{T}(S)$. We also assume $ g_{ij}(\delta)=\delta_{ij}$. 

Futhermore, instead of using only the English letters $i,j,k$ to denote arbitrary holomorphic differentials of degree 2 and 3, we let the English letters $i,j,k$ to only refer to quadratic differentials $q_{i}, q_{j}, q_{k} \in H^{0}(X,K^2)$ from now on. Therefore the corresponding Hitchin deformations $\varphi(q_i),\varphi(q_j),\varphi(q_k)$ are tangential directions to Fuchsian locus in $T_X\mathcal{H}_{3}(S)$. We also use the Greek letters $\alpha, \beta, \gamma$  to refer to cubic differentials $q_{\alpha}, q_{\beta}, q_{\gamma} \in H^{0}(X,K^3)$.  Then the corresponding Hitchin deformation $\varphi(q_{\alpha}), \varphi(q_{\beta}), \varphi(q_{\gamma}) $ are normal directions to the Fuchsian locus in $T_X\mathcal{H}_{3}(S)$ with respect to the pressure metric.

With the above notation understood, we have in total six types of first derivatives of metric tensors that need to be considered: $\partial_{k}g_{ij}, \partial_{j}g_{i \alpha}, \partial_{\alpha}g_{ij}, \partial_{i}g_{\alpha \beta},\partial_{\beta}g_{i\alpha}, \partial_{\gamma}g_{\alpha \beta}$. Our goal is to prove they all vanish.

We first notice the following facts.

\begin{enumerate}
\item $\partial_{k}g_{ij}(\sigma)=0$.

To see this, note that the pressure metric is a constant multiple of the Weil-Petersson metric on Teichm{\"u}ller space $\mathcal{T}(S)$.
Because the coordinates system in terms of quadratic differentials from the Hitchin reparametrization agrees with Bers coordinates through second order in the case of $\mathcal{T}(S)$ (\cite[Corollary 5.2, Corollary 5.4]{Mike_Harmonic}). The Bers coordinates are geodesic (\cite{Some_remarks_on_TeichSpace}) for the Weil-Petersson metric implies that for the pressure metric: $\partial_{k}g_{ij}(\sigma)=0$.   

\item
$\partial_{j}g_{\alpha i}(\sigma)=0 \Longrightarrow \partial_{\alpha}g_{ij}(\sigma)=0$.

The contragredient involution $\kappa: \mathrm{PSL}(3,\mathbb{R})\to \mathrm{PSL}(3,\mathbb{R})$ given by $\kappa(g)=(g^{-1})^{t}$ induces an involution $\hat{\kappa}$ on $\mathcal{H}_{3}(S)$ by $\hat{\kappa}(\rho)(\gamma)=\kappa(\rho(\gamma))$. Because $\hat{\kappa}$ is an isometry of $\mathcal{H}_{3}(S)$ with respect to the pressure metric and the fixed points set of $\hat{\kappa}$ is $\mathcal{T}(S)$, the Fuchsian locus is in fact totally geodesic in $\mathcal{H}_{3}(S)$( see \cite{SimpleLengthRigidity}). So for  $\tilde{\nabla}$ the Levi-Civita connection of the pressure metric and any $X,Y \in T_{\sigma} \mathcal{T}(S)$, we have
\begin{equation}
    \Pi(X,Y)=(\tilde{\nabla}_XY)^{\bot}=0.
\end{equation}
 Thus the Christoffel symbols for connection $\tilde{\nabla}$ satisfy: $\Gamma_{ij}^{\alpha}(\sigma)=0$ and because
 \begin{align*}
   \Gamma_{ij}^{\alpha}&=\frac{1}{2}g^{\beta \alpha}(\partial_{j}g_{i\beta}+\partial_{i}g_{j\beta}-\partial_{\beta}g_{ji}) \hspace{.5 in}\text{since  $g_{ k \alpha}(\sigma)=0,g^{ k \alpha}(\sigma)=0$} \\
   &=\frac{1}{2}g^{\alpha \alpha}(\partial_{j}g_{i\alpha}+\partial_{i}g_{j\alpha}-\partial_{\alpha}g_{ji}) \hspace{.5 in}\text{since  $g_{\alpha \beta}=\sigma_{\alpha \beta}$}
 \end{align*}
 It suffices to know $\partial_{j}g_{i\alpha}(\sigma)=0$ and $\partial_{i}g_{j\alpha}(\sigma)=0 $ to conclude $\partial_{\alpha}g_{ij}(\sigma)=0$.

\item 
$\partial_{\beta}g_{\alpha \alpha}(\sigma)=0 \Longrightarrow \partial_{\gamma}g_{\alpha \beta}(\sigma)=0$.\\
$\partial_{i}g_{\alpha \alpha}(\sigma)=0 \Longrightarrow \partial_{i}g_{\alpha \beta}(\sigma)=0$.

This is because
\begin{align*}
    \partial_{\gamma}g_{\alpha\beta}& =\frac{1}{2}(\partial_{\gamma}g_{\alpha+\beta, \alpha+\beta}-\partial_{\gamma}g_{\alpha \alpha}-\partial_{\gamma}g_{\beta \beta})\\
     \partial_{i}g_{\alpha\beta}&=\frac{1}{2}(\partial_{i}g_{\alpha+\beta, \alpha+\beta}-\partial_{i}g_{\alpha \alpha}-\partial_{i}g_{\beta \beta})
\end{align*}
\end{enumerate}

The remaining four cases left to prove are as follows.
 \begin{enumerate}[label=(\roman*)]
    \item $\partial_{\beta}g_{\alpha \alpha}(\sigma)=0$ ;
    \item $\partial_{i}g_{\alpha \alpha}(\sigma)=0$;
    \item $\partial_{j}g_{\alpha i}(\sigma)=0$;
    \item $\partial_{\beta}g_{\alpha i}(\sigma)=0$.
\end{enumerate}
We will have a general method to prove them. We first give a general formula for first derivatives of the pressure metric in the next subsection. The computation for the model case $\partial_{\beta}g_{\alpha \alpha}(\sigma)$ will be shown in Section 5 and Section 6. The other three cases will be discussed in Section 7.

\subsection{ First derivatives of the pressure metric}

This subsection is devoted to a formula of first derivatives of pressure metric. We also prove we have some freedom to choose representatives for variations of reparametrization functions from the Liv\v{s}ic cohomologous class.

Suppose $\{\rho(u,v,w)\}_{(u,v,w)\in\{(-1,1)\}^{3}}$ is an analytic three-parameter family of representations in the Hitchin component $\mathcal{H}_n(S)$ with base point $\rho(0,0,0)\in \mathcal{T}(S)$ corresponding to $X$. Suppose $\{f_{\rho(u,v,w)}\}_{(u,v,w)\in\{(-1,1)\}^{3}}$ are associated reparametrization functions. For simplicity of notation, we denote the renormalized reparametrization functions as
$$F(u,v,w)=f_{\rho(u,v,w)}^N=-h(\rho(u,v,w))f_{\rho(u,v,w)}$$

We also denote $F(0)=F(0,0,0)$ and  $\rho(0)=\rho(0,0,0)$. 

In the case of the Fuchsian representation, the topological entropy and the reparameterization function are simple. We have $h(\rho(0))=1$ (See  Theorem \ref{thm Entropy}). Since $\Phi_{\rho(0)}=\Phi$, the reparametrization function $f_{\rho(0)}$ can be chosen to be $1$ in the Liv\v{s}ic cohomologous class. Therefore one can choose $F(0)=-1$. 

The following characterization of the equilibrium measure for $F(0)$ is important.
\begin{lem}
The equilibrium state $m_{F(0)}$ for $F(0)$ is the Liouville measure $m_{L}$.
\end{lem}

\begin{proof}
Since the Liouville measure $m_{L}$ coincides with the Bowen-Margulis measure (Remark \ref{rem Liouville}). This follows easily from the variational principle (Proposition \ref{prop defn_pressure}).
\end{proof}
  
The measure $m_{L}$ is both $\Phi_t$ invariant and rotationally invariant on $UX$, i.e. $(e^{i\theta})^{*}m_{L}=m_{L}$ where $e^{i\theta}$ acts on $UX$ by usual multiplication. We will repeatedly use these important properties of the Liouville measure for our proofs later.

\begin{prop}
\label{prop formula_1st_dev}
The first derivatives of the pressure metric at $\rho(0)$ satisfy
\begin{align*}
&\partial_{w} \Bigg( {\langle \partial_{u}\rho(0,0,w), \partial_{v}\rho(0,0,w) \rangle}_{\Sub P}\Bigg) \Bigg|_{w=0}\\
&=\lim_{r\to \infty}\frac{1}{r}\Bigg(
  \int_{\Sub UX} \int_{0}^{r}\partial_{u}F(0)\mathrm{d}t \int_{0}^{r}\partial_{v}F(0)\mathrm{d}t \int_{0}^{r}\partial_{w}F(0)\mathrm{d}t\mathrm{d}m_{0}\\
  &+\int_{\Sub UX} \int_{0}^{r}\partial_{u}F(0)\mathrm{d}t\int_{0}^{r}\partial_{wv}F(0)\mathrm{d}t \mathrm{d}m_{0}
+\int_{\Sub UX}\int_{0}^{r}\partial_{v}F(0)\mathrm{d}t \int_{0}^{r}\partial_{wu}F(0)\mathrm{d}t\mathrm{d}m_{0}\Bigg),
\end{align*}
where the flow $\Phi_{t}(x)$ is omitted for simplicity.

\end{prop}

\begin{proof}

Starting from the Fuchsian point $\rho(0)$, along the ray parameterized by $\{(0,0,w)\}_{w\in (-1,1)}$, the pressure metric 
$
\langle \cdot,\cdot \rangle_{P}: T_{(0,0,w)}\mathcal{H}_{n}(S)\times T_{(0,0,w)}\mathcal{H}_{n}(S)\longrightarrow \mathbb{R} 
$
satisfies:
 \begin{align*}
  {\langle \partial_{u}\rho(0,0,w), \partial_{v}\rho(0,0,w) \rangle}_{\Sub P}
  &= 
  -\frac{\mathrm{Cov}(\partial_{u}F(0,0,w),\partial_{v}F(0,0,w), m_{F(0,0,w)})}{\int_{\Sub UX}F(0,0,w) \mathrm{d}m_{F(0,0,w)} }\\
  &=-\frac{ \partial_{v} \partial_{u} \mathbf{P}(F(0,0,w))-\int_{\Sub UX} \partial_{uv}F(0,0,w) \mathrm{d}m_{F(0,0,w)}}{\int_{\Sub UX}F(0,0,w) \mathrm{d}m_{F(0,0,w)}} \hspace{.2in}\text{equation (\ref{eq: Second dev of P})}
\end{align*}
We  first notice $\int_{\Sub UX}F(0) \mathrm{d}m_{0}=-1$ and from equation (\ref{eq: VS}),
\begin{equation*}
    \partial_{w}\bigg|_{w=0}(\int_{\Sub UX} F(0,0,w)\mathrm{d}m_{F(0,0,w)})= 
    Cov(F(0),\partial_{w}F(0),m_{0})+\int_{\Sub UX}\partial_{w}F(0)\mathrm{d}m_{0}=0
\end{equation*}

Therefore, 
 \begin{align*}
  &\partial_{w} \Bigg( {\langle \partial_{u}\rho(0,0,w), \partial_{v}\rho(0,0,w) \rangle}_{\Sub P}\Bigg) \Bigg|_{w=0}\\
  =&\partial_{w}\partial_{v}\partial_{u}\mathbf{P}(F(0))-\partial_{w}\int_{\Sub UX}\partial_{uv}F(0)\mathrm{d}m_{\Sub F(0,0,w)}\bigg|_{w=0}\\
  =&\partial_{w}\partial_{v}\partial_{u} \mathbf{P}(F(0))-Cov(\partial_{uv}F(0), \partial_{w}F(0), m_{0})-\int_{\Sub UX}\partial{uvw}F(0)\mathrm{d}m_{0}  \hspace{.25in}\text{by equation \eqref{eq: VS}}\\
 =&\lim_{r\to \infty}\frac{1}{r}\Bigg(
  \int_{\Sub UX} \int_{0}^{r}\partial_{u}F(0)\mathrm{d}t \int_{0}^{r}\partial_{v}F(0)\mathrm{d}t \int_{0}^{r}\partial_{w}F(0)\mathrm{d}t\mathrm{d}m_{0}\\
  +&\int_{\Sub UX} \int_{0}^{r}\partial_{u}F(0)\mathrm{d}t\int_{0}^{r}\partial_{vw}F(0)\mathrm{d}t \mathrm{d}m_{0}
+\int_{\Sub UX}\int_{0}^{r}\partial_{v}F(0)\mathrm{d}t \int_{0}^{r}\partial_{wu}F(0)\mathrm{d}t\mathrm{d}m_{0}\Bigg)\hspace{.25in}\text{by equation \eqref{eq:third dev of P}} 
\end{align*}
\end{proof}

\begin{prop}
The formula for the first derivatives of the pressure metric in Proposition \ref{prop formula_1st_dev} only depends on the Liv\v{s}ic class of each component function: $\partial_{u}F(0)$, $\partial_{v}F(0)$, $\partial_{w}F(0)$, $\partial_{wv}F(0)$, and $\partial_{wu}F(0)$. 
\label{prop: firstdev}
\end{prop}

\begin{proof}
We know from the proof of Proposition \ref{prop formula_1st_dev} that

\begin{align*}
&\partial_{w} \Bigg( {\langle \partial_{u}\rho(0,0,w), \partial_{v}\rho(0,0,w) \rangle}_{\Sub P}\Bigg) \Bigg|_{w=0}\\
=&\partial_{w}\partial_{v}\partial_{u} \mathbf{P}(F(0))-\int_{\Sub UX}\partial{uvw}F(0)\mathrm{d}m_{0}-Cov(\partial_{uv}F(0), \partial_{w}F(0), m_{0})
\end{align*}
 
By Remark \ref{eq: Second dev of P}, in general, if we take two mean zero H{\"o}lder functions $h_1$ and $h_2$ with respect to $m_0$, then

\begin{align*}
&Cov(h_1, h_2, m_{0})\\
=&\partial_{u}\partial_{v}\mathbf{P}(F(0)+uh_1+v h_2)\bigg|_{u,v=0}
\end{align*}

As the value of the pressure function $\mathbf{P}$ only depends on the Liv\v{s}ic class, we see changing $h_1$ and $h_2$ in its cohomology class does not change $Cov(h_1, h_2, m_{0})$. In particular, this holds for  $Cov(\partial_{uv}F(0),\partial_{w}F(0), m_{0})$.

Similarly, from equation \ref{eq:third dev of PII}, it is clear that

\begin{align*}
&\partial_{w}\partial_{v}\partial_{u} \mathbf{P}(F(0))-\int_{\Sub UX}\partial{uvw}F(0)\mathrm{d}m_{0}\\
=&\partial_{w}\partial_{v}\partial_{u} \mathbf{P}(F(0) + u\partial_{u}F(0) +v \partial_{v}F(0)+w\partial_{w} F(0)+uv\partial_{uv} F(0)+uw\partial_{uw} F(0)+vw\partial_{vw}F(0))\bigg|_{u,v,w=0}
\end{align*}

Again the above pressure function $\mathbf{P}$ does not change value if we change each component function. So together we know the first derivatives of the pressure metric only depends on the Liv\v{s}ic class of each component function: $\partial_{u}F(0)$, $\partial_{v}F(0)$, $\partial_{w}F(0)$, $\partial_{wv}F(0)$, and $\partial_{wu}F(0)$. 
\end{proof}

\subsection{A gauge theoretical formula}
\label{subsection gauge}

In \cite{Variation_along_FuchsianLocus}, Labourie and Wentworth show the variation of reparametrization functions can be expressed by a gauge-theoretical formula. This formula will be crucial for our computation in the next section. We include the formula and its proof here for completeness. We add some assumptions which are natural for our case of Hitchin components $\mathcal{H}_{n}(S)$.

We consider $(E,H)$ a rank $n$ Hermitian bundle over the surface $S$ equipped with a Riemannian metric $g$. We let $\gamma$ be a closed curve on $S$ with arc length parametrization $\gamma(t)$. 
Suppose $D_{A^{0}}$ is a flat connection on $E$ so that the holonomy of it has distinct eigenvalues along $\gamma$. Suppose $\lambda_{\gamma}$ is one eigenvalue with a corresponding eigenline $\mathcal{L}_{\gamma}$ and $\mathcal{H}_{\gamma}$ is the complementary hyperplane stablized by the holonomy. We denote by $\mathcal{L}_{\gamma}(t)$ the line generated by  the parallel transports of $\mathcal{L}_{\gamma}$ along $\gamma$ at time t, by $\mathcal{H}_{\gamma}(t)$ the hyperplane generated by complementary eigenvectors and by $\pi(t)$ the projection on $\mathcal{L}_{\gamma}(t)$ along $\mathcal{H}_{\gamma}(t)$. Then we have

\begin{prop} (Labourie-Wentworth, \cite{Variation_along_FuchsianLocus} )
 \label{thm:Tr}

For $D_{\Sub A^{s}}$  a smooth one parameter family of flat connections, we have a unique smooth function $\lambda_{\gamma}(s)$ so that for $s$ small enough, $\lambda_{\gamma}(s)$ is the eigenvalue of the holonomy of $D_{A^{s}}$ with $\lambda_{\gamma}(0)=\lambda_{\gamma}$. Moreover,

\begin{equation}
 \frac{d\log \lambda_{\gamma}(s)}{ds}\bigg|_{s=0}= -\int_{0}^{l_{\gamma}} Tr(\partial_{s}{D}_{\Sub A^0}(t) \cdot \pi(t))dt, 
 \label{eq:tr}
\end{equation}
 Here the notation is $\partial_{s}{D}_{A^0}(t):=\partial_s  D_{A^{s}}(\dot{\gamma}(t)\frac{\partial}{\partial t})|_{s=0}$, where $\partial_{s}{D}_{A^0}$ is a $End(E)$-valued $1$-form and $\dot{\gamma}(t)\frac{\partial}{\partial t}$ is the tangent vector field along $\gamma(t)$. 
\end{prop}

\begin{proof}
We prove equation \ref{eq:tr} here.

Let $\{g_{s}\}$ be a family of gauge transformation acting on $\{D_{A^{s}}\}$ with $g_{0}=id$. Denote the new connection 1-forms ${\tilde{A}}^{s}:=g_s^*A^{s}$.
We first prove:
\begin{equation*}
   \int_{0}^{l_{\gamma}}Tr(\partial_{s}{D}_{\Sub A^{0}}(t) \cdot \pi(t))dt=\int_{0}^{l_{\gamma}}Tr(\partial_{s}{D}_{\tilde{A^{0}}}(t) \cdot \pi(t))dt,  
\end{equation*}

Note here $\partial_{s}{D}_{A^0}(t)$ is a $0$-form since we have contracted the $1$-form $\partial_{s}{D}_{A^s}|_{s=0}$ with the tangential vector field. Therefore $Tr(\partial_{s}{D}_{\Sub A^{0}}(t)\cdot \pi(t))$ is a function in $t$ or in $\dot{\gamma}(t)$.

Taking a derivative of ${\tilde{A}}^{s}:=g_s^*A^{s}$ at $s=0$ yields

\begin{equation*}
 \partial_{s}{D}_{\tilde{A^{0}}}=\partial_{s}{D}_{A^0}+D_{A^{0}} \dot{g}
\end{equation*}
    
where $\dot{g}$ denoting $\frac{\partial g_s}{\partial_s}\bigg|_{s=0}$ is a section of $End(E)$ and the connection $D_{A^{0}}$ acts on $\dot{g}$ as $D_{A^{0}} \dot{g}=d\dot{g}+[A^{0},\dot{g}]$.

We want to show
\begin{equation*}
    \int_{0}^{l_{\gamma}}Tr((D_{A^0}\dot{g})\pi)dt=0
\end{equation*}

To simplify the notation, we will always omit the variable $t$ when writing our formulas. For example, here $(D_{A^0}\dot{g})\pi:=(D_{A^0}(t)\dot{g}(t))\pi(t)$.

We start from proving that $\pi$ is a $D_{A^{0}}$-parallel section in $End(E)$. Given any section $v\in \Gamma(E)$, we can write it as a linear combination of eigenvectors of holonomy. Set
$
v(t)=\sum_{i=1}^{n}a_{i}(t)e_{i}(t) 
$
where $e_{i}(t)$ satisfies parallel transport equation $D_{A_{0}}e_{i}=0$ with boundary conditions
   $  e_{i}(l_{\gamma})=\lambda^{i}_{\gamma}e_{i}(0), 
    ||e_{i}(0)||=1 
$. In particular, we assume $\lambda^{1}_{\gamma}=\lambda_{\gamma}$ and $\mathcal{L}_{\gamma}(t)$ is generated by $e_1(t)$.
 Then  
\begin{align*}
    (D_{A^{0}} \pi)(v)&=[D_{A^{0}},\pi]v\\
    &=D_{A^{0}}(\pi v)-\pi(D_{A^{0}}v)\\
    &=D_{A^{0}}(a_{1}(t)e_{1}(t))-\pi(\sum_{i=1}^{n}(da_{i}(t)e_{i}(t)+a_{i}(t)D_{A^{0}}e_{i}(t)))\\
    &=da_{1}(t)e_1(t)-da_{1}(t)e_1(t)\\
    &=0
\end{align*}
Thus 
\begin{align*}
    \int_{0}^{l_{\gamma}}\frac{d}{d t}(Tr(\dot{g}\cdot \pi))dt&=
    \int_{0}^{l_{\gamma}}Tr(\frac{\partial}{\partial t}(\dot{g}\pi))dt\\
    &=  \int_{0}^{l_{\gamma}}Tr(D_{A^0}(\dot{g}\pi))dt  \hspace{.5in}\text{since $Tr([A^{0},\dot{g}\pi])=0$}\\
    &= \int_{0}^{l_{\gamma}}Tr([D_{A^0}, \dot{g}\pi])dt  \hspace{.5in}\text{notice $\dot{g}\pi\in \Gamma(End(E))$}\\
    &= \int_{0}^{l_{\gamma}}Tr([D_{A^0},\dot{g}]\pi+\dot{g}[D_{A^{0}},\pi])dt \\
    &= \int_{0}^{l_{\gamma}}Tr((D_{A^0}\dot{g})\pi+\dot{g}(D_{A^{0}}\pi))dt \hspace{.5in}\text{action of a connection on $\Gamma(End(E))$}\\
    &=\int_{0}^{l_{\gamma}}Tr((D_{A^0}\dot{g})\pi)dt \hspace{.5in}\text{Since $D_{\Sub A^{0}}\pi =0$}\\
\end{align*}
 So
 \begin{align*}
\int_{0}^{l_{\gamma}}Tr((D_{A^0}\dot{g})\pi)dt
&=\int_{0}^{l_{\gamma}}\frac{d}{dt}(Tr(\dot{g}\cdot\pi))dt\\
&=Tr(\dot{g}(l_{\gamma})\pi(l_{\gamma}))-Tr(\dot{g}(0)\pi(0))=0
\end{align*}

As $s$ varies, the eigenline $\mathcal{L}_{\gamma}^{s}(t)$ corresponding to $\lambda_{\gamma}(s)$ varies according to s and so is the complementary hyperplane $\mathcal{H}_{\gamma}^{s}(t)$ . By picking suitable gauges $\{g_{s}\}$, we can assume, for ${\tilde{A}}^{s}:=g_s^*A^{s}$, the eigenlines $\tilde{\mathcal{L}}_{\gamma}^{s}(t)$ and complementary hyperplanes $\tilde{\mathcal{H}}_{\gamma}^{s}(t)$ satisfy  $\tilde{\mathcal{L}}_{\gamma}^{s}(t)=\mathcal{L}_{\gamma}(t)$ and $\tilde{\mathcal{H}}_{\gamma}^{s}(t)=\mathcal{H}_{\gamma}(t)$.

Without lose of generality, we assume $D_{A^{s}}$ is itself the connection after suitable gauge and the set $\{e_{i}^{s}\}$ are eigenvectors for $A^{s}$ with $e_{1}^{s}$ corresponding to $\mathcal{L}_{\gamma}^{s}$ . Thus we have the following equations.

\begin{equation*}
    \begin{cases}
      D_{A^{s}}e_{i}^{s}(t)=0\\
      e_{i}^{s}(l_{\gamma})=\lambda_{\gamma}^{i}(s)e_{i}^{s}(0)
    \end{cases}
  \end{equation*}
  
In particular, we can assume
\begin{equation*}
    \begin{cases}
      D_{A^{s}}e_{1}^{s}(t)=0, \hspace{.25in} e_{1}^{s}(l_{\gamma})=\lambda_{\gamma}^{1}(s)e_{1}^{s}(0)\\
      e_{1}^{s}(t)=c_{s}(t)e_{1}^{0}(t), \hspace{.25in}  e_{1}^{s}(0)=e_{1}^{0}(0)
    \end{cases}
  \end{equation*}
So \begin{equation*}
    e_{1}^{s}(l_{\gamma})=c_{s}(l_{\gamma})e_{1}^{0}(l_{\gamma})=c_{s}(l_{\gamma})\lambda_{\gamma}^{1}(0)e_{1}^{0}(0)=\lambda_{\gamma}^{1}(s)e_{1}^{s}(0)=\lambda_{\gamma}^{1}(s)e_{1}^{0}(0)
\end{equation*}
and thus
$c_{s}(l_{\gamma})= \frac{\lambda_{\gamma}^{1}(s)}{\lambda_{\gamma}^{1}(0)}$
and $c_{0}(l_{\gamma})=1$.
Notice 
\begin{equation*}
\frac{H(e_{1}^{0}(t), D_{A^{s}}e_{1}^{0}(t))}{H(e_{1}^{0}(t), e_{1}^{0}(t))}
=\frac{H(e_{1}^{0}(t), D_{A^{s}}(\frac{e_{1}^{s}(t)}{c_{s}(t)}))}{H(e_{1}^{0}(t), \frac{e_{1}^{s}(t)}{c_{s}(t)})}
=\frac{\partial_{t}(\frac{1}{c_{s}(t)})}{\frac{1}{c_{s}(t)}}
=-\frac{\partial(\log c_{s}(t))}{\partial t}
\end{equation*}
So 
\begin{equation*}
\int_{0}^{l_{\gamma}} Tr(\partial_{s}{D}_{A^{0}}\pi)dt=\int_{0}^{l_{\gamma}}\frac{H(e_{1}^{0}(t), \partial_{s}{D}_{A^{0}}e_{1}^{0}(t))}{H(e_{1}^{0}(t), e_{1}^{0}(t))}dt = -\int_{0}^{l_{\gamma}}\frac{\partial}{\partial s}\bigg|_{s=0} \left(\frac{\partial(\log c_{s}(t))}{\partial t}\right)dt=-\frac{d\log \lambda_{\gamma}^{1}(s)}{ds}\bigg|_{s=0}
\end{equation*}

\end{proof}

\section{Computation of variation of reparametrization functions for a model case}

In this and the next sections, we consider the model case $\partial_{\beta}g_{\alpha \alpha}(\sigma)$. Note the treatment of this case will involve all the steps needed for the other cases. This justifies the expositional decision that we consider it here first and in isolation.

In this case, we are given parameters $(u,v)\in \{(-1,1)\}^{2}$ with (conjugacies classes of) representations \{$\rho(u,v)\}$ in $\mathcal{H}_{3}(S)$ corresponding to $\{(0,uq_{\alpha}+vq_{\beta})\}\subset H^{0}(X,K^2)\bigoplus H^{0}(X,K^3)$ by Hitchin parametrization (see Remark \ref{rem HitchinPara}). In particular, at the Fuchsian point $\rho(0)=X$, we identify $\partial_{u}\rho(0,0)$ with $\varphi(q_{\alpha})$ and $\partial_{v}\rho(0,0)$ with $\varphi(q_{\beta})$, where $\varphi(q_{\alpha})$ and  $\varphi(q_{\beta})$ are the Hitchin deformation given in Definition \ref{defn HitchinDeformation}. We suppose $\{f_{\rho(u,v)}\}$ is an associated two-parameter family of reparametrization functions. 
By Proposition \ref{prop formula_1st_dev}, the formula for $\partial_{\beta}g_{\alpha \alpha}(\sigma)$ is
\begin{align*}
    \partial_{\beta}g_{\alpha \alpha}(\sigma)&=\partial_{v} \Bigg( {\langle \partial_{u}\rho(0,v), \partial_{u}\rho(0,v) \rangle}_{\Sub P}\Bigg) \Bigg|_{v=0}\\
    &=\lim\limits_{r\to \infty}\frac{1}{r}\left[\int_{\Sub UX}\left(\int_{0}^{r} \partial_{u}f_{\rho(0)}^{N}\mathrm{d}t \right)^2\int_{0}^{r}\partial_{v}f_{\rho(0)}^{N}\mathrm{d}t\mathrm{d}m_{0}+
    2\int_{\Sub UX} \int_{0}^{r}\partial_{u}f_{\rho(0)}^{N} \mathrm{d}t\int_{0}^{r} \partial_{uv}f_{\rho(0)}^{N}\mathrm{d}t\mathrm{d}m_{0}\right].
\end{align*}

Because $\partial_{u}h(\rho(u,0))=\partial_{v}h(\rho(0,v))=0$ on Fuchsian locus $\mathcal{T}(S)$. By Theorem \ref{thm Entropy}, the variations of reparametrization functions that need to be computed are the following:

\begin{enumerate}[label=(\roman*)]
    \item $\partial_{u}f_{\rho(0)}^{N}=-\partial_{u}f_{\rho(0)}$;
    \item $\partial_{v}f_{\rho(0)}^{N}=-\partial_{v}f_{\rho(0)}$;
    \item $\partial_{uv}f_{\rho(0)}^{N}=-\partial_{uv}h(\rho(0))-\partial_{uv}f_{\rho(0)}$.
\end{enumerate}

Before proceeding to compute (i), (ii) and (iii), we explain our general strategies to compute variations of reparametrization functions. Our computation will be based on Proposition \ref{thm:Tr} and tools from Higgs bundles theory. Let us first set up our Higgs bundles. 

In the component $\mathcal{H}_{3}(S)$ we are considering, the rank-$3$ holomorphic vector bundle is fixed as $E=K \bigoplus\mathcal{O}\bigoplus K^{-1}$. Associated to a representation $\rho$ in $\mathcal{H}_{3}(S)$ is a Hermitian metric $H$ on $E$ that solves Hitchin's equation (\ref{eq Hitchin'sEquation}) and a flat connection $D_{H}=\nabla_{{\bar{\partial}}_{E},H} + \Phi+ {\Phi}^{*H}$ where $\nabla_{{\bar{\partial}}_{E},H}$ is the chern connection (see Theorem \ref{thmChernConnection}). 

Given a parameter $s\in(-1,1)$, suppose we are considering a family of conjugacy classes of representations $\{\rho_{s} \}$ in $\mathcal{H}_{3}(S)$. On the one hand, there is a family of flat connections $\{D_{\Sub H(s)}\}$ given by equation (\ref{eq flatD}) assiociated to $\{\rho_{s} \}$. On the other hand, there is a family of reparametrization functions $\{f_{\rho_{s}}\}_{s\in(-1,1)}$  assiociated to $\{\rho_{s} \}$ from the thermodynamical point of view. Recall our notation (\ref{eq notation1}), (\ref{eq notation2}). For a family of flat connection $\{D_{H(s)}\}$, we denote
 \begin{align*}
     \partial_{s}D_{H(0)}=\frac{\partial}{\partial_s}\bigg|_{s=0}D_{H(s)}.
 \end{align*}
 and for a family of reparametrization functions $\{f_{\rho_{s}}\}$, we denote
 \begin{align*}
     \partial_{s}f_{\rho_{0}}=\frac{\partial}{\partial_s}\bigg|_{s=0}f_{\rho_{s}}.
 \end{align*}

By Proposition \ref{thm:Tr} and Liv\v{s}ic's Theorem, we know the H{\"o}lder function $-\Tr(\partial_{s}D_{H(0)}\pi)(x)$ and $\partial_{s}f_{\rho_{0}}(x)$ are in the same Liv\v{s}ic cohomology class. Recall our notation in Definition \ref{defn Livsic}, 
\begin{align}
 \partial_{s}f_{\rho_{0}}(x)\sim -\Tr(\partial_{s}D_{H(0)}\pi)(x).
    \label{eq firstVar}
\end{align}

Here we define $\Tr(\partial_{s}D_{H(0)}\pi)(\Phi_t(x)):=\Tr(\partial_{s}D_{H(0)}(t)\pi(t))$ following Proposition  \ref{thm:Tr}. The curve $\gamma(t)$ in Proposition  \ref{thm:Tr} from now on will be a unit speed geodesic starting from $x$. Therefore, we have $x=\dot{\gamma}(0)\frac{\partial}{\partial t}$ and $\Phi_t(x)=\dot{\gamma}(t)\frac{\partial}{\partial t}$.

Proposition \ref{prop: firstdev} allows us to consider first and second variations of reparametrization functions in terms of Liv\v{s}ic cohomology class instead of individual functions. From now on, for first and second variations of reparametrization functions, we will no longer distinguish cohomologous elements.

Because $X$ is a hyperbolic surface and the geodesic flow is Anosov. The vectors tangent to periodic geodesics are dense in $TX$. To recover the information of $\partial_{s}f_{\rho_{0}}$, it suffices to compute $\Tr(\partial_{s}D_{H(0)}\pi)$ on each closed geodesic. Similarly, to compute the second variations of reparametrization functions, it suffices to compute them on each closed geodesic.

Now we start to give a complete computation of the first and second variations of reparametrization functions for the case $\partial_{\beta}g_{\alpha \alpha}(\sigma)$. The steps of our argument are divided into different subsections:
\begin{enumerate}
    \item
          We set up coordinates adapted to the closed geodesics we study and conclude special properties of affine metrics with respect to chosen coordinates on these geodesics.
    \item
         We first construct a homogeneous ODE arising from the parallel transport equation for the base flat connection at $\rho(0)=\sigma \in \mathcal{T}(S)$. This leads formulas fro  first variation of reparametrization functions proved in \cite{Variation_along_FuchsianLocus}.
    \item 
         We consider a family of parallel transport equations associated to a family of flat connections by solving Hitchin's equations based at $\rho(0)=\sigma \in \mathcal{T}(S)$. The variation of this family of parallel transport equations at $\sigma$ gives rise to some nonhomogeneous ODEs and yields solutions to second variations of reparametrization functions on the closed geodesics we consider. 
    \item
        We extend our computation from the closed geodesics to the surface.
\end{enumerate}

\subsection{Setting up coordinates on surfaces}

In this subsection, we set up coordinates adapted to the closed geodesics we study. We will obtain some important properties for the affine metric after setting up the coordinates. They can be used in computation of first and second variations of reparametrization functions in the following sections. The first variations have been computed in \cite{Variation_along_FuchsianLocus} by advanced Lie theoretic methods.

The convention we use for a Hermitian metric $H$ on $E$ is: it is $\mathbb{C}$-linear in the second variable and conjugate-linear in the first variable.
Suppose on a coordinate chart $(U,z)$, the bundle $E=K \bigoplus\mathcal{O}\bigoplus K^{-1}$ is trivialized as $E|_{U} \cong U \times \mathbb{C}^3$. Locally we have a holomorphic frame $(s_1,s_2,s_3)$ on $U$. With respect to the local holomorphic frame and our convention of the Hermitian metric, the $(1,0)$-part of the Chern connection $\nabla_{{\bar{\partial}}_{E},H}$ is $H^{-1}\partial H$. The Hermitian conjugate is $\Phi^{*H}=H^{-1}\bar{\Phi}^{t}H$. The connection one form $A$ of the flat connection $D_{H}$ is thus
\begin{align*}
    A= H^{-1}\partial H + \Phi+ \Phi^{*H}.
\end{align*}
Associated to representations $\{\rho(u,v)\}$ are a two-parameter family of flat connections $\{D_{H(u,v)}\}$. We will study their connection one-forms in holomorphic frames with respect to some carefully chosen coordinates on the surface $X$.
 
 When the Higgs field is 
 \[
  \Phi(u,v)=
\begin{bmatrix}
   0 & 0 & uq_{\alpha}+vq_{\beta} \\
   1 & 0 & 0 \\
   0 & 1 & 0 
  \end{bmatrix},
\]
 Baraglia proves the Hermitian metric $H(u,v)$ that solves Hitchin's equation (\ref{eq Hitchin'sEquation}) is diagonal (see \cite{Baraglia_thesis}). Following Baraglia's notation in $\cite{Baraglia_thesis}$, we denote the Hermitian metric as $H(u,v)=e^{2\Omega(u,v)}$. We have
 
 \[
  H(u,v)=
  \begin{bmatrix}
   h(u,v)^{-1} & 0 & 0 \\
   0 & 1 & 0 \\
   0 & 0 & h(u,v) 
  \end{bmatrix},
\]
where $h=h(u,v)$ is a section of $\bar{K}\otimes K$ and 

\[
  \Omega(u,v)=
  \begin{bmatrix}
   -\omega(u,v) & 0 & 0 \\
   0 & 0 & 0 \\
   0 & 0 & \omega(u,v) \\
  \end{bmatrix}
\]

with $\omega(u,v)=\frac{1}{2}\log h(u,v)$.

 We denote the corresponding flat connection by $D_{H(u,v)}=\nabla_{{\bar{\partial}}_{E},H(u,v)}+\Phi(u,v) + \Phi(u,v)^{*H(u,v)}$. The connection one-form $A(u,v)\in \Gamma(T^{*}X \otimes EndE)$ is thus
 
\begin{equation}
  A(u,v)=
  \left[ {\begin{array}{ccc}
   -2\partial \omega(u,v) & h(u,v) & uq_{\alpha}+vq_{\beta} \\
   1 & 0 & h(u,v) \\
   h^{-2}(u\bar{q}_{\alpha}+v\bar{q}_{\beta}) & 1 & 2\partial \omega(u,v) \\
  \end{array} } \right]
\label{eq:connection}
\end{equation}

In fact $2h(u,v)$ is an affine metric for some hyperbolic affine sphere in the conformal class of $\sigma$ (see \cite{Baraglia_thesis}).

We denote $\phi=\log \frac{2h}{\sigma}$. Note $\phi=\phi(u,v,z)$ is actually a globally well-defined function on $X$ that does not depend on coordinate systems. The Hitchin's equation  (\ref{eq Hitchin'sEquation}), also the integrability condition for affine sphere (see \cite{Loftin_SurveyOnAffine}) can be written as a equation of $\phi$

\begin{equation}
 \Delta_{\sigma}\phi + 16\norm{uq_{\alpha}+vq_{\beta}}^2_{\sigma}e^{-2\phi} -2e^{\phi}+2=0
\label{eq:affinePDE1Global}
 \end{equation}
where $\norm{\cdot}_{\sigma}$ is the induced norm on cubic differentials. It satisfies $\norm{q}_{\sigma}^2=\frac{|q|^2}{{\sigma^3}}$. The notation we adopt for Laplacian is $\Delta_{\sigma}=\frac{  4\partial_{\bar{z}}\partial_{z}}{\sigma}$.

For simplicity of notation, we sometimes omit variables and write $\phi$ as $\phi(u,v)$ or $\phi(z)$ depending on our needs.

We have the following observation from equation (\ref{eq:affinePDE1Global}):
\begin{itemize}
    \item 
    When $(u,v)=(0,0)$, the only solution of the equation (\ref{eq:affinePDE1Global}) is $\phi=\phi(0,0)=0$. The affine metric $2h=\sigma$ is indeed the hyperbolic metric of constant curvature $-1$.
    \item
    Taking $u$-derivative or $v$-derivative of equation (\ref{eq:affinePDE1Global}) at $(u,v)=(0,0)$ yields
    \begin{align}
        \Delta_{\sigma}\phi_{u}-2e^{\phi}\phi_{u}&=0, \label{eq:affinePDE1GlobalDev_u}\\
        \Delta_{\sigma}\phi_{v}-2e^{\phi}\phi_{v}&=0.\label{eq:affinePDE1GlobalDev_v}
    \end{align}
    Therefore, at $(u,v)=(0,0)$, the fact that $\phi=\phi(0,0)=0$ implies $\phi_{u}=\phi_{u}(0,0)=0$ and $\phi_{v}=\phi_{v}(0,0)=0$. 
\end{itemize}
We now choose a special coordinate system that facilitates the study of holonomy problems on a closed geodesic.  Let $z$ be a local holomorphic coordinate on $X$. Suppose the affine metric in this coordinate is $e^{\psi(u,v,z)}|dz|^2$ and the hyperbolic metric in this coordinate is $\sigma=e^{\delta(z)}|dz|^{2}$. Suppose $\gamma(t)$ is any closed geodesic with respect to the hyperbolic metric $\sigma$ on the Riemann surface $X$. Then written in the $z$-coordinate, it is

\begin{equation*}
    \gamma(t)=z(t)=\Re\gamma(t)+ i \Im\gamma(t)
\end{equation*}
and
\begin{equation*}
    \dot{\gamma}(t)\frac{\partial}{\partial t}= (\Re\dot{\gamma}(t)+ i \Im\dot{\gamma}(t))\frac{\partial}{\partial z}+ (\Re\dot{\gamma}(t)- i \Im\dot{\gamma}(t))\frac{\partial}{\partial \bar{z}}.
\end{equation*}

In particular, we can model $\gamma(t)$ on a strip $S=\{x+iy| \text{ } |y|< \frac{\pi}{2} \}$ with the hyperbolic metric $ds=\frac{|dz|}{\cos y}$ and $\gamma(t)=(t,0)$. This coordinate around $\gamma$ is called a Fermi coordinate and satisfies $\Re\dot{\gamma}(t)=1$, $\Im\dot{\gamma}(t)=0$. Thus it's easy to check on $\gamma$ one has $\gamma^{*}ds=|dz|$ and $\delta(z)=0$.

The variable $t$ is then the arc length parameter for our choice of coordinates. Therefore if one denotes $\dot{\gamma}(0)\frac{\partial}{\partial t}=x \in UX$, then $\dot{\gamma}(t)\frac{\partial}{\partial t}= \Phi_{t}(x)$. We will always assume $\dot{\gamma}(0)\frac{\partial}{\partial t}=x$ in our discussion.

With the Fermi coordinate understood, from the fact that the only solution of equation (\ref{eq:affinePDE1Global}) is $\phi=0$, we conclude
\begin{align*}
     \psi(z)=\phi(z)+\delta(z)=\delta(z)=0
\end{align*}
From equation \ref{eq:affinePDE1GlobalDev_u} together with equation \ref{eq:affinePDE1GlobalDev_v}  and their solutions $\phi_u=\phi_v=0$, we obtain
\begin{align*}
   &\psi_{u}(z)=\phi_{u}(z)=0, \\
  &\psi_{v}(z)=\phi_{v}(z)=0. 
\end{align*}

Also $\psi(z)=0$ implies
\begin{align*}
\psi_{z}(z)=\delta_{z}(z)=0
\end{align*}

All these information about the affine metric $\psi$ with respect to the Fermi coordinate will be important in computation in later sections.

\subsection{Homogeneous ODEs for holonomy and first variations of reparametrization functions}

In this subsection, we show formula of first variations of reparametrization functions from \cite{Variation_along_FuchsianLocus}. We also construct homogeneous ODEs arising from the parallel transport equations for the base flat connection at $\sigma \in \mathcal{T}(S)$. These serve as the first step for the computation of second variations in later subsections.

We first explain our notation. For $q_i=q_i(z)dz^2$ any quadratic differential and $q_{\alpha}=q_{\alpha}(z)dz^3$ any cubic differential, we also use $q_{i}$ and $q_{\alpha}$ to denote H{\"o}lder functions on unit tangent bundle $UX$ as follows. We let $q_i:UX \to \mathbb{C}$ and $q_{\alpha}: UX \to \mathbb{C}$ be 
\begin{equation}
    q_i(x):=q_i(z)dz^2(x,x)=q_i(z)(dz(x))^2
    \label{eq quadratic}
\end{equation}
\begin{equation}
    q_{\alpha}(x):=q_{\alpha}(z)dz^3(x,x,x)=q_{\alpha}(z)(dz(x))^3\\ \label{eq cubic}
\end{equation}

First variations of reparametrization functions for our cases have been computed in \cite{Variation_along_FuchsianLocus} as follows.
\begin{prop}\cite[Theorem 4.0.2]{Variation_along_FuchsianLocus} \hfill

\label{prop case1rep1}
The first variations of reparametrization functions $\partial_{u}f_{\rho(0)}:UX \to \mathbb{R}$ and $\partial_{v}f_{\rho(0)}:UX \to \mathbb{R}$ for our model case $\partial_{\beta}g_{\alpha \alpha}(\sigma)$ satisfy 
 \begin{align*}
        &-\partial_{u}f_{\rho(0)}(x) \sim \Re q_{\alpha}(x),\\ &-\partial_{v}f_{\rho(0)}(x) \sim  \Re q_{\beta}(x).
\end{align*}
where the notation $\sim$ denote Liv\v{s}ic equivalance (Definition \ref{defn Livsic}).
\end{prop}

We refer the reader to \cite{Variation_along_FuchsianLocus} for a proof of Proposition \ref{prop case1rep1}. It is an implement of the formula (\ref{eq firstVar}).

We then study parallel transport equations for the connection $D_{H(0)}$ arising from holonomy problems based at $\rho(0)\in\mathcal{T}(S)$. With the coordinates introduced in last section, they become homogenous ODE systems that are easy to solve. We list some important computations involved here. These will be important for second variation of reparametrzation functions. 

The parallel transport equation for the connection $D_{H(0)}$ on the closed geodesic $\gamma$ is as follows,
\begin{equation}
    D_{H(0),\dot{\gamma}}V=0,
    \label{eq:parallel}
\end{equation}
where $V\in \Gamma(E)$ is a parallel section with boundary conditions:
\begin{equation}
    V(l_{\gamma})= \lambda_{i}(\gamma,\rho(0)) V(0). \label{eq:Boundary}
\end{equation}
Here $\lambda_{i}(\gamma,\rho(0))$ is one of the eigenvalues for holonomy of $D_{H(0)}$ on $\gamma$ for $i=1,2,3$.
We want to write the equation (\ref{eq:parallel}) on a specific holomorphic frame which can be constructed as follows.

We cover $\gamma$ by $m$ charts $\{(U_i,z_i)\}_{i=1}^{m}$ so that $z_{i}: U_{i}\to z_{i}(U_{i})\subset\mathbb{C}$ is a diffeomorphism for $1\leq i \leq m$. We assume our holomorphic bundle $E$ is trivialized on each $U_{i}$. Furthermore we assume the transition map on every overlap is either the identity or a hyperbolic translation viewed on the universal cover $\mathbb{D}$. Since $dz_{i}$ is a local holomorphic section of $K$ on $U_{i}$ and $\frac{\partial}{\partial{z_{i}}}$ is a local holomorphic section of $K^{-1}$ on $U_{i}$, we can define a local holomorphic frame $s^i=(s_1^i,s_2^i,s_3^i)$ for $E=K \bigoplus\mathcal{O}\bigoplus K^{-1}$ on $U_i$, where $s_1^{i}=dz_{i}$ and $s_2^{i}=1$ and $s_3^{i}=\frac{\partial}{\partial z_{i}}$. Setting $(U_{m+1},z_{m+1})=(U_{1},z_{1})$ and $s_{j}^{m+1}=s_{j}^{1}$, this yields a well-defined holomorphic frame for $\gamma$. Because on each overlap and for $j=1,2,3$, we have $s_{j}^{i}=s_{j}^{i+1}$  on $\gamma|_{U_i} \cap \gamma|_{U_i+1}$ with $1\leq i \leq m$.

We will simply write the holomorphic frame on $\gamma$ as $s_j$ for $j=1,2,3$.
With respect to this frame given by $(s_1, s_2, s_3)$, the parallel transport equation for $V(t)=\sum\limits_{i=1}^{3} V_{i}(t)s_{i}(t)$ becomes
\begin{align*}
  \partial_{t}\left[ \begin{array}{c} V^1(t) \\ V^2(t) \\ V^3(t) \end{array} \right] + \begin{bmatrix} 0 & \frac{1}{2} & 0\\ 1 & 0 & \frac{1}{2} \\ 0 & 1 &0 \end{bmatrix}  \left[ \begin{array}{c} V^{1}(t) \\ V^2(t) \\ V^3(t) \end{array} \right]=0.
\end{align*}

There are three eigenvalues for this ordinary differential equation system: $\lambda_{1}(\gamma,\rho(0))=e^{l_\gamma}$, $\lambda_{2}(\gamma,\rho(0))=1$ and  $\lambda_{3}(\gamma,\rho(0))= e^{-l_{\gamma}}$. The solutions for $V$ (assuming norm 1 at the starting point with respect to the Hermitian metric $H(0)$), denoted as $e_{i}$ corresponding to $\lambda_{i}(\gamma)$ for $i=1,2,3$, are

\begin{equation*}
    e_{1}=\frac{\sqrt{2}}{2}e^{t}\left[\begin{array}{c} \frac{1}{2}\\ -1 \\ 1 \end{array}  \right], \hspace{.25in}
    e_{2}=\frac{1}{2} \left[\begin{array}{c} -1\\ 0 \\ 2 \end{array}  \right],\hspace{.25in}
     e_{3}=\frac{\sqrt{2}}{2}e^{-t}\left[\begin{array}{c} \frac{1}{2}\\ 1 \\ 1 \end{array}  \right].
\end{equation*} 

We note at the Fuchsian point $\rho(0)\in\mathcal{T}(S)$, the eigenvectors $e_{1}, e_{2}$ and $e_{3}$ are orthogonal. In our holomorphic frame, the projection $\pi(0)=\pi(\rho(0))$ can be computed as 

\[
  \pi(0)=\frac{1}{2}
  \left[ {\begin{array}{ccc}
   \dfrac{1}{2} & -\dfrac{1}{2} & \dfrac{1}{4} \\[1ex]
   -1 & 1 & -\dfrac{1}{2} \\[1ex]
   1 & -1 & \dfrac{1}{2} \\[1ex]
  \end{array} } \right].
\]

The eigenvections $e_i$ and projection $\pi$ will play important roles in later sections.

\subsection{Inhomogenous ODEs and the second variation of the reparametrization functions}

We wil compute the second variation of the reparametrization functions $\partial_{uv}f_{\rho(0)}$ in this and next subsection. With our formula (\ref{eq firstVar}), we have
\begin{align*}
 \partial_{uv}f_{\rho(0)}& \sim -\partial_{v}(\Tr(\partial_{u}D_{H(0,v)} \pi(0,v)))|_{v=0}\\
 &=-\Tr(\frac{\partial^2 D_{H(0)}}{\partial u \partial v}\pi(0)) - \Tr(\partial_{u}D_{H(0)}\partial_{v}\pi(0)).
\end{align*}
 In this subsection, we compute $ \partial_{uv}f_{\rho(0)}$ along a closed geodesic. We study variation of holomony problems along a closed geodesic and construct associated inhomogeneous ODEs. In the next subsection, we extend the computation of $ \partial_{uv}f_{\rho(0)}$ to the whole surface.

The computation of $\partial_{v}(\Tr(\partial_{u}D_{H(0,v)} \pi(0,v)))|_{v=0}$ in this subsection along a closed geodesic is divided into computations of $\Tr(\frac{\partial^2 D_{H(0)}}{\partial u \partial v}\pi(0))$ and of $\Tr(\partial_{u}D_{H(0)}\partial_{v}\pi(0))$:

\begin{itemize}
    \item  Compute $\Tr(\frac{\partial^2 D_{H(0)}}{\partial u \partial v}\pi(0))$.

With the holomorphic frames and Fermi coordinates set up as before, one obtains on $\gamma$
\begin{align*}
  \partial_{uv}D_{H(0)}(x)=
   \begin{bmatrix}
   -(\psi_{z})_{uv}(z) & \frac{1}{2}\psi_{uv}(z) & 0\\[1ex]
   0 & 0 & \frac{1}{2}\psi_{uv}(z) \\[1ex]
   0 & 0 & (\psi_{z})_{uv}(z) \\[1ex]
  \end{bmatrix}. 
\end{align*}

Thus
\begin{align*}
   \Tr(\frac{\partial^2 D_{H(0)}}{\partial u \partial v}(x)\pi(0))=-\frac{1}{2}\psi_{uv}(z).
\end{align*}
More explicitly, we have $\Tr(\frac{\partial^2 D_{H(0)}}{\partial u \partial v}\pi(0)): UX \to \mathbb{R}$ satisfies
\begin{align*}
 \Tr(\frac{\partial^2 D_{H(0)}}{\partial u \partial v}\pi(0))(x)=-\frac{1}{2}\psi_{uv}(z(p(x)))=-\frac{1}{2}\phi_{uv}(p(x)) ,  
\end{align*}

where $p: UX \to X$ is the projection from the unit tangent bundle to our surface and $z$ is the Fermi coordinate we choose evaluating at the point $p(x)\in X$. We remark here the affine metric $\psi$ is always real and  $\phi=\psi-\sigma$ does not depend on coordinates we choose.

\item Compute $\Tr(\partial_{u}D_{H(0)}\partial_{v}\pi(0))$.

To study $\partial_{v}\pi(0)$ takes some effort. We set $u=0$ and take a family of flat connections $ \{D_{H(v)}\}$ with connection one forms $A(0,v)$ (recall equation (\ref{eq:connection})). Associated to each of them is a parallel transport equation along the closed geodesic $\gamma$ on $(S,\sigma)$:
\begin{equation}
    D_{H(v), \dot{\gamma}}V(v,t)=0
    \label{eq:FamilyOfParallel}
\end{equation}
with the assumption $\norm{V(v,0)}_{H(0)}=1$.

In \cite{Labourie_AnosovFlows}, Labourie proves images of every Hitchin representation are purely loxodromic. For $\rho(0,v)$ in $\mathcal{H}_{3}(S)$, we know $\rho(0,v)(\gamma)$ has distinct eigenvalues: $\lambda_{1}(\gamma,v)>\lambda_{2}(\gamma,v)>\lambda_{3}(\gamma,v)$. The holonomy problem for $\rho(0,v)$ has three distinct eigenvectors which are parallel sections $\{e_{i}(v,t)\}_{i=1}^{3}$ along $\gamma(t)$. Each section $V(v,t)=e_{i}(v,t)$ satisfies equation (\ref{eq:FamilyOfParallel}). In addition to the norm 1 condition at starting point : $\norm{V(v,0)}_{H(0)}=1$, we also impose another boundary condition in order to guarantee these are eigenvectors. The boundary conditions are for $i=1,2,3$,
\begin{enumerate}[label=(\roman*)]
    \item      $\norm{e_i(v,0)}_{H(0)}=1$;
    \item   $e_{i}(v,l_{\gamma})=\lambda_{i}(\gamma,v)e_{i}(v,0)$.
\end{enumerate}

The reader may notice that up to now, there are two frames for $E$ along $\gamma$ mentioned, the holomorphic frame $(s_{1},s_{2},s_{3})$ and the frame spanned by eigenvectors $(e_{1},e_{2}, e_{3})$. On the one hand, we can write our holomorphic frames as linear combinations of eigenvectors: $s_{i}(t)=\sum\limits_{j=1}^{3}a_{ij}(v,t)e_{j}(v,t)$ for i=1,2,3. On the other hand, we can write the eigenvectors as linear combinations of our holomorphic frames: $e_{j}(v,t)=\sum\limits_{k=1}^{3}e_{jk}(v,t)s_{k}(t)$ for $j=1,2,3$. We have the following observation:

With respect to the holomorphic frame $(s_{1},s_{2},s_{3})$, the projection onto $e_{1}$ along the hyperplane spanned by $(e_{2},e_{3})$ in matrix form is,
\begin{align*}
\pi(v,t)&=
\left[ \begin{array}{ccc}\pi(v,t)s_1(t),&\pi(v,t)s_2(t),&\pi(v,t)s_3(t)\end{array} \right]\\
\hfill\\
&=\left[  \begin{array}{ccc}a_{11}(v,t)e_{1}(v,t),&a_{21}(v,t)e_{1}(v,t),&a_{31}(v,t)e_{1}(v,t)\end{array} \right]\\
\hfill\\
&=\begin{bmatrix} a_{11}(v,t)e_{11}(v,t) & a_{21}(v,t)e_{11}(v,t) & a_{31}(v,t)e_{11}(v,t)\\ a_{11}(v,t)e_{12}(v,t) & a_{21}(v,t)e_{12}(v,t) & a_{31}(v,t)e_{12}(v,t) \\ a_{11}(v,t)e_{13}(v,t) & a_{21}(v,t)e_{13}(v,t) &a_{31}(v,t)e_{13}(v,t) \end{bmatrix}. 
\end{align*}

To understand $\partial_{v}\pi(0)$, we need to know $\partial_{v}e_{1}(0)$ and $\partial_{v}\alpha_{i1}(0)$ for $i=1,2,3$.
One can check in the holomorphic frame,
\begin{equation}
 \Tr(\partial_{u}D_{A(0)}\partial_{v}\pi(0))=q_{\alpha}(\partial_{v}a_{11}(0)e_{13}(0)+a_{11}(0)\partial_{v}e_{13}(0))+4\bar{q}_{\alpha}(\partial_{v}a_{31}(0)e_{11}(0)+a_{31}(0)\partial_{v}e_{11}(0)),
 \label{eq traceFormula}
\end{equation}

where $e_{11}(0)$ and $e_{13}(0)$ are known. Thus we need to compute $\partial_{v}e_{1}(0)$ and $\partial_{v}a_{11}(0)$ and $\partial_{v}a_{31}(0)$.
 
We first show how to obtain $\partial_{v}e_{1}(0)$. $\partial_{v}e_{1}(0,t)$ as the solution of an inhomogeneous ODE system arising from taking $v$-derivative for a family of parallel transport equations $(\ref{eq:FamilyOfParallel})$ at $v=0$,

\begin{align*}
  \partial_{t}\left[ \begin{array}{c} \partial_{v}e_{11}(0,t) \\ \partial_{v}e_{12}(0,t)\\ \partial_{v}e_{13}(0,t) \end{array} \right] + \begin{bmatrix} 0 & \frac{1}{2} & 0\\ 1 & 0 & \frac{1}{2} \\ 0 & 1 &0 \end{bmatrix}  \left[ \begin{array}{c} \partial_{v}e_{11}(0,t) \\ \partial_{v}e_{12}(0,t) \\ \partial_{v}e_{13}(0,t)  \end{array} \right]= -\frac{\sqrt{2}}{2}e^{t} \left[ \begin{array}{c} q_{\beta}(\Phi_{t}(x)) \\ 0 \\ 2\overline{q_{\beta}(\Phi_{t}(x))} \end{array} \right].
\end{align*}
with boundary conditions
\begin{align*}
    &H(\partial_{v}e_{1}(0,0),e_{1}(0,0))=0;\\
   &\partial_{v}e_{1}(0,l_{\gamma})=-e^{l_{\gamma}}\left(\int_{0}^{l_{\gamma}}\Re q_{\beta}(\Phi_{s}(x))\mathrm{d}s\right) e_{1}(0,0)+e^{l_{\gamma}}\partial_{v}e_{1}(0,0).\\ 
\end{align*}
The boundary conditions arise from taking $v$-derivative for boundary conditions $(i)$ and $(ii)$ of the parallel transport equation \eqref{eq:FamilyOfParallel} that the maximum eigenvector $e_1$ satisfies.

With these boundary conditions, we solve
\begin{align*}
  \left[\begin{array}{c} \partial_{v}e_{11}(t)\\ \partial_{v}e_{12}(t) \\ \partial_{v}e_{13}(t) \end{array}  \right]&=\left[\begin{array}{c} -\frac{\sqrt{2}}{2}\int_{0}^{t}e^s(\cosh(t-s)\Re q_{\beta}+i \Im q_{\beta})\mathrm{d}s\\[1ex]
  \sqrt{2}\int_{0}^{t}e^s(\sinh(t-s)\Re q_{\beta}\mathrm{d}s \\[1ex] -\sqrt{2}\int_{0}^{t}e^s(\cosh(t-s)\Re q_{\beta}-i \Im q_{\beta})\mathrm{d}s \end{array}  \right]\\[1ex]
  &+\left[\begin{array}{c}   -\frac{\sqrt{2}}{4}(e^{2l_{\gamma}}-1)^{-1}\int_{0}^{l_{\gamma}}e^{2s-t}\Re q_{\beta}\mathrm{d}s- \frac{\sqrt{2}}{2}i(e^{l_{\gamma}}-1)^{-1}
  \int_{0}^{l_{\gamma}}e^{s}\Im q_{\beta}\mathrm{d}s \\[1ex] -\frac{\sqrt{2}}{2}(e^{2l_{\gamma}}-1)^{-1}\int_{0}^{l_{\gamma}}e^{2s-t}\Re q_{\beta}\mathrm{d}s \\[1ex] -\frac{\sqrt{2}}{2}(e^{2l_{\gamma}}-1)^{-1}\int_{0}^{l_{\gamma}}e^{2s-t}\Re q_{\beta}\mathrm{d}s+ \sqrt{2}i(e^{l_{\gamma}}-1)^{-1}
  \int_{0}^{l_{\gamma}}e^{s}\Im q_{\beta}\mathrm{d}s \end{array}  \right].
\end{align*}
Here $ q_{\beta}$ refers to $q_{\beta}(\Phi_{s}(x))$ defined in notation \eqref{eq cubic}.

We continue to compute $\partial_{v}\alpha_{11}(0)$ and $\partial_{v}\alpha_{31}(0)$.
Combining $e_{j}(v,t)=\sum\limits_{k=1}^{3}e_{jk}(v,t)s_{k}(t)$ and $s_{i}(t)=\sum\limits_{j=1}^{3}a_{ij}(v,t)e_{j}(v,t)$ gives
\begin{equation}
     \sum\limits_{j=1}^{j=3}a_{ij}(v,t)e_{jk}(v,t)=\sigma_{ik}.
     \label{eq: ae}
\end{equation}

Recall $e_{jk}(0,t)$ are known:
\begin{align*}
    e_{1}(0,t)&=\left[\begin{array}{c} e_{11}(0,t)\\ e_{12}(0,t) \\ e_{13}(0,t) \end{array}  \right]=\frac{\sqrt{2}}{2}e^{t}\left[\begin{array}{c} \frac{1}{2}\\ -1 \\ 1 \end{array}  \right],\\
    e_{2}(0,t)&=\left[\begin{array}{c} e_{21}(0,t)\\ e_{22}(0,t) \\ e_{23}(0,t) \end{array}  \right]=\frac{1}{2} \left[\begin{array}{c} -1\\ 0 \\ 2 \end{array}  \right],\\
     e_{3}(0,t)&=\left[\begin{array}{c} e_{31}(0,t)\\ e_{32}(0,t) \\ e_{33}(0,t) \end{array}  \right]=\frac{\sqrt{2}}{2}e^{-t}\left[\begin{array}{c} \frac{1}{2}\\ 1 \\ 1 \end{array}  \right].
\end{align*} 

Then one obtains
\begin{equation*}
a(0,t)=\begin{bmatrix} a_{11} & a_{12} & a_{13}\\ a_{21} & a_{22} & a_{23} \\ a_{31} & a_{32} &a_{33} \end{bmatrix}=\begin{bmatrix} \frac{\sqrt{2}}{2}e^{-t} & -1 & \frac{\sqrt{2}}{2}e^{t}\\ -\frac{\sqrt{2}}{2}e^{-t} & 0 & \frac{\sqrt{2}}{2}e^{t} \\ \frac{\sqrt{2}}{4}e^{-t} & \frac{1}{2} &\frac{\sqrt{2}}{4}e^{t} \end{bmatrix}.
\end{equation*}

Taking the $v$-derivative of equation (\ref{eq: ae}) at $v=0$,
\begin{equation*}
  \sum\limits_{j=1}^{j=3}\partial_{v}a_{ij}(0,t)e_{jk}(0,t)+ \sum\limits_{j=1}^{j=3}a_{ij}(0,t)\partial_{v}e_{jk}(0,t)=0.
\end{equation*}

Solutions of $\partial_{v}a_{ij}(0,t)$ can be expressed in terms of $\partial_{v}e_{1}(0,t)$, $\partial_{v}e_{2}(0,t)$ and $\partial_{v}e_{3}(0,t)$. We have just solved $\partial_{v}e_{1}$. Similarly,$\partial_{v}e_{2}(0)$ and $\partial_{v}e_{3}(0)$ are solutions of another two systems of nonhomogeneous ODEs deduced from equation (\ref{eq:FamilyOfParallel}). We now proceed to solve $\partial_{v}e_{2}(0,t)$ and $\partial_{v}e_{3}(0,t)$.

\begin{enumerate}
    \item
For $\partial_{v}e_{2}(0,t)$, we have
\begin{align*}
  \partial_{t}\left[ \begin{array}{c} \partial_{v}e_{21}(0,t)\\ \partial_{v}e_{22}(0,t) \\ \partial_{v}e_{23}(0,t) \end{array} \right] + \begin{bmatrix} 0 & \frac{1}{2} & 0\\ 1 & 0 & \frac{1}{2} \\ 0 & 1 &0 \end{bmatrix}  \left[ \begin{array}{c} \partial_{v}e_{21}(0,t) \\ \partial_{v}e_{32}(0,t) \\ \partial_{v}e_{23}(0,t) \end{array} \right]=  \left[ \begin{array}{c} -q_{\beta}(\Phi_{t}(x))\\ 0 \\ 2\overline{q_{\beta}(\Phi_{t}(x))} \end{array} \right]
\end{align*}
with boundary conditions
\begin{align*}
    &H(\partial_{v}e_{2}(0,0),e_{2}(0,0))=0;\\
   &\partial_{v}e_{2}(0,l_{\gamma})=2\left(\int_{0}^{l_{\gamma}}\Re q_{\beta}(\Phi_{s}(x))\mathrm{d}s\right) e_{2}(0,0)+\partial_{v}e_{2}(0,0).\\ 
\end{align*}

\item
For $\partial_{v}e_{3}(0,t)$, we get
\begin{align*}
  \partial_{t}\left[ \begin{array}{c} \partial_{v}e_{31}(0,t) \\ \partial_{v}e_{32}(0,t) \\ \partial_{v}e_{33}(0,t) \end{array} \right] + \begin{bmatrix} 0 & \frac{1}{2} & 0\\ 1 & 0 & \frac{1}{2} \\ 0 & 1 &0 \end{bmatrix}  \left[ \begin{array}{c} \partial_{v}e_{31}(0,t) \\ \partial_{v}e_{32}(0,t) \\ \partial_{v}e_{33}(0,t) \end{array} \right]= -\frac{\sqrt{2}}{2}e^{-t} \left[ \begin{array}{c} q_{\beta}(\Phi_{t}(x)) \\ 0 \\ 2\overline{q_{\beta}(\Phi_{t}(x))} \end{array} \right]
\end{align*}
with boundary conditions
\begin{align*}
    &H(\partial_{v}e_{3}(0,0),e_{3}(0,0))=0;\\
   &\partial_{v}e_{3}(0,l_{\gamma})=-e^{-l_{\gamma}}\left(\int_{0}^{l_{\gamma}}\Re q_{\beta}(\Phi_{s}(x))\mathrm{d}s\right) e_{3}(0,0)+e^{-l_{\gamma}}\partial_{v}e_{3}(0,0).\\ 
\end{align*}
\end{enumerate}
We obtain solutions respectively as follows
\begin{align*}
  \left[\begin{array}{c} \partial_{v}e_{21}(t)\\ \partial_{v}e_{22}(t) \\ \partial_{v}e_{23}(t) \end{array}  \right]&=\left[\begin{array}{c} -\int_{0}^{t}\Re q_{\beta}+ i\cosh(t-s)\Im q_{\beta} \mathrm{d}s\\ 2\int_{0}^{t}i \sinh(t-s)\Im q_{\beta}\mathrm{d}s \\[1ex]
  2\int_{0}^{t}\Re q_{\beta}- i\cosh(t-s)\Im q_{\beta} \mathrm{d}s  \end{array} \right]\\[1ex]
  &+\left[\begin{array}{c}   -\frac{i}{2}\int_{0}^{l_{\gamma}}\Im q_{\beta}((1-e^{
 l_{\gamma}})^{-1}e^{l_{\gamma}+t-s}+(1-e^{
 -l_{\gamma}})^{-1}e^{-l_{\gamma}-t+s}) \mathrm{d}s \\[1ex] i\int_{0}^{l_{\gamma}}\Im q_{\beta}((1-e^{
 l_{\gamma}})^{-1}e^{l_{\gamma}+t-s}-(1-e^{
 -l_{\gamma}})^{-1}e^{-l_{\gamma}-t+s})\mathrm{d}s \\[1ex] -i\int_{0}^{l_{\gamma}}\Im q_{\beta}((1-e^{
 l_{\gamma}})^{-1}e^{l_{\gamma}+t-s}+(1-e^{
 -l_{\gamma}})^{-1}e^{-l_{\gamma}-t+s})\mathrm{d}s \end{array}  \right]
 \end{align*} 
 
 and
 
\begin{align*}
  \left[\begin{array}{c} \partial_{v}e_{31}(t)\\ \partial_{v}e_{32}(t) \\ \partial_{v}e_{33}(t) \end{array}  \right]&=\left[\begin{array}{c} -\frac{\sqrt{2}}{2}\int_{0}^{t}e^{-s}(\cosh(t-s)\Re q_{\beta}+ i\Im q_{\beta}) \mathrm{d}s\\ \sqrt{2}\int_{0}^{t} e^{-s}\sinh(t-s)\Re q_{\beta}\mathrm{d}s \\[1ex]
  -\sqrt{2}\int_{0}^{t}e^{-s}(\cosh(t-s)\Re q_{\beta}- i\Im q_{\beta}) \mathrm{d}s  \end{array} \right]\\[1ex]
  &+\left[\begin{array}{c}  -\frac{\sqrt{2}}{4}(e^{-2l_{\gamma}}-1)^{-1}\int_{0}^{l_{\gamma}}e^{t-2s}\Re q_{\beta} \mathrm{d}s -\frac{\sqrt{2}}{2}i(e^{-l_{\gamma}}-1)^{-1}\int_{0}^{l_{\gamma}}e^{-s}\Im q_{\beta} \mathrm{d}s  \\[1ex] \frac{\sqrt{2}}{2}(e^{-2l_{\gamma}}-1)^{-1}\int_{0}^{l_{\gamma}}e^{t-2s}\Re q_{\beta} \mathrm{d}s \\[1ex] -\frac{\sqrt{2}}{2}(e^{-2l_{\gamma}}-1)^{-1}\int_{0}^{l_{\gamma}}e^{t-2s}\Re q_{\beta} \mathrm{d}s +\sqrt{2}i(e^{-l_{\gamma}}-1)^{-1}\int_{0}^{l_{\gamma}}e^{-s}\Im q_{\beta} \mathrm{d}s \end{array}  \right],
 \end{align*}  
 
where again $q_{\beta}$ in the solutions again refers to $q_{\beta}(\Phi_{s}(x))$ defined in notation \eqref{eq cubic}.

We are therefore able to solve $\partial_{v}a_{ij}(0,t)$ from $\partial_{v}e_{1}(0,t)$, $\partial_{v}e_{2}(0,t)$ and $\partial_{v}e_{3}(0,t)$. For a closed geodesic $\gamma$ of length $l_{\gamma}$ starting from $\dot{\gamma}(0)=x$, we compute from equation \eqref{eq traceFormula}, 
\begin{align}
&\Tr(\partial_{u}D_{A(0)}\partial_{v}\pi(0))(\Phi_{t}(x))\nonumber  \\
=&\Re q_{\alpha}(\Phi_{t}(x))\int_{0}^{t}(e^{2(t-s)}-e^{2(s-t)})\Re q_{\beta}(\Phi_{s}(x))\mathrm{d}s\nonumber \\
+&2\Im q_{\alpha}(\Phi_{t}(x))\int_{0}^{t}(e^{t-s}-e^{s-t})\Im q_{\beta}(\Phi_{s}(x))\mathrm{d}s\nonumber \\
+&\Re q_{\alpha}(\Phi_{t}(x))\int_{0}^{l_{\gamma}}(\frac{e^{2(t-s)}}{e^{-2l_{\gamma}}-1}-\frac{e^{2(s-t)}}{e^{2l_{\gamma}}-1})\Re q_{\beta}(\Phi_{s}(x))\mathrm{d}s\nonumber \\
+&2\Im q_{\alpha}(\Phi_{t}(x))\int_{0}^{l_{\gamma}}(\frac{e^{t-s}}{e^{-l_{\gamma}}-1}-\frac{e^{s-t}}{e^{l_{\gamma}}-1})\Im q_{\beta}(\Phi_{s}(x))\mathrm{d}s. \label{eq TrMixI}
\end{align}

In particular, at $t=0$,
\begin{align}
&\Tr(\partial_{u}D_{A(0)}\partial_{v}\pi(0))(x)\nonumber \\ 
=&\Re q_{\alpha}(x)\int_{0}^{l_{\gamma}}(\frac{e^{-2s}}{e^{-2l_{\gamma}}-1}-\frac{e^{2s}}{e^{2l_{\gamma}}-1})\Re q_{\beta}(\Phi_{s}(x))\mathrm{d}s \nonumber\\ 
+&2\Im q_{\alpha}(x)\int_{0}^{l_{\gamma}}(\frac{e^{-s}}{e^{-l_{\gamma}}-1}-\frac{e^{s}}{e^{l_{\gamma}}-1})\Im q_{\beta}(\Phi_{s}(x))\mathrm{d}s.
\label{eq TrMixII}
\end{align}

\begin{rem}
One may notice every point on the closed geodesic $\gamma$ plays equivalent roles. We can always let $y=\Phi_{t}(x)$ be the initial point of our $\gamma$ and set up boundary conditions for our ODEs based at $y$ instead of $x$. The solution of this new ODE system is  equation \eqref{eq TrMixII} treating $y=\Phi_{t}(x)$ as the initial point. It is in fact the same as starting from $x$ and obtain $\Tr(\partial_{u}D_{A(0)}\partial_{v}\pi(0))(\Phi_{t}(x))$ from equation \eqref{eq TrMixI}.
\end{rem}

\subsection{H{\"o}lder extension to the surface}
The holonomy problems are only able to be solved on closed geodesics as they can be simplified as linear ODEs with boundary conditions. However it is still possible to extend the computation for second variations of reparametrization functions from closed geodesics to the Riemann surface $X$. This will be our goal in this subsection. In particular, We will prove  in the end of this subsection the main proposition about second variation of reparametrization functions.

\begin{prop}
\label{prop secondVar}
The second variation of reparametrization functions $\partial_{uv}f_{\rho(0)}:UX \to \mathbb{R}$ for our model case $\partial{\beta}g_{\alpha \alpha}(\sigma)$ satisfies
\begin{align*}
\partial_{uv}f_{\rho(0)}(x)\sim&\frac{1}{2}\phi_{uv}(p(x)) \nonumber \\
 +&\Re q_{\alpha}(x)\int_{0}^{\infty}e^{-2s}\Re q_{\beta}(\Phi_{s}(x))\mathrm{d}s+\Re q_{\alpha}(x)\int_{-\infty}^{0}e^{2s}\Re q_{\beta}(\Phi_{s}(x))\mathrm{d}s \nonumber \\
    +&2\Im q_{\alpha}(x)\int_{0}^{\infty}e^{-s}\Im q_{\beta}(\Phi_{s}(x))\mathrm{d}s+2\Im q_{\alpha}(x)\int_{-\infty}^{0}e^{s}\Im q_{\beta}(\Phi_{s}(x))\mathrm{d}s,
\end{align*}
where we recall that $\phi$ is defined in equation \eqref{eq:affinePDE1Global} and  $p: UX \to X$ is the projection from the unit tangent bundle $UX$ to our Riemann surface $X$.
\end{prop}

Let's define the second part in the above formula $\partial_{uv}f_{\rho(0)}$ as a function $\eta:UX \to \mathbb{R}$,
\begin{align}
    \eta(x)
    =&-\Re q_{\alpha}(x)\int_{0}^{\infty}e^{-2s}\Re q_{\beta}(\Phi_{s}(x))\mathrm{d}s-\Re q_{\alpha}(x)\int_{-\infty}^{0}e^{2s}\Re q_{\beta}(\Phi_{s}(x))\mathrm{d}s \nonumber \\
    -&2\Im q_{\alpha}(x)\int_{0}^{\infty}e^{-s}\Im q_{\beta}(\Phi_{s}(x))\mathrm{d}s-2\Im q_{\alpha}(x)\int_{-\infty}^{0}e^{s}\Im q_{\beta}(\Phi_{s}(x))\mathrm{d}s.
    \label{eq eta}
\end{align}.

We will prove that $\eta(x)$ coincides with $\Tr(\partial_{u}D_{A(0)}\partial_{v}\pi(0))(x)$ on periodic orbits and that $\eta(x)$ is a H{\"o}lder function. Denoting the subset of $UX$ that consists of all unit tangent vectors to closed geodesics as $W$, we first show

\begin{prop}
For any $x\in W$, $\eta(x)=\Tr(\partial_{u}D_{A(0)}\partial_{v}\pi(0))(x)$.
\label{prop eta}
\end{prop}

To prove Proposition \ref{prop eta}, from the computation of $\Tr(\partial_{u}D_{A(0)}\partial_{v}\pi(0))(x)$ in equation \eqref{eq TrMixII}, we introduce an intermediate function $\psi: W \times \mathbb{R}^{+} \longrightarrow \mathbb{R}$ by
\begin{align*}
&\psi(x,r)=\Re q_{\alpha}(x)\int_{0}^{r}(\frac{e^{-2s}}{e^{-2r}-1}-\frac{e^{2s}}{e^{2r}-1})\Re q_{\beta}(\Phi_{s}(x))\mathrm{d}s \nonumber\\ 
+&2\Im q_{\alpha}(x)\int_{0}^{r}(\frac{e^{-s}}{e^{-r}-1}-\frac{e^{s}}{e^{r}-1})\Im q_{\beta}(\Phi_{s}(x))\mathrm{d}s.
\end{align*}

Given $x\in W$, if we denote the closed geodesic that $x$ is tangential to as ${\gamma}_{\Sub x}$ with length $l_{{\gamma}_{\Sub x}}$, then it is clear that $\Tr(\partial_{u}D_{A(0)}\partial_{v}\pi(0))(x)=\psi(x,l_{\gamma_x})$. To prove Proposition \ref{prop eta} for the set $W$, we need the following Lemma \ref{lem kl}. It states that $\psi(x,r)$ attains the same value when $r$ is any positive integer multiple of $l_{\gamma_x}$,
\begin{lem}
$\psi(x,kl_{\gamma_x}) =\psi(x,l_{\gamma_x})  \hspace{.5 in}    \forall x\in W,\text{ } \forall k\in \mathbb{Z}^{+}$.
\label{lem kl}
\end{lem}

\begin{proof}
For any $k\in \mathbb{Z}^{+}$, we have
\begin{align*}
 &\int_{0}^{kl_{\gamma_x}}(\frac{e^{-2s}}{e^{-2kl_{\gamma_x}}-1}-\frac{e^{2s}}{e^{2kl_{\gamma_x}}-1})\Re q_{\beta}(\Phi_{s}(x))\mathrm{d}s \nonumber\\
 =&\sum\limits_{i=1}^{k}\int_{(i-1)l_{\gamma_x}}^{il_{\gamma_x}}(\frac{e^{-2s}}{e^{-2kl_{\gamma_x}}-1}-\frac{e^{2s}}{e^{2kl_{\gamma_x}}-1})\Re q_{\beta}(\Phi_{s}(x))\mathrm{d}s \nonumber\\
  =&\frac{1}{e^{-2kl_{\gamma_x}}-1}\sum\limits_{i=1}^{k} \int_{(i-1)l_{\gamma_x}}^{il_{\gamma_x}}e^{-2s}\Re q_{\beta}(\Phi_{s}(x))\mathrm{d}s- \frac{1}{e^{2kl_{\gamma_x}}-1}\sum\limits_{i=1}^{k} \int_{(i-1)l_{\gamma_x}}^{il_{\gamma_x}}e^{2s}\Re q_{\beta}(\Phi_{s}(x))\mathrm{d}s\nonumber\\
  =&\int_{0}^{l_{\gamma_x}}(\frac{e^{-2s}}{e^{-2l_{\gamma_x}}-1}-\frac{e^{2s}}{e^{2l_{\gamma_x}}-1})\Re q_{\beta}(\Phi_{s}(x))\mathrm{d}s.
\end{align*}
Similar arguments hold for $\int_{0}^{l_{\gamma_x}}(\frac{e^{-s}}{e^{-l_{\gamma_x}}-1}-\frac{e^{s}}{e^{l_{\gamma_x}}-1})\Im q_{\beta}(\Phi_{s}(x))\mathrm{d}s$.

Thus we obtain $\psi(x,kl_{\gamma_x}) =\psi(x,l_{\gamma_x})$.
\end{proof}

\begin{rem}
This equality is clear if one understands that $\psi(x,kl_{\gamma})$ is the solution of the holonomy problem that goes around our closed geodesic $\gamma$ $k$-times with the same boundary conditions.
\end{rem}

Now we are able to prove Proposition \ref {prop eta}.
\begin{proof}
[Proof of Proposition \ref {prop eta}]\hfill

Instead of flowing from $x$ to $\Phi_{l_{\gamma_x}}(x)$, we view $x$ as our midpoint and consider our flow from $\Phi_{-\frac{l_{\gamma_x}}{2}}(x)$ to $x$ and then from $x$ to  $\Phi_{\frac{l_{\gamma_x}}{2}}(x)$. From this point of view, we can write $\psi(x,l_{\gamma_x})$ as
\begin{align*}
\psi(x,l_{\gamma_x})=&\Re q_{\alpha}(x)\int_{0}^{\frac{l_{\gamma_x}}{2}}(\frac{e^{-2s}}{e^{-2l_{\gamma_x}}-1}-\frac{e^{2s}}{e^{2l_{\gamma_x}}-1})\Re q_{\beta}(\Phi_{s}(x))\mathrm{d}s \\ 
+&\Re q_{\alpha}(x)\int_{-\frac{l_{\gamma_x}}{2}}^{0}(\frac{e^{2s}}{e^{-2l_{\gamma_x}}-1}-\frac{e^{-2s}}{e^{2l_{\gamma_x}}-1})\Re q_{\beta}(\Phi_{s}(x))\mathrm{d}s \\ 
+&2\Im q_{\alpha}(x)\int_{0}^{\frac{l_{\gamma_x}}{2}}(\frac{e^{-s}}{e^{-l_{\gamma_x}}-1}-\frac{e^{s}}{e^{l_{\gamma_x}}-1})\Im q_{\beta}(\Phi_{s}(x))\mathrm{d}s\\
+&2\Im q_{\alpha}(x)\int_{-\frac{l_{\gamma_x}}{2}}^{0}(\frac{e^{s}}{e^{-l_{\gamma_x}}-1}-\frac{e^{-s}}{e^{l_{\gamma_x}}-1})\Im q_{\beta}(\Phi_{s}(x))\mathrm{d}s.
\end{align*}
The above also holds if we replace $l_{\gamma_x}$ by $kl_{\gamma_x}$.
We now show $\eta(x)=\Tr(\partial_{u}D_{A(0)}\partial_{v}\pi(0))(x)$ for $x\in W$ by taking $k\to \infty$ in the above formula.

Suppose $\max\limits_{x\in UX}\left\{|\Re q_{\alpha}(x)|, |\Im q_{\alpha}(x)|, |\Re q_{\beta}(x)|, |\Im q_{\beta}(x)|\right\}=M$. Then notice
\begin{align*}
    &|\psi(x,kl_{\gamma_x})-\eta(x)|\\
    \leq&2M^{2}\int_{\frac{kl_{\gamma_x}}{2}}^{\infty}e^{-2s}\mathrm{d}s
    +4M^{2}\int_{\frac{kl_{\gamma_x}}{2}}^{\infty}e^{-s}\mathrm{d}s\\
+&2M^{2}\int_{0}^{\frac{kl_{\gamma_x}}{2}}\left|\frac{e^{-2s}}{e^{-2kl_{\gamma_x}}-1}-\frac{e^{2s}}{e^{2kl_{\gamma_x}}-1}+e^{-2s}\right|\mathrm{d}s \\ 
+&4M^{2}\int_{0}^{\frac{kl_{\gamma_x}}{2}}\left|\frac{e^{-s}}{e^{-kl_{\gamma_x}}-1}-\frac{e^{s}}{e^{kl_{\gamma_x}}-1}+e^{-s}\right|\mathrm{d}s\\
\longrightarrow&{0}   \hspace{.5in} \text{when $k\to\infty$}
\end{align*}

Thus by Lemma \ref{lem kl}, we obtain for any $x \in W$,
\begin{align*}
  \Tr(\partial_{u}D_{A(0)}\partial_{v}\pi(0))(x)=\psi(x,l_{\gamma_x})=\lim\limits_{k\to\infty}\psi(x,kl_{\gamma_x})=\eta(x). 
\end{align*}
\end{proof}

We also need the following proposition about regularity of the function $\eta$.
\begin{prop}    
$\eta(x): UX \to \mathbb{R}$ is a H{\"o}lder function. \label{prop Holder}
\end{prop}
\begin{proof}

We starting by showing $\int_{0}^{\infty}e^{-s}\Im q_{\beta}(\Phi_{s}(x))\mathrm{d}s$ is H{\"o}lder. Let $x$ and $y$ be close so that $d(x,y)= \epsilon<<1$. It is classical for a hyperbolic surface $(S,\sigma)$, we have standard ODE estimates on the geodesic flow as follows.
\begin{align*}
    d(\Phi_{s}(x), \Phi_{s}(y)) \leq N e^{s}d(x,y) = \epsilon N e^s,
\end{align*} 
where $N>0$ is some constant and the distance function $d$ on $UX$ is induced from the canonical (Sasaki) metric $\langle \cdot,\cdot \rangle$ on $UX$.

Consider $T=-\log(\epsilon)$. Then dividing the integral into two parts, from $0$ to $T$ and from $T$ to $\infty$, yields

\begin{align*}
    &|\int_{0}^{\infty}e^{-s}\Im q_{\beta}(\Phi_{s}(x))\mathrm{d}s-\int_{0}^{\infty}e^{-s}\Im q_{\beta}(\Phi_{s}(y))\mathrm{d}s|\\
    =&\bigg|\int_{0}^{T}e^{-s}\bigg(\Im q_{\beta}(\Phi_{s}(x))-\Im q_{\beta}(\Phi_{s}(y))\bigg)\mathrm{d}s\bigg|+\bigg|\int_{T}^{\infty}e^{-s}\bigg(\Im q_{\beta}(\Phi_{s}(x))-\Im q_{\beta}(\Phi_{s}(y))\bigg)\mathrm{d}s\bigg|\\
    \leq& \int_{0}^{T} e^{-s}N_1N\epsilon e^s\mathrm{d}s+ 2N_2e^{-T}\\
    \leq& -N_1N\epsilon\log(\epsilon) + 2N_2\epsilon\\
     \leq& (N_1N+2N_2)d(x,y)^{\frac{1}{2}}
\end{align*}

Here we use the fact that $\Im q_{\beta}$ is smooth. So we can assume its Lipschitz constant to be $N_1$. We also use that $UX$ is compact and we assume $\sup\limits_{x\in UX} \Im q_{\beta}(x)=N_2$. 

It then follows easily that $\Im q_{\alpha}(x)\int_{0}^{\infty}e^{-2s}\Im q_{\beta}(\Phi_{s}(x))\mathrm{d}s$ is also a H{\"o}lder function. The arguments to show that the other three terms in $\eta(x)$ are H{\"o}lder are the same. We therefore conclude that $\eta(x)$ is a H{\"o}lder function. 
\end{proof}

Finally with Proposition \ref{prop eta} and Proposition \ref{prop Holder}, we are able to prove Proposition \ref{prop secondVar} about second variations of reparametrization functions on the Riemann surface $X$.

\begin{proof}[Proof of Proposition \ref{prop secondVar}]
We have done most of necessary elements for this proof in previous estimates. We assemble everything together here. 
Because $\Tr(\partial_{u}D_{A(0)}\partial_{v}\pi(0))(\Phi_{t}(x))$ is a H{\"o}lder function, and because it equals to the H{\"o}lder function $\eta(x)$ \eqref{eq eta} on a dense subset of $UX$. We conclude it coincides everywhere with $\eta(x)$ on $UX$.
We obtain
\begin{align*}
  \partial_{uv}f_{\rho(0)}\sim& -\partial_{u}(\Tr(\partial_{v}D_{H(0)} \pi(0)))\\
 =&-\Tr(\frac{\partial^2 D_{H(0)}}{\partial u \partial v}\pi(0)) - \Tr(\partial_{v}D_{H(0)}\partial_{u}\pi(0))\\
 =&\frac{1}{2}\phi_{uv}(p(x)) \nonumber \\
 &+\Re q_{\alpha}(x)\int_{0}^{\infty}e^{-2s}\Re q_{\beta}(\Phi_{s}(x))\mathrm{d}s+\Re q_{\alpha}(x)\int_{-\infty}^{0}e^{2s}\Re q_{\beta}(\Phi_{s}(x))\mathrm{d}s \nonumber \\
    &+2\Im q_{\alpha}(x)\int_{0}^{\infty}e^{-s}\Im q_{\beta}(\Phi_{s}(x))\mathrm{d}s+2\Im q_{\alpha}(x)\int_{-\infty}^{0}e^{s}\Im q_{\beta}(\Phi_{s}(x))\mathrm{d}s.
\end{align*}
where we recall here $\phi=\log \frac{2h}{\sigma}$ is a globally well-defined function defined in \eqref{eq:affinePDE1Global} evaluating at the point $p(x)\in X$ and $p: UX \to X$ is the projection from the unit tangent bundle to our surface.

\end{proof}

\end{itemize}

\section{Evaluation on Poincar\'e disk for the model case}
After the computation of first and second variations of reparametrization functions on $UX$ in the last two section, we are able to evaluate $\partial_{\beta}g_{\alpha \alpha}(\sigma)$. Our goal in this section is to show the following,

\begin{prop}
\label{prop Model}
 For $\sigma\in\mathcal{T}(S),\partial_{\beta} g_{\alpha \alpha}(\sigma)=0$.
\end{prop}

Let's first write down the expression for $\partial_{\beta}g_{\alpha \alpha}(\sigma)$,
\begin{align*}
 \partial_{\beta}g_{\alpha \alpha}(\sigma)&=\partial_{v} \Bigg( {\langle \partial_{u}\rho(0,v), \partial_{u}\rho(0,v) \rangle}_{\Sub P}\Bigg) \Bigg|_{v=0}\\
    &=\lim\limits_{r\to \infty}\frac{1}{r}\left[\int_{\Sub UX}\left(\int_{0}^{r} \partial_{u}f_{\rho(0)}^{N}\mathrm{d}t \right)^2\int_{0}^{r}\partial_{v}f_{\rho(0)}^{N}\mathrm{d}t\mathrm{d}m_{0}+
    2\int_{\Sub UX} \int_{0}^{r}\partial_{u}f_{\rho(0)}^{N} \mathrm{d}t\int_{0}^{r} \partial_{uv}f_{\rho(0)}^{N}\mathrm{d}t\mathrm{d}m_{0}\right]\\
     &=\lim\limits_{r\to \infty}\frac{1}{r}\int_{\Sub UX}\left(\int_{0}^{r} \Re q_{\alpha}(\Phi_{t}(x))\mathrm{d}t \right)^2\int_{0}^{r}\Re q_{\beta}(\Phi_{t}(x))\mathrm{d}t\mathrm{d}m_{0}\\
     &+\lim\limits_{r\to \infty}\frac{1}{r}\int_{\Sub UX} 2\int_{0}^{r}\Re q_{\alpha}(\Phi_{t}(x)) \mathrm{d}t\int_{0}^{r} -\partial_{uv}h(\rho(0))-\partial_{uv}f_{\rho(0)}(\Phi_{t}(x))\mathrm{d}t\mathrm{d}m_{0}\\
     &=\mathrm{I}+\mathrm{II}
\end{align*}
Here the first term is denoted as $\mathrm{I}$ and the second term is denoted as $\mathrm{II}$. The formula for $\partial_{uv}f_{\rho(0)}$ is given in Proposition (\ref{prop secondVar}).

We aim to prove both $\mathrm{I}$ and $\mathrm{II}$ are zero for Proposition \ref{prop Model}. The following lemma will be crucial.
\begin{lem} \label{lem 3cubic}

For any $t,s\in \mathbb{R}$, we have
\begin{align}
&\int_{\Sub UX}\Re q_{\alpha}(x)\Re q_{\alpha}(\Phi_{t}(x))\Re q_{\beta}(\Phi_{s}(x))\mathrm{d}m_{0}(x)=0.\label{eq 3cubic1} \\
&\int_{\Sub UX}\Re q_{\alpha}(x)\Im q_{\alpha}(\Phi_{t}(x))\Im q_{\beta}(\Phi_{s}(x))\mathrm{d}m_{0}(x)=0. 
\label{eq 3cubic2}
\end{align}
\end{lem}

We use the methods in  \cite{Variation_along_FuchsianLocus} to show the integrals are zero. Similarly to the proof of Theorem 6.3.1 in \cite{Variation_along_FuchsianLocus}, the key is to use the symmetries properties of the Liouville measure $m_{0}=m_{L}$ and homogeneity of holomorphic differentials viewed as functions on $UX$. We transfer the problem of evaluating the integrals in equation \text{\eqref{eq 3cubic1}} and equation \text{\eqref{eq 3cubic2}} to analyzing the Fourier coefficients of holomorphic differentials.

Before we start our proof, we first explain the coordinates we will use to do the computation following \cite{Variation_along_FuchsianLocus}. We take Poincar\'e disk as our charts. Pick a point $x\in UX$. We identify the universal cover of $(X,\sigma)$ with $\mathbb{D}$ by the unique isometry
that takes $\pi(x)\in X$ to $0\in \mathbb{D}$ and identify the vector $x\in UX$ with vector $(1,0)\in T_{0}\mathbb{D}$. 

We express our holomorphic differentials in these coordinates. For the holomorphic cubic differential $q_{\alpha}$, it has the following analytic expansion in the coordinate based on $x$,
\begin{equation*}
    q_{\alpha,x}(z)=\sum\limits_{n=1}^{\infty}a_{n}(x)z^{n}dz^{3}.
\end{equation*}
Recall the hyperbolic distance $d_{H}$ in the Poincar\'e disk model satisfies,
\begin{equation*}
    d_{H}(0,Re^{i\theta})=r(R)=\frac{1}{2}\log (\frac{1+R}{1-R}).
\end{equation*}
Thus $\frac{\partial}{\partial_{r}}=(1-R^{2})\frac{\partial}{\partial R}$ and 
\begin{equation*}
    dz(\frac{\partial}{\partial r})\bigg|_{Re^{i\theta}}=(1-R^{2})e^{i\theta}.
\end{equation*}

Denoting $\tilde{q}_{\alpha,x}(z):=\Re \left(q_{\alpha,x}(z)(\frac{\partial}{\partial r},\frac{\partial}{\partial r},\frac{\partial}{\partial r})\right)$, one has
\begin{equation}
   \Re q_{\alpha}(\Phi_{r}(e^{i\theta}x))= \tilde{q}_{\alpha,x}(R e^{i\theta})=\Re \Bigg(\sum\limits_{n=0}^{\infty}a_{n}(x)R^{n}(1-R^2)^{3}e^{i(n+3)\theta} \Bigg).
   \label{eq: analyticExI}
\end{equation}
In particular,
 when $r=0$, 
 \begin{equation*}
    \lim\limits_{R\to 0} dz(\frac{\partial}{\partial r})\bigg|_{R e^{i\theta}}=e^{i\theta}.
\end{equation*}

Therefore,
\begin{equation}
   \Re q_{\alpha}(e^{i\theta}x)= \tilde{q}_{\alpha,x}(0\cdot e^{i\theta})=
   \lim\limits_{R\to 0}\Re \left(q_{\alpha,x}(R e^{i\theta})(\frac{\partial}{\partial r},\frac{\partial}{\partial r},\frac{\partial}{\partial r})\right)=\Re(a_{0}(x)e^{i3\theta}).
  \label{eq: analyticExIS} 
\end{equation}

Suppose the coefficients of the analytic expansion for $q_{\beta}$ are $b_{n}$, then
\begin{equation}
   \Re q_{\beta}(\Phi_{r}(e^{i\theta}x))= \tilde{q}_{\beta,x}(Re^{i\theta})=\Re \Bigg(\sum\limits_{n=0}^{\infty}b_{n}(x)R^{n}(1-R^2)^{3}e^{i(n+3)\theta} \Bigg).
   \label{eq: analyticExII}
\end{equation}
For the convenience of computation later for other cases, we also write down here two analytic expansions for holomorphic quadratic differentials $q_{i}, q_{j}$ with coefficients $c_{n}$ and $d_{n}$ respectively.
\begin{align}
\Re q_{i}(\Phi_{r}(e^{i\theta}x))&= \tilde{q}_{i,x}(Re^{i\theta})=\Re \Bigg(\sum\limits_{n=0}^{\infty}c_{n}(x)R^{n}(1-R^2)^{2}e^{i(n+2)\theta} \Bigg).
   \label{eq: analyticExIII}\\
   \Re q_{j}(\Phi_{r}(e^{i\theta}x))&= \tilde{q}_{j,x}(Re^{i\theta})=\Re \Bigg(\sum\limits_{n=0}^{\infty}d_{n}(x)R^{n}(1-R^2)^{2}e^{i(n+2)\theta} \Bigg).
   \label{eq: analyticExIIII}
\end{align}

\begin{proof}[Proof of Lemma \ref {lem 3cubic}]\hfill

We begin with showing equation \text{\eqref{eq 3cubic1}}. 
 
The proof of it will be divided into two cases: 
\begin{itemize}
    \item  $t\geq 0$ and $s \geq 0$; 
    \item $t < 0$ or $s<0$.
\end{itemize}

In the first case,  we work with the analytic expansions (\ref{eq: analyticExI}) and (\ref{eq: analyticExII}). We choose for two special situations: $s=t$ and $s=\frac{t}{2}$. We observe some symmetries in these two situations and argue from these symmetries that equation \text{\eqref{eq 3cubic1}} holds for the first case. We then apply the results for the first case to the second case by flow invariant properties of $m_{L}$. Equation \text{\eqref{eq 3cubic2}} then follows easily from equation \text{\eqref{eq 3cubic1}} once we find the relation between them.

Since $m_{0}=m_{L}$ is rotationally invariant, i.e. $(e^{i\theta})^{*}m_{L}=m_{L}$, we have
\begin{align*}
&\int_{\Sub UX}\Re q_{\alpha}(x)\Re q_{\alpha}(\Phi_{t}(x))\Re q_{\beta}(\Phi_{s}(x))\mathrm{d}m_{0}(x)\\
=&\frac{1}{2\pi}\int_{0}^{2\pi}\int_{\Sub UX}\Re q_{\alpha}(e^{i\theta}x)\Re q_{\alpha}(\Phi_{t}(e^{i\theta}x))\Re q_{\beta}(\Phi_{s}(e^{i\theta}x))\mathrm{d}m_{0}(x)\mathrm{d}\theta.
\end{align*}

\begin{enumerate}
\item 
We restrict ourselves to the case $t,s \geq 0$ of equation \text{\eqref{eq 3cubic1}} so that we can work with the analytic expansions (\ref{eq: analyticExI}) and (\ref{eq: analyticExII}). 

 We denote $t(T)=\frac{1}{2}\log (\frac{1+T}{1-T})$ and $s(S)=\frac{1}{2}\log (\frac{1+S}{1-S})$. We first consider $t>0$ and $s>0$. 
Then if we first integral over the $\theta$-variable, in terms of the analytic expansion, we get
\begin{align}
&\int_{\Sub UX}\Re q_{\alpha}(x)\Re q_{\alpha}(\Phi_{t}(x))\Re q_{\beta}(\Phi_{s}(x))\mathrm{d}m_{0}(x)\nonumber \\
=&\frac{1}{4}\sum\limits_{n=0}^{\infty}\bigg(\int_{UX}\Re(a_{0}a_{n}\bar{b}_{n+3})\mathrm{d}m_{0}T^{n}(1-T^{2})^{3}S^{n+3}(1-S^{2})^{3}\nonumber \\
+&\int_{UX}\Re(a_{0}\bar{a}_{n+3}b_{n})\mathrm{d}m_{0}T^{n+3}(1-T^{2})^{3}S^{n}(1-S^{2})^{3}\bigg).
\label{eq writeOutExpansion}
\end{align}
We denote $A_n= \int_{UX}\Re(a_{0}a_{n}\bar{b}_{n+3})\mathrm{d}m_{0}$ and $B_n=\int_{UX}\Re(a_{0}\bar{a}_{n+3}b_{n})\mathrm{d}m$. To show equation \eqref{eq 3cubic1} holds for $t,s \geq 0$, it suffices to prove for $n\geq0$,
\begin{equation}
   A_n=B_n=0.
    \label{eq coeff}
\end{equation}

If $t=0$ or $s=0$, equation \text{\eqref{eq 3cubic1}} is equivalent to the follows which are included in equation \text{\eqref{eq coeff}}:
 \begin{equation*}
    A_0=B_0=0. 
 \end{equation*}

To prove equation \text{\eqref{eq coeff}}, we consider two special cases of equation \text{\eqref{eq 3cubic1}}: flow time $s=t$ and $s=\frac{t}{2}$.

 By the $\Phi_{t}$-invariance of $m_{0}$, flow time $s=t$ satisfies 
 \begin{align*}
 &\int_{\Sub UX}\Re q_{\alpha}(x)\Re q_{\alpha}(\Phi_{t}(x))\Re q_{\beta}(\Phi_{t}(x))\mathrm{d}m_{0}(x)\\
 =&\int_{\Sub UX}\Re q_{\alpha}(\Phi_{-t}(x))\Re q_{\alpha}(x)\Re q_{\beta}(x)\mathrm{d}m_{0}(x).
\end{align*}

 A convenient observation is flowing from $x$ backwards for time $t$ is the opposite of flowing forwards for time $t$ from $-x$, i.e.
 $\Phi_{-t}(x)=-\Phi_{t}(-x)$.  Let $y=-x$ and notice $(e^{i\pi})^{*}m_{0}=m_{0}$, we have
 \begin{align*}
 &\int_{\Sub UX}\Re q_{\alpha}(\Phi_{-t}(x))\Re q_{\alpha}(x)\Re q_{\beta}(x)\mathrm{d}m_{0}(x)\\
 =&-\int_{\Sub UX}\Re q_{\alpha}(\Phi_{t}(y))\Re q_{\alpha}(y)\Re q_{\beta}(y)\mathrm{d}m_{0}(y).
 \end{align*}

 Therefore 
 \begin{align*}
&\int_{\Sub UX}\Re q_{\alpha}(x)\Re q_{\alpha}(\Phi_{t}(x))\Re q_{\beta}(\Phi_{t}(x))\mathrm{d}m_{0}(x)\\
 =&-\int_{\Sub UX}\Re q_{\alpha}(x)\Re q_{\alpha}(\Phi_{t}(x))\Re q_{\beta}(x)\mathrm{d}m_{0}(x).
 \end{align*}
 This implies
 \begin{align}
 &\sum\limits_{n=0}^{\infty}\bigg(A_n +B_n \bigg) T^{2n+3}(1-T^{2})^{6}\nonumber \\
 &=-B_0 T^{3}(1-T^{2})^{3}. \label{eq comb1}
 \end{align}

 The coefficient of $T^{0}$ yields 
\begin{equation}
  A_{0}+2B_{0}=0 \label{eq 1stComparison}  
\end{equation}

Similarly when flow time $s=\frac{t}{2}$. We let $y=-x$ and again use the fact $(e^{i\pi})^{*}m_{0}=m_{0}$.
\begin{align*}
&\int_{\Sub UX}\Re q_{\alpha}(x)\Re q_{\alpha}(\Phi_{t}(x))\Re q_{\beta}(\Phi_{\frac{1}{2}t}(x))\mathrm{d}m_{0}(x)\\
=&\int_{\Sub UX}\Re q_{\alpha}(\Phi_{-t}(x))\Re q_{\alpha}(x)\Re q_{\beta}(\Phi_{-\frac{1}{2}t}(x))\mathrm{d}m_{0}(x)\\
=&-\int_{\Sub UX}\Re q_{\alpha}(\Phi_{t}(-x))\Re q_{\alpha}(-x)\Re q_{\beta}(\Phi_{\frac{1}{2}t}(-x))\mathrm{d}m_{0}(x)\\
=&-\int_{\Sub UX}\Re q_{\alpha}(\Phi_{t}(y))\Re q_{\alpha}(y)\Re q_{\beta}(\Phi_{\frac{1}{2}t}(y))\mathrm{d}m_{0}(y).
\end{align*}
Thus $\int_{\Sub UX}\Re q_{\alpha}(x)\Re q_{\alpha}(\Phi_{t}(x))\Re q_{\beta}(\Phi_{\frac{1}{2}t}(x))\mathrm{d}m_{0}(x)=0$. 

Recall $t(T)=\frac{1}{2}\log (\frac{1+T}{1-T})$ and $s=\frac{1}{2}\log (\frac{1+S}{1-S})$. In the case $s=\frac{1}{2}t$, we have $T=\frac{2S}{S^{2}+1}$.
The analytic expansion for $\int_{\Sub UX}\Re q_{\alpha}(x)\Re q_{\alpha}(\Phi_{t}(x))\Re q_{\beta}(\Phi_{\frac{1}{2}t}(x))\mathrm{d}m_{0}(x)=0$ with condition $T=\frac{2S}{S^{2}+1}$ simplifies to 
\begin{equation*}
   \sum\limits_{n=0}^{\infty}(A_{n}(S^{2}+1)^{3}+8B_{n})(\frac{2S^{2}}{S^{2}+1})^{n}=0. 
\end{equation*}
Denote $W=\frac{S^2}{S^2+1}$ where $0<W<\frac{1}{2}$. Then the above is equivalent to 
\begin{align*}
    \sum\limits_{n=0}^{\infty}(A_{n}\sum_{k=0}^{\infty}\frac{1}{2}(k+1)(k+2)W^k+8B_{n})2^{n}W^{n}=0. 
\end{align*}
This give relations
\begin{align*}
    2^{n+3}B_{n}+\sum\limits_{k=0}^{n}(n-k+1)(n-k+2)2^{k-1}A_{k}=0, \hspace{.5in} \text{$n\geq 0$.}
\end{align*}
When $n=0$, combining with equation \text{\eqref{eq comb1}}, we obtain $A_{0}=B_{0}=0$. Then equation \text{\eqref{eq comb1}} with right hand side zero yields $A_{n}+B_{n}=0$ for all $n\in \mathbb{N}$. This fact combining with the above formula gives $A_{n}=B_{n}=0$ and equation \eqref{eq 3cubic1} holds for $t,s \geq 0$. 

\item
We then move on to $t < 0$ or $s < 0$, there are three cases we need to discuss.
\begin{itemize}
    \item If $t\leq s$ and $t < 0$, then as $m_{0}$ is $\Phi_{t}$-invariant,
\begin{align*}
&\int_{\Sub UX}\Re q_{\alpha}(x)\Re q_{\alpha}(\Phi_{t}(x))\Re q_{\beta}(\Phi_{s}(x))\mathrm{d}m_{0}(x)\\
=&\int_{\Sub UX}\Re q_{\alpha}(\Phi_{-t}(x))\Re q_{\alpha}(x)\Re q_{\beta}(\Phi_{s-t}(x))\mathrm{d}m_{0}(x).\\
\end{align*}
This is the same as $s,t\geq 0$ case.

   \item If $s< t \leq 0$, then
   
\begin{align*}
&\int_{\Sub UX}\Re q_{\alpha}(x)\Re q_{\alpha}(\Phi_{t}(x))\Re q_{\beta}(\Phi_{s}(x))\mathrm{d}m_{0}(x)\\
=&\int_{\Sub UX}\Re q_{\alpha}(\Phi_{-t}(x))\Re q_{\alpha}(x)\Re q_{\beta}(\Phi_{s-t}(x))\mathrm{d}m_{0}(x)  \\
=&-\int_{\Sub UX}\Re q_{\alpha}(\Phi_{-t}(x))\Re q_{\alpha}(x)\Re q_{\beta}(\Phi_{t-s}(-x))\mathrm{d}m_{0}(x)=0.
\end{align*}
This is from the observation that the analytic expansion of $\Re q_{\beta}(\Phi_{r}(-e^{i\theta}x))$ based at $x$ for $r > 0$ is
\begin{equation*}
   \Re q_{\beta}(\Phi_{r}(-e^{i\theta}x))=\Re  q_{\beta}(\Phi_{r}(e^{i(\theta+\pi)}x))= \tilde{q}_{\beta,x}(Re^{i(\theta+\pi)})=\Re \Bigg(\sum\limits_{n=0}^{\infty}b_{n}(x)R^{n}(1-R^2)^{3}e^{i(n+3)(\theta+\pi)} \Bigg).
\end{equation*}
and that for $ n \geq 0$
\begin{align*}
    e^{-i (n+6) \pi}\int_{UX}\Re(a_{0}a_{n}\bar{b}_{n+3})\mathrm{d}m_{0}&=0,\\
    e^{i (n+3) \pi}\int_{UX}\Re(a_{0}\bar{a}_{n+3}b_{n})\mathrm{d}m_{0}&=0.
\end{align*}

\item If $s<0\leq t $, then we consider 
\begin{align*}
    &\int_{\Sub UX}\Re q_{\alpha}(x)\Re q_{\alpha}(\Phi_{t}(x))\Re q_{\beta}(\Phi_{s}(x))\mathrm{d}m_{0}(x)\nonumber \\
    =&\int_{\Sub UX}\Re q_{\alpha}(\Phi_{t}(-x))\Re q_{\alpha}(x)\Re q_{\beta}(\Phi_{t-s}(-x))\mathrm{d}m_{0}(x)=0. 
\end{align*}
The argument is essentially the same as other cases. This finishes the proof of equation (\ref{eq 3cubic1}).
\end{itemize}
\end{enumerate}

Equation (\ref{eq 3cubic2}) follows easily from equation (\ref{eq 3cubic1}) because of the facts that for all $ t,s \in \mathbb{R}$,
\begin{align*}
&\Re\bigg(\int_{\Sub UX}\Re q_{\alpha}(x) q_{\alpha}(\Phi_{t}(x)) q_{\beta}(\Phi_{s}(x))\mathrm{d}m_{0}(x)\bigg)\\
=&\int_{\Sub UX}\Re q_{\alpha}(x)\Re q_{\alpha}(\Phi_{t}(x))\Re q_{\beta}(\Phi_{s}(x))\mathrm{d}m_{0}(x)-\int_{\Sub UX}\Re q_{\alpha}(x)\Im q_{\alpha}(\Phi_{t}(x))\Im q_{\beta}(\Phi_{s}(x))\mathrm{d}m_{0}(x).
\end{align*}
 
and 
\begin{equation}
\int_{\Sub UX}\Re q_{\alpha}(x) q_{\alpha}(\Phi_{t}(x)) q_{\beta}(\Phi_{s}(x))\mathrm{d}m_{0}(x)=0. 
\label{eq 1Re2q}
\end{equation}
This is easy to see from the fact that $\int_{0}^{2\pi}\Re(a_{0}(x)e^{i3\theta})e^{i(n+3)\theta}e^{i(m+3)\theta}\mathrm{d}\theta=0$ for all $n,m \geq 0$
and thus for $t,s > 0$
\begin{align*}
&\int_{\Sub UX}\Re q_{\alpha}(x) q_{\alpha}(\Phi_{t}(x)) q_{\beta}(\Phi_{s}(x))\mathrm{d}m_{0}(x) \\
=& \frac{1}{2\pi}\int_{0}^{2\pi}\int_{\Sub UX}\Re q_{\alpha}(e^{i\theta}x) q_{\alpha}(\Phi_{t}(e^{i\theta}x)) q_{\beta}(\Phi_{s}(e^{i\theta}x))\mathrm{d}m_{0}(x)\mathrm{d}\theta \\
=&\frac{1}{2\pi}\sum\limits_{m,n\geq 0}\int_{\Sub UX}\int_{0}^{2\pi} \Re(a_{0}(x)e^{i3\theta})\Bigg(a_{n}(x)T^{n}(1-T^2)^{3}e^{i(n+3)\theta} \Bigg)\Bigg(b_{m}(x)S^{m+3}(1-S^2)^{3}e^{i(m+3)\theta} \Bigg)\mathrm{d}\theta\mathrm{d}m_{0}(x)\\
=&0
\end{align*}
The argument for $t\leq0$ or $s\leq0$ can be transferred back to $t>0$ and $s>0$ cases. One needs the observation that $-\Phi_{-t}(-x)=\Phi_{t}(x)$ and $-e^{i\theta}x=e^{i(\theta+\pi)}x$. We conclude equation \eqref{eq 1Re2q} holds for all $t,s \in \mathbb{R}$ and thus equation (\ref{eq 3cubic2}) holds.

\end{proof}

With these preliminaries accomplished, we can now prove Proposition \ref{prop Model}.

\begin{proof}[Proof of Proposition \ref{prop Model}]\hfill

 We start to show $\mathrm{I}=\mathrm{II}=0$.
 
 $\mathrm{I}=0$ reduces to equation \text{\eqref{eq 3cubic1}} of Lemma ($\ref{lem 3cubic}$) if we take $r \to \infty$ for the following 
\begin{align*}
&\frac{1}{r}\int_{\Sub UX}\left(\int_{0}^{r} \Re q_{\alpha}(\Phi_{t}(x))\mathrm{d}t \right)^2\int_{0}^{r}\Re q_{\beta}(\Phi_{t}(x))\mathrm{d}t\mathrm{d}m_{0} \\
=&\frac{1}{r}\int_{0}^{r}\int_{0}^{r}\int_{0}^{r}\int_{\Sub UX}\Re q_{\alpha}(\Phi_{t}(x))\Re q_{\alpha}(\Phi_{s}(x))\Re q_{\beta}(\Phi_{\mu}(x))\mathrm{d}m_{0}\mathrm{d}\mu\mathrm{d}t\mathrm{d}s \hspace{.25in} \text{Fubini's theorem}\\
=&\frac{1}{r}\int_{0}^{r}\int_{0}^{r}\int_{0}^{r}\int_{\Sub UX}\Re q_{\alpha}(\Phi_{t-s}(x))\Re q_{\alpha}(x)\Re q_{\beta}(\Phi_{\mu-s}(x))\mathrm{d}m_{0}\mathrm{d}\mu\mathrm{d}t\mathrm{d}s \hspace{.25in} \text{since $m_{0}$ is $\Phi_{t}$-invariant}\\
=&0.
\end{align*}
We next look into $\mathrm{II}$.
\begin{align*}
\mathrm{II}=&\lim\limits_{r\to \infty}\frac{1}{r}\int_{\Sub UX} 2\int_{0}^{r}\Re q_{\alpha}(\Phi_t(x)) \mathrm{d}t\int_{0}^{r} -\partial_{uv}h(\rho(0))-\partial_{uv}f_{\rho(0)}(\Phi_t(x))\mathrm{d}t\mathrm{d}m_{0} \\ 
=&-\lim\limits_{r\to \infty}\frac{1}{r}\int_{\Sub UX} 2\int_{0}^{r}\Re q_{\alpha}(\Phi_t(x)) \mathrm{d}t\int_{0}^{r} \partial_{uv}h(\rho(0))\mathrm{d}t\mathrm{d}m_{0}\\
-&\lim\limits_{r\to \infty}\frac{1}{r}\int_{\Sub UX} 2\int_{0}^{r}\Re q_{\alpha}(\Phi_t(x)) \mathrm{d}t\int_{0}^{r}Tr(\frac{\partial^2 D_{A(0)}}{\partial u \partial v}\pi(0))(\Phi_t(x))\mathrm{d}t\mathrm{d}m_{0}\\
+ &\lim\limits_{r\to \infty}\frac{1}{r}\int_{\Sub UX} 2\int_{0}^{r}\Re q_{\alpha}(\Phi_t(x))  Tr(\partial_{v}D_{A(0)}\partial_{u}\pi(0))(\Phi_t(x))\mathrm{d}t\mathrm{d}m_{0}.
\end{align*}
There are three terms here. Since $\partial_{uv}h(\rho(0))$ is a constant, the first term is
\begin{align*}
    &\lim\limits_{r\to \infty}\frac{1}{r}\int_{\Sub UX} 2\int_{0}^{r}\Re q_{\alpha}(\Phi_{t}(x)) \mathrm{d}t\int_{0}^{r} \partial_{uv}h(\rho(0))\mathrm{d}t\mathrm{d}m_{0}\\
    =&\lim\limits_{r\to \infty}2\partial_{uv}h(\rho(0))\int_{UX}\int_{0}^{r}\Re q_{\alpha}(\Phi_{t}(x)) \mathrm{d}t\mathrm{d}m_{0}.
\end{align*}
Recall our expressions given by formula (\ref{eq: analyticExI}) and formula (\ref{eq: analyticExIII}). Then
\begin{align*}
&\int_{UX}\int_{0}^{r}\Re q_{\alpha}(\Phi_{t}(x)) \mathrm{d}t\mathrm{d}m_{0}\\
=&\int_{0}^{r}\int_{UX}\Re q_{\alpha}(\Phi_{t}(x))\mathrm{d}m_{0}\mathrm{d}t\\
=&\int_{0}^{r}\int_{UX}\Re q_{\alpha}(x)\mathrm{d}m_{0}\mathrm{d}t \hspace{.35in} \text{since $m_{0}$ is $\Phi_{t}$-invariant}\\
=&\frac{1}{2\pi}\int_{0}^{r}\int_{UX}\int_{0}^{2\pi}\Re q_{\alpha}(e^{i\theta}x)\mathrm{d}\theta \mathrm{d}m_{0} \mathrm{d}t \hspace{.35in} \text{since $m_{0}$ is rotationally invariant} \\
=&\frac{r}{2\pi}\int_{UX}\int_{0}^{2\pi} \Re(a_{0}(x)e^{i3\theta})\mathrm{d}\theta\mathrm{d}m_{0}\\
=&0.
\end{align*}
The second term in $\mathrm{II}$ is 
\begin{align*}
 &-\lim\limits_{r\to \infty}\frac{1}{r}\int_{\Sub UX} 2\int_{0}^{r}\Re q_{\alpha}(\Phi_t(x)) \mathrm{d}t\int_{0}^{r}\Tr(\frac{\partial^2 D_{A(0)}}{\partial u \partial v}\pi(0))(\Phi_t(x)) \mathrm{d}t\mathrm{d}m_{0}\\
 =&\lim\limits_{r\to \infty}\frac{1}{r}\int_{\Sub UX} 2\int_{0}^{r}\Re q_{\alpha}(\Phi_t(x)) \mathrm{d}t\int_{0}^{r} \frac{1}{2}\phi_{uv}(\Phi_t(x)) \mathrm{d}t\mathrm{d}m_{0}, 
\end{align*}
recalling that $\phi$ is a globally well-defined function on $X$ (see formula \eqref{eq:affinePDE1Global}).
$$\frac{1}{2}\phi_{uv}(p(\Phi_{t}(x)))=\frac{1}{2}\phi_{uv}(p(\Phi_{t}(e^{i\theta}x))).$$
So
\begin{align*}
 &\frac{1}{r}\int_{\Sub UX} 2\int_{0}^{r}\Re q_{\alpha}(\Phi_{t}(x)) \mathrm{d}t\int_{0}^{r} \frac{1}{2}\phi_{uv}(\Phi_{t}(x)) \mathrm{d}t\mathrm{d}m_{0}\\
 =&\frac{1}{r}\int_{\Sub UX}\int_{0}^{2\pi} 2\int_{0}^{r}\Re q_{\alpha}(\Phi_{t}(e^{i\theta}x)) \mathrm{d}t\int_{0}^{r} \frac{1}{2}\phi_{uv}(p(\Phi_{t}(e^{i\theta}x))) \mathrm{d}t \mathrm{d}\theta \mathrm{d}m_{0}\\
 =&\frac{1}{r}\int_{\Sub UX}\int_{0}^{2\pi} 2\int_{0}^{r}\Re q_{\alpha}(\Phi_{t}(e^{i\theta}x)) \mathrm{d}t\int_{0}^{r} \frac{1}{2}\phi_{uv}(p(\Phi_{t}(x))) \mathrm{d}t \mathrm{d}\theta \mathrm{d}m_{0}\\
  =&\frac{1}{r} \int_{0}^{r}\int_{0}^{r}\int_{\Sub UX}\phi_{uv}(p(\Phi_{t-s}(x)))\int_{0}^{2\pi}\Re q_{\alpha}(e^{i\theta}x)  \mathrm{d}\theta \mathrm{d}m_{0}\mathrm{d}s\mathrm{d}t.
\end{align*}
Again by the fact $ \int_{0}^{2\pi}\Re q_{\alpha}(e^{i\theta}x)\mathrm{d}\theta
=\int_{0}^{2\pi}\Re(a_{0}(x)e^{i3\theta})\mathrm{d}\theta=0$, we conclude 
\begin{align*}
\lim\limits_{r\to \infty}\frac{1}{r}\int_{\Sub UX} 2\int_{0}^{r}\Re q_{\alpha} \mathrm{d}t\int_{0}^{r}\Tr(\frac{\partial^2 D_{A(0)}}{\partial u \partial v}\pi(0)) \mathrm{d}t\mathrm{d}m_{0}=0.
\end{align*}
It remains to show
\begin{align*}
\lim\limits_{r\to \infty}\frac{1}{r}\int_{\Sub UX} 2\int_{0}^{r}\Re q_{\alpha} \mathrm{d}t\int_{0}^{r} \Tr(\partial_{v}D_{A(0)}\partial_{u}\pi(0))\mathrm{d}t\mathrm{d}m_{0}=0
\end{align*}
This is
\begin{align*}
&\lim\limits_{r\to \infty}\frac{1}{r}\int_{\Sub UX} 2\int_{0}^{r}\Re q_{\alpha}(\Phi_{t}(x)) \mathrm{d}t\int_{0}^{r} \eta(\Phi_{t}(x))\mathrm{d}t\mathrm{d}m_{0}\\
=&-\lim\limits_{r\to \infty}\frac{1}{r}\Bigg(\int_{\Sub UX} 2\int_{0}^{r}\Re q_{\alpha}(\Phi_{t}(x)) \mathrm{d}t\int_{0}^{r}\Re q_{\alpha}(\Phi_{\mu}(x))\int_{0}^{\infty}e^{-2s}\Re q_{\beta}(\Phi_{\mu+s}(x))\mathrm{d}s\mathrm{d}\mu\mathrm{d}m_{0}\\
&+\int_{\Sub UX} 2\int_{0}^{r}\Re q_{\alpha}(\Phi_{t}(x)) \mathrm{d}t\int_{0}^{r}\Re q_{\alpha}(\Phi_{\mu}(x))\int_{-\infty}^{0}e^{2s}\Re q_{\beta}(\Phi_{\mu+s}(x))\mathrm{d}s\mathrm{d}\mu\mathrm{d}m_{0}\\
&+\int_{\Sub UX} 2\int_{0}^{r}\Re q_{\alpha}(\Phi_{t}(x)) \mathrm{d}t\int_{0}^{r}2\Im q_{\alpha}(\Phi_{\mu}(x))\int_{0}^{\infty}e^{-s}\Im q_{\beta}(\Phi_{\mu+s}(x))\mathrm{d}s\mathrm{d}\mu \mathrm{d}m_{0}\\
&+\int_{\Sub UX} 2\int_{0}^{r}\Re q_{\alpha}(\Phi_{t}(x)) \mathrm{d}t\int_{0}^{r}2\Im q_{\alpha}(\Phi_{\mu}(x))\int_{-\infty}^{0}e^{s}\Im q_{\beta}(\Phi_{\mu+s}(x))\mathrm{d}s\mathrm{d}\mu\mathrm{d}m_{0}\bigg).
\end{align*}
We have estimates for the following tail terms,
\begin{align*}
    &\frac{1}{r}\int_{\Sub UX} 2\int_{0}^{r}\Re q_{\alpha}(\Phi_{t}(x)) \mathrm{d}t\int_{0}^{r}\Re q_{\alpha}(\Phi_{\mu}(x))\int_{r}^{\infty}e^{-2s}\Re  q_{\beta}(\Phi_{\mu+s}(x))\mathrm{d}s\mathrm{d}\mu\mathrm{d}m_{0}\\
    +&\frac{1}{r}\int_{\Sub UX} 2\int_{0}^{r}\Re q_{\alpha}(\Phi_{t}(x)) \mathrm{d}t\int_{0}^{r}\Re q_{\alpha}(\Phi_{\mu}(x))\int_{-\infty}^{-r}e^{2s}\Re q_{\beta}(\Phi_{\mu+s}(x))\mathrm{d}s\mathrm{d}\mu\mathrm{d}m_{0}\\
    \leq&\frac{4M^{3}}{r}r^{2}\int_{r}^{\infty}e^{-2s}\mathrm{d}s=2M^{3}r e^{-2r} \xrightarrow{r\to\infty} 0.
\end{align*}
The other two tail terms with integrals involving $\Im q_{\alpha}$ and $\Im q_{\beta}$ also go to zeros for the same reason. So in fact
\begin{align*}
&\lim\limits_{r\to \infty}\frac{1}{r}\int_{\Sub UX} 2\int_{0}^{r}\Re q_{\alpha} \mathrm{d}t\int_{0}^{r} \Tr(\partial_{v}D_{A(0)}\partial_{u}\pi(0))\mathrm{d}t\mathrm{d}m_{0}\\
=&-\lim\limits_{r\to \infty}\frac{1}{r}\Bigg(\int_{\Sub UX} 2\int_{0}^{r}\Re q_{\alpha}(\Phi_{t}(x)) \mathrm{d}t\int_{0}^{r}\Re q_{\alpha}(\Phi_{\mu}(x))\int_{0}^{r}e^{-2s}\Re q_{\beta}(\Phi_{\mu+s}(x))\mathrm{d}s\mathrm{d}\mu\mathrm{d}m_{0}\\
&+\int_{\Sub UX} 2\int_{0}^{r}\Re q_{\alpha}(\Phi_{t}(x)) \mathrm{d}t\int_{0}^{r}\Re q_{\alpha}(\Phi_{\mu}(x))\int_{-r}^{0}e^{2s}\Re q_{\beta}(\Phi_{\mu+s}(x))\mathrm{d}s\mathrm{d}\mu\mathrm{d}m_{0}\\
&+\int_{\Sub UX} 2\int_{0}^{r}\Re q_{\alpha}(\Phi_{t}(x)) \mathrm{d}t\int_{0}^{r}2\Im q_{\alpha}(\Phi_{\mu}(x))\int_{0}^{r}e^{-s}\Im q_{\beta}(\Phi_{\mu+s}(x))\mathrm{d}s\mathrm{d}\mu \mathrm{d}m_{0}\\
&+\int_{\Sub UX} 2\int_{0}^{r}\Re q_{\alpha}(\Phi_{t}(x)) \mathrm{d}t\int_{0}^{r}2\Im q_{\alpha}(\Phi_{\mu}(x))\int_{-r}^{0}e^{s}\Im q_{\beta}(\Phi_{\mu+s}(x))\mathrm{d}s\mathrm{d}\mu\mathrm{d}m_{0}\bigg).
\end{align*}
Similar to $\mathrm{I}$, the above equals to $0$ reduces to equation (\ref{eq 3cubic2}). This finishes our proof of Proposition (\ref{prop Model}) and so concludes the discussion of the model case $\partial_{\beta}g_{\alpha \alpha}(\sigma)$.
\end{proof}

\section{The remaining cases}

We will show in this section the proofs of the remaining three cases, i,e. $\partial_{i}g_{\alpha \alpha}(\sigma)=0$, $\partial_{j}g_{\alpha i}(\sigma)=0$ and $\partial_{\beta}g_{\alpha i}(\sigma)=0$. They provide a complete proof of Theorem \ref{thm main}.

\subsection{The case of $\partial_{i}g_{\alpha \alpha}(\sigma)$}

In this case, given parameters $(u,v)\in \{(-1,1)\}^{2}$, we obtain a family of (conjugacy classes of) representations $\{\rho(u,v)\}$ in $\mathcal{H}_{3}(S)$ corresponding to $\{(vq_{i},uq_{\alpha})\}\subset H^{0}(X,K^2)\bigoplus H^{0}(X,K^3)$ by the Hitchin parametrization. In particular, we have $\partial_{u}\rho(0,0)$ is identifed with $\varphi(q_{\alpha})$ and $\partial_{v}\rho(0,0)$ is identified with $\varphi(q_{i})$. The formula for $\partial_{i}g_{\alpha \alpha}(\sigma)$ is
\begin{align*}
    \partial_i g_{\alpha \alpha}(\sigma)&=\partial_{v} \Bigg( {\langle \partial_{u}\rho(0,v), \partial_{u}\rho(0,v) \rangle}_{\Sub P}\Bigg) \Bigg|_{v=0}\\
    =&\lim\limits_{r\to \infty}\frac{1}{r}\left[\int_{\Sub UX}(\int_{0}^{r} \partial_{u}f_{\rho(0)}^{N}\mathrm{d}t)^2\int_{0}^{r}\partial_{v}f_{\rho(0)}^{N}\mathrm{d}t\mathrm{d}m_{0}+
    2\int_{\Sub UX} \int_{0}^{r}\partial_{u}f_{\rho(0)}^{N} \mathrm{d}t\int_{0}^{r} \partial_{uv}f_{\rho(0)}^{N}\mathrm{d}t\mathrm{d}m_{0}\right].
\end{align*}
where the first and second variations are 
\begin{enumerate}[label=(\roman*)]
    \item $\partial_{u}f_{\rho(0)}^{N}=-\partial_{u}f_{\rho(0)};$
    \item $\partial_{v}f_{\rho(0)}^{N}=-\partial_{v}f_{\rho(0)};$
    \item $\partial_{uv}f_{\rho(0)}^{N}=-\partial_{uv}h(\rho(0))-\partial_{vu}f_{\rho(0)};$
\end{enumerate}
\subsubsection{First and second variations of reparametrization functions}

We compute first and second variations for the case of $\partial_{i}g_{\alpha \alpha}(\sigma)$ in this subsection.

We have Higgs fields 

 \[
  \Phi(u,v)=
  \left[ {\begin{array}{ccc}
   0 & vq_{i} & uq_{\alpha} \\
   1 & 0 & vq_{i} \\
   0 & 1 & 0 \\
  \end{array} } \right].
\]

Following the steps and methods for our model case $\partial_{\beta}g_{\alpha \alpha}(\sigma)$ in section 5, we show in this subsection

\begin{prop}
\label{prop case1rep2}
The first variation of reparametrization functions $\partial_{u}f_{\rho(0)}:UX \to \mathbb{R}$ and $\partial_{v}f_{\rho(0)}:UX \to \mathbb{R}$ for the case $\partial_{i}g_{\alpha \alpha}(\sigma)$ satisfy 
 \begin{align*}
        &\partial_{u}f_{\rho(0)}(x) \sim -Re q_{\alpha}(x),\\ &\partial_{v}f_{\rho(0)}(x) \sim 2Re q_{i}(x).
\end{align*}
and the second variation of reparametrization function $\partial_{vu}f_{\rho(0)}:UX \to \mathbb{R}$ for the case $\partial_{i}g_{\alpha \alpha}(\sigma)$ satisfies
\begin{align*}
  \partial_{uv}f_{\rho(0)}(x) \sim \frac{1}{2}\Re y_{21}(x)-2\Im q_{\alpha}(x)\left(\int_{0}^{\infty}\Im q_{i}(\Phi_{s}(x))e^{-s}\mathrm{d}s+\int_{-\infty}^{0}\Im q_{i}(\Phi_{s}(x))e^{s}\mathrm{d}s\right).
\end{align*}
where $p: UX \to X$ is the projection from the unit tangent bundle $UX$ to our Riemann surface $X$. Understanding a section of $End(E)$ as a  linear map on each fiber of $E=K \bigoplus\mathcal{O}\bigoplus K^{-1}$ over a point of $X$, the element $y_{21}$ is the component of the section $Y=H^{-1}\partial_{uv}H$ that takes $K$ to $\mathcal{O}$. As a function on $UX$,  $y_{21}$ transforms as $y_{21}(e^{i\theta}x)=e^{-i\theta}y_{21}(x)$.

\end{prop}

\begin{proof}
First variation is in \cite{Variation_along_FuchsianLocus}. The computation of the second variation of reparametrization functions $\partial_{uv}f_{\rho(0)}\sim -\Tr(\frac{\partial^2 D_{H(0)}}{\partial u \partial v}\pi(0)) - \Tr(\partial_{u}D_{H(0)}\partial_{v}\pi(0))$ is again divided into computation of  $\Tr(\frac{\partial^2 D_{H(0)}}{\partial u \partial v}\pi(0))$ and computation of $\Tr(\partial_{u}D_{H(0)}\partial_{v}\pi(0))$.

\begin{itemize}
    \item
    Compute $\Tr(\frac{\partial^2 D_{H(0)}}{\partial u \partial v}\pi(0)).$
    
    The major difference between the case $\partial_{i}g_{\alpha \alpha}(\sigma)$ and $\partial_{\beta}g_{\alpha \alpha}(\sigma)$ is the computation of this term.
As before, our flat connection is $D_{H(u,v)}=\nabla_{{\bar{\partial}}_{E},H(u,v)}+\Phi(u,v) + \Phi(u,v)^{*H(u,v)}$. For the computation of $\partial_{u}f_{\rho(0)}$ and $\partial_{v}f_{\rho(0)}$, when $u=0$ or $v=0$, the harmonic metric $H(u,v)$ is diagonal and one obtains
\[
  \partial_{u}D_{H(0)}=
  \begin{bmatrix}
   0 & 0 & q_{\alpha}\\
   0 & 0 & 0 \\
   4\bar{q}_{\alpha} & 0 & 0 \\
  \end{bmatrix}
\]
\[
  \partial_{v}D_{H(0)}=
  \begin{bmatrix}
   0 & q_{i} & 0\\
   2\bar{q}_{i} & 0 & q_{i} \\
   0 & 2\bar{q}_{i} & 0 \\
 \end{bmatrix}
\]

However ,when $u\neq 0, v\neq 0$ both hold, the harmonic metric $H(u,v)$ corresponding to our Higgs field $\Phi(u,v)$ is not diagonal. The computation of $\frac{\partial^2 D_{H(0)}}{\partial u \partial v}$ requires an analysis of the Hitchin's equations.

 We start from a family of Hitchin's equations 
 \begin{equation}
     F_{D_{H(u,v)}}+[\Phi(u,v),\Phi(u,v)^{*H(u,v)}]=0
     \label{eq familyHitchin}
 \end{equation}

We take $u,v$-derivatives of Hitchin's equations \eqref{eq familyHitchin} at $u,v=0$,
 \begin{equation}
     \partial_{u}\partial_{v}|_{u,v=0}( F_{D_{H(u,v)}}+[\Phi(u,v),\Phi(u,v)^{*H(u,v)}])=0.
     \label{eq devFamilyHitchin}
 \end{equation}
 We consider taking $H^{-1}\partial_{vu}H$ as a variable. We define
\[
  Y=H^{-1}\partial_{vu}H=
  \left[ {\begin{array}{ccc}
   y_{11} & y_{12} & y_{13} \\
   y_{21} & y_{22} & y_{23} \\
   y_{31} & y_{32} & y_{33} \\
  \end{array} } \right].
\] 
 
 $Y=H^{-1}\partial_{uv}H$ is a section of $End(E)$.
 
We now work with local coordinates and local trivialization. When varing the $u,v$ real parameters, the holomorphic structure of our bundle $E$ does not change. Thus fixing a local holomorphic frame for all $\{(u,v)\}$, the Chern connection one-form under this frame compatible with the Hermitian metric $H(u,v)$ is $A(u,v)=H(u,v)^{-1}\partial{H(u,v)}$. The curvature term in our holomorphic frame is $$F_{D_{H(u,v)}}=dA(u,v)+A(u,v) \wedge A(u,v)=\bar{\partial}(H(u,v)^{-1}\partial{H(u,v)}).$$
 
The section $Y \in \Gamma(End(E))$ in a local holomorphic frame has the following properties:
\begin{enumerate}[label=(\roman*)]
    \item 
     $\Tr(Y)=0$.
    \item 
    $H(u,v)^{*}=H(u,v)$. Also because $u,v$ are real parameters, we have $\partial_{uv}H=\partial_{uv}(H^{*})=(\partial_{uv}H)^{*}$ and $Y^{*}=HYH^{-1}$.
\end{enumerate} 

we can express $\frac{\partial^2 D_{H(0)}}{\partial u \partial v}$ in terms of $Y$ on $\gamma$. With respect to the local holomorphic frame introduced in the model case adapted to Fermi coordinate, we have $\partial H=0$ on $\gamma$. So 
\begin{align}
    \partial_{uv}(D_{H(0)})|_{\gamma}&=\partial_{uv}(H(u,v)^{-1}\partial H(u,v)+\Phi(u,v)+\Phi(u,v)^{*H(u,v)})|_{u,v=0} \nonumber \\
    &=-YH^{-1}\partial H+ H^{-1}\partial H Y+\partial Y + {\Phi}^{*H}Y-Y{\Phi}^{*H}\nonumber \\
    &=\partial Y + {\Phi}^{*H}Y-Y{\Phi}^{*H}.  \label{eq relationY}
\end{align}   
    
 We want to simplify equation \eqref{eq devFamilyHitchin} as a equation about $Y$ and then solve $Y$ from equation \eqref{eq devFamilyHitchin}.
 
Before we continue, we first fix some notation. We will denote 
$$H=H(0,0),$$
$$\Phi=\Phi(0,0),$$
 $$\partial_{u}H=\frac{\partial H(u,v)}{\partial u}\bigg|_{u,v=0} \text{   },$$ $$\partial_{v}H=\frac{\partial H(u,v)}{\partial v}\bigg|_{u,v=0}\text{  },$$ $$\partial_{uv}H=\frac{\partial H(u,v)}{\partial u \partial v}\bigg|_{u,v=0}\text{  }.$$ 
 As a generalization of the classic result of Ahlfors, the first variation of the harmonic metric vanishes at the Fuchsian point (see \cite[Thm.3.5.1]{Variation_along_FuchsianLocus}). In particular,
 \begin{equation*}
     \partial_{u}H=\partial_{v}H=0.
 \end{equation*}
Taking $H^{-1}\partial_{uv}H$ as a variable, one can verify from equation \eqref{eq devFamilyHitchin} that
 \begin{align}
      0=&\bar{\partial}\partial(H^{-1}\partial_{uv}H)-H^{-1}\partial{H}\wedge\bar{\partial}(H^{-1}\partial_{uv}H)-\bar{\partial}(H^{-1}\partial_{uv}H)\wedge H^{-1}\partial{H}\nonumber \\
     +&\bar{\partial}(H^{-1}\partial H) H^{-1}\partial_{uv}H-H^{-1}\partial_{uv}H\bar{\partial}(H^{-1}\partial H)\nonumber \\
     +&[\partial_{u} \Phi, (\partial_{v}\Phi)^{*H}]+[\partial_{v} \Phi, (\partial_{u}\Phi)^{*H}]+[\Phi,[-H^{-1}\partial_{uv}H,\Phi^{*H}]]
     \label{eq Simp}
 \end{align}

Equation \eqref{eq Simp} can be simplified by the following observation.
\begin{align*}
   &\bar{\partial}(H^{-1}\partial H) H^{-1}\partial_{uv}H-H^{-1}\partial_{uv}H\bar{\partial}(H^{-1}\partial H)\nonumber+[\Phi,[-H^{-1}\partial_{uv}H,\Phi^{*H}]]\\ 
   =&[H^{-1}\partial_{uv}H,[\Phi,\Phi^{*H}]]-[\Phi,[H^{-1}\partial_{uv}H,\Phi^{*H}]]\hspace{.5 in} \text{Hitchin equation}\\
   =&[[H^{-1}\partial_{uv}H, \Phi], \Phi^{*H}] \hspace{.5 in} \text{Jacobi Identity}
\end{align*}

As $Y=H^{-1}\partial_{uv}H$, this yields 
\begin{align}
      &\bar{\partial}\partial Y+[\Phi^{*H},[Y,\Phi]]-H^{-1}\partial H \wedge \bar{\partial}Y-\bar{\partial}Y\wedge H^{-1}\partial H\nonumber \\
    =&-[\partial_{u}\Phi, (\partial_{v}\Phi)^{*H}]-[\partial_{v}\Phi, (\partial_{u}\Phi)^{*H}].  \label{eq HitchinFinal}
\end{align}

The PDE system \eqref{eq HitchinFinal} in local holomorphic frames is equivalent to the following nine scalar equations about $y_{ij}$.

\begin{enumerate}
    \item $\bar{\partial}\partial y_{11}+h(y_{22}-y_{11})=0;$
    \item $\bar{\partial}\partial y_{22}+h(y_{33}-2y_{22}+y_{11})=0;$
    \item $\bar{\partial}\partial y_{33}+h(y_{22}-y_{33})=0;$
    \item $\bar{\partial}\partial y_{21}+h (y_{32}-y_{21})+h^{-1}\partial h \bar{\partial}y_{21}=h^{-2}q_{i}\overline{{q}_{\alpha}};$
    \item $\bar{\partial}\partial y_{32}+h(y_{21}-y_{32})+h^{-1}\partial h \bar{\partial}y_{32}=-h^{-2}q_{i}\overline{q_{\alpha}};$
    \item $\bar{\partial}\partial y_{12}+h(y_{23}-2y_{12})-h^{-1}\partial h \bar{\partial}y_{12}=h^{-1}q_{\alpha}\overline{q_{i}};$
    \item $\bar{\partial}\partial y_{23}+h(y_{12}-2y_{23})-h^{-1}\partial h \bar{\partial}y_{23}=-h^{-1}q_{\alpha}\overline{q_{i}};$
    \item $\bar{\partial}\partial y_{31}+2h^{-1}\partial h \bar{\partial}y_{31}=0;$
    \item $\bar{\partial}\partial y_{13}+2h y_{13}-2h^{-1}\partial h \bar{\partial}y_{13}=0.$
\end{enumerate}

 From porperty (ii) of $Y$, one can thus verify
(4) is equivalent to (6). (5) is equivalent to (7). (8) is equivalent to (9). Thus it suffices to consider the following six equations.
\begin{itemize}
    \item $\bar{\partial}\partial y_{11}+h(y_{22}-y_{11})=0;$
    \item $\bar{\partial}\partial y_{22}+h(y_{33}-2y_{22}+y_{11})=0;$
    \item $\bar{\partial}\partial y_{33}+h(y_{22}-y_{33})=0;$
    \item $\bar{\partial}\partial y_{21}+h(y_{32}-y_{21})+h^{-1}\partial h \bar{\partial}y_{21}=h^{-2}q_{i}\overline{q_{\alpha}};$
    \item $\bar{\partial}\partial y_{32}+h(y_{21}-y_{32})+h^{-1}\partial h \bar{\partial}y_{32}=-h^{-2}q_{i}\overline{q_{\alpha}};$
    \item $\bar{\partial}\partial y_{31}+2h^{-1}\partial h \bar{\partial}y_{31}=0.$
\end{itemize}

We first take a look at the first three equations. We deduce from them
\begin{align*}
    \bar{\partial}\partial (y_{11}+y_{22}+y_{33})&=0;\\
    \bar{\partial}\partial (y_{11}-y_{33})-h(y_{11}-y_{33})&=0;\\
    \bar{\partial}\partial (y_{11}+y_{33})+h(2y_{22}-(y_{11}+y_{33}))&=0.
\end{align*}

As $Y=H^{-1}\partial_{uv}H$ is a section of $End(E)$, the components $y_{ii}\in \Gamma(\mathcal{O})$ are acturally just functions on the surface $X$ for $i=1,2,3$. Recall our notation $\Delta_{\sigma}=\frac{4\partial_{z}\partial_{\bar{z}}}{\sigma}$ and the fact $h=h(0,0)=\frac{1}{2}\sigma$, the above equations can be written independent of coordinate charts on our surface as follows,                                                                                                                                               
\begin{align*}
    \Delta_{\sigma}(y_{11}+y_{22}+y_{33})&=0;\\
    \Delta_{\sigma}(y_{11}-y_{33})-2(y_{11}-y_{33})&=0;\\
    \Delta_{\sigma}(y_{11}+y_{33})+2(2y_{22}-(y_{11}+y_{33}))&=0.
\end{align*}

We have the following observations,
\begin{itemize}
    \item From the first equation, we obtain $y_{11}+y_{22}+y_{33}=C$ where $C$ is a constant.
    \item Since all eigenvalues of $\Delta_{\sigma}$ should be non-positive, the second equation can hold only when $y_{11}-y_{33}=0$.
    \item The third equation is $\Delta_{\sigma}(y_{11}+y_{33})-6(y_{11}+y_{33})=-4C$. By a maximum principle argument, one gets $y_{11}+y_{33}=\frac{2}{3}C$.
\end{itemize}

Thus property (i) of $Y$ gives $y_{11}=y_{22}=y_{33}=0.$

We then continue on the other three equations. From them, we deduce
\begin{align*}
   \bar{\partial}\partial (y_{21}+y_{32})+h^{-1}\partial h \bar{\partial}(y_{21}+y_{32})&=0;\\
    \bar{\partial}\partial (y_{21}-y_{32})-2h(y_{21}-y_{32})+h^{-1}\partial h \bar{\partial}(y_{21}-y_{32})&=2h^{-2}q_{i}\overline{q_{\alpha}};\\
    \bar{\partial}\partial y_{31}+2h^{-1}\partial h \bar{\partial}y_{31}&=0.
\end{align*}
 
Let $w=y_{21}+y_{32}$. We want to compute $\Delta_{h}\norm{w}^{2}_{h}$ where the $h$-norm $\norm{.}_{h}$ is defined as: 
$$\norm{s}^{2}_{h}=h^{-i}s\bar{s}$$ 
for a section $s\in \Gamma(K^{i})$ and $i\in \mathbb{Z}$. 
 
Because $h=h(0,0)=\frac{1}{2}\sigma$ and $\sigma=e^{\delta(z)}|dz|^{2}$ is a hyperbolic metric with curvature $K(\sigma)=-\Delta_{\sigma}(\log \sigma)=-1$. we have h satisfies
\begin{equation}
    \bar{\partial}\partial h=\frac{\partial{h}\bar{\partial}h}{h}+\frac{1}{2}h^{2}. \label{eq metrich}
\end{equation}

Note $w\in \Gamma(K^{-1})$. The metric $h$ induces a Chern connection $\nabla^{h}$ on $K^{-1}$ and in our local holomorphic frames, one has formula: \begin{align*}
\nabla^{h, (1,0)}w=\partial w+h^{-1}\partial{h}w.
\end{align*}

One recognizes $\nabla^{h, (1,0)}w$ is a section of $\Omega^{(1,0)}(K^{-1})=\Gamma(\mathcal{O})$. Therefore,
\begin{align}
    \norm{\nabla^{h,(1,0)}w}^{2}_{h}=(\partial w+h^{-1}\partial{h}w)(\overline{\partial w+h^{-1}\partial{h}w})\label{eq forW}
\end{align}
Combining equation \eqref{eq metrich} and equation \eqref{eq forW} gives
 \begin{align*}
     \Delta_{h}\norm{w}^{2}_{h}&=\frac{4\bar{\partial}\partial(hw\bar{w})}{h}\\
     &=2\norm{w}^{2}_{h}+4\norm{\partial \bar{w}}_{h}+4\norm{\nabla^{h,(1,0)}w}^{2}_{h}\geq 0. 
 \end{align*}
This is an inequality independent of coordinates valid on the Riemann surface. By a maximum principle argument, $\norm{w}_{h}^{2}$ must be a constant $M$. If $M \neq 0$, then $0=\Delta_{h}(M)\geq 2M >0$ leading to a contradiction. Thus $M=0$ and $y_{21}+y_{32}=0$.

We have similar arguments for $\bar{\partial}\partial y_{31}+2h^{-1}\partial h \bar{\partial}y_{31}=0$. We begin with computing $\Delta_{h}\norm{y_{31}}^{2}_{h}$.

Since $y_{31}$ is a section of $\Gamma(K^{-2})$, in lcoal holomorphic frames, the Chern connection $\nabla^{h,}$ induced from $h$ in this case acts as $\nabla^{h,(1,0)}y_{31}=\partial y_{31}+ h^{-2}\partial(h^{2})y_{31}$.

We obtain
 \begin{align*}
     \Delta_{h}\norm{y_{31}}^{2}_{h}&=\frac{4\bar{\partial}\partial(h^{2}y_{31}\overline{y_{31}})}{h}\\
     &=\norm{y_{31}}^{2}_{h}+4\norm{\bar{\partial} y_{31}}_{h}+4\norm{\nabla^{h,(1,0)}y_{31}}^{2}_{h}\geq 0.
 \end{align*}
Similar to the argument for $w$, this leads to $y_{31}=0$. 

We conclude up to this point that $Y=H^{-1}\partial_{vu}H \in \Gamma(End(E))$ in our local frame is of the form
\[
  Y=H^{-1}\partial_{uv}H=
  \left[ {\begin{array}{ccc}
   0 & h\overline{y_{21}} & 0 \\
   y_{21} & 0 & -h\overline{y_{21}} \\
   0 & -y_{21} & 0 \\
  \end{array} } \right]
\]
with $\bar{\partial}\partial y_{21}-2hy_{21}+h^{-1}\partial h \bar{\partial}y_{21}=h^{-2}q_{i}\overline{q_{\alpha}}.$
 
With respect to the Fermi coordinate, we have $h(z)=\frac{1}{2}$ and $\partial_{z}h =0$ on $\gamma$.  Also, we know $Y^{*}=HYH^{-1}$, we finally obtain on $\gamma$ from equation \eqref{eq relationY} ,
\begin{equation*}
    \Tr(\frac{\partial^2 D_{H(0)}}{\partial u \partial v}\pi(0))(x)=\Tr(\frac{\partial^2 D_{H(0)}}{\partial u \partial v}(x)\pi(0))=-\frac{1}{2}\Re y_{21}(x).
\end{equation*}
\begin{rem}
We remark here $y_{21}(x)=y_{21}(z)$ where $x=\dot{\gamma}(0)$ is the starting point of $\gamma$. Recall $y_{21}$ is the component of $Y\in\Gamma(End(E))$ taking $\mathcal{K}$ to $\mathcal{O}$ and $y_{21}(z)$ is  $y_{21}$ evaluating at $p(x)$ in the trivialization given by the holomorphic frame adapted to the Fermi coordinate $z$ for $\gamma$.

In particular, if we consider another closed geodesic $\gamma_{2}$ starting from $\gamma_{2}'(0)=e^{i\theta}x$ with its Fermi coordinate around $\gamma_{2}$ to be $w$, then
$y_{21}(e^{i\theta}x)=y_{21}(w)$. We have $y_{21}(w)=y_{21}(z) \frac{dw}{dz}=y_{21}(z)e^{i\theta}$.

Because the vectors tangent to periodic orbits are dense in $TX$. We can extend $y_{21}$ to be everywhere defined on $UX$. We conclude as a function on $UX$, $y_{21}$ transfers in the following way:
$$y_{21}(e^{i\theta}x)=e^{-i\theta}y_{21}(x).$$

\end{rem}

This finishes the computation of  $ \Tr(\frac{\partial^2 D_{H(0)}}{\partial u \partial v}\pi(0))$ on $UX$. We now move to $\Tr(\partial_{u}D_{H(0)}\partial_{v}\pi(0))$ which together provides an expression for second variation of reparametrization functions.

\item  Compute $\Tr(\partial_{u}D_{H(0)}\partial_{v}\pi(0)).$

We have 
\begin{align*}
 &\Tr(\partial_{u}D_{H(0)}\partial_{v}\pi(0));\\
 =&q_{\alpha}(\partial_{v}a_{11}(0)e_{13}(0)+a_{11}(0)\partial_{v}e_{13}(0))+4\bar{q}_{\alpha}(\partial_{v}a_{31}(0)e_{11}(0)+a_{31}(0)\partial_{v}e_{11}(0)).
\end{align*}
Similar to the model case $\partial_{\beta}g_{\alpha \alpha}$, here $\partial_{v}e_{1}(0)=y$ is the solution of a nonhomogeneous ODE system which arises from taking a $v$-derivative on the system of parallel transport equation \text{\eqref{eq:FamilyOfParallel}} at $v=0$:

\begin{align*}
  \partial_{t}\left[ \begin{array}{c} y_1(t) \\ y_2(t) \\ y_3(t) \end{array} \right] + \begin{bmatrix} 0 & \frac{1}{2} & 0\\ 1 & 0 & \frac{1}{2} \\ 0 & 1 &0 \end{bmatrix}  \left[ \begin{array}{c} y_{1}(t) \\ y_2(t) \\ y_3(t) \end{array} \right]= \frac{\sqrt{2}}{2}e^{t} \left[ \begin{array}{c} q_{i}(\Phi_{t}(x)) \\ -2\Re q_{i}(\Phi_{t}(x)) \\ 2\overline{q_{i}(\Phi_{t}(x))} \end{array} \right].
\end{align*}

with boundary conditions  
\begin{align*}
    &H(y(0),e_{1}(0,0))=0;\\
   &y(l_{\gamma})=e^{l_{\gamma}}\left(\int_{0}^{l_{\gamma}}2\Re q_{i}(\Phi_{s}(x))\mathrm{d}s\right) e_{1}(0,0)+e^{l_{\gamma}}y(0).
\end{align*}
The boundary conditions are set up based on the same consideration as the case of $\partial_{\beta}g_{\alpha \alpha}(\sigma)$.
The solution is 
\begin{align*}
  \left[\begin{array}{c} \partial_{v}e_{11}(t)\\ \partial_{v}e_{12}(t) \\ \partial_{v}e_{13}(t) \end{array}  \right]&=\left[\begin{array}{c} \frac{\sqrt{2}}{2}\int_{0}^{t}(e^{t}\Re q_{i}+i e^{s}\Im q_{i})\mathrm{d}s\\[1ex]
 -\sqrt{2}\int_{0}^{t}e^{t}\Re q_{i}\mathrm{d}s \\[1ex] \sqrt{2}\int_{0}^{t}(e^{t}\Re q_{i}-i e^{s}\Im q_{i})\mathrm{d}s \end{array}  \right]\\[1ex]
  &+\left[\begin{array}{c}   \frac{\sqrt{2}}{2}(e^{l_{\gamma}}-1)^{-1}\int_{0}^{l_{\gamma}}ie^{s} \Im q_{i} \mathrm{d}s \\[1ex] 0\\[1ex]
 -\sqrt{2}(e^{l_{\gamma}}-1)^{-1}\int_{0}^{l_{\gamma}}ie^{s} \Im q_{i} \mathrm{d}s \end{array}  \right].
\end{align*}

Similarly, one can compute $\partial_{v}e_{2}(0)$ and $\partial_{v}e_{3}(0)$ by this method.
It turns out that
\begin{align*}
    &\Tr(\partial_{u}D_{H(0)}\partial_{v}\pi(0))(\Phi_{t}(x))\\
    =&2\Im q_{\alpha}(\Phi_{t}(x))\int_{0}^{t}(e^{s-t}-e^{t-s})\Im q_{i}(\Phi_{s}(x))\mathrm{d}s\\
    +&2\Im q_{\alpha}(\Phi_{t}(x))\int_{0}^{l_{\gamma}}(\frac{e^{s-t}}{e^{l_{\gamma}}-1}-\frac{e^{t-s}}{e^{-l_{\gamma}}-1})\Im q_{i}(\Phi_{s}(x))\mathrm{d}s.
\end{align*}

We therefore obtain, for a closed geodesic $\gamma$ of length $l_{\gamma}$ starting from $\dot{\gamma(0)}=x$,

\begin{align*}
    &\Tr(\partial_{u}D_H(0)\partial_{v}\pi(0))(x)\\
    =&2\Im q_{\alpha}(x)\int_{0}^{l_{\gamma}}(\frac{e^{s}}{e^{l_{\gamma}}-1}-\frac{e^{-s}}{e^{-l_{\gamma}}-1})\Im q_{i}(\Phi_{s}(x))\mathrm{d}s.
\end{align*}
\end{itemize}

Similar to our model case of $g_{\alpha\alpha,\beta}(\sigma)$, one can define a function $\eta: W\to \mathbb{R}$,
\begin{align*}
    \eta(x)=2\Im q_{\alpha}(x)\left(\int_{0}^{\infty}e^{-s}\Im q_{i}(\Phi_{s}(x))\mathrm{d}s+\int_{-\infty}^{0}e^{s}\Im q_{i}(\Phi_{s}(x))\mathrm{d}s\right).
\end{align*}
and we verify that $\eta(x)$ is H{\"o}lder such that $\Tr(\partial_{u}D_{H(0)}\partial_{v}\pi(0))(x)\equiv \eta(x)$ on $UX$.

We conclude
\begin{align*}
  \partial_{uv}f_{\rho(0)}(x) \sim &-\partial_{v}(\Tr(\partial_{u}D_{H(0)} \pi(0)))(x)\\
  =&\frac{1}{2}\Re y_{21}(x)-2\Im q_{\alpha}(x)\left(\int_{0}^{\infty}e^{-s}\Im q_{i}(\Phi_{s}(x))\mathrm{d}s+\int_{-\infty}^{0}e^{s}\Im q_{i}(\Phi_{s}(x))\mathrm{d}s\right).
\end{align*}
This finishes the proof of Proposition \ref{prop case1rep2}.
\end{proof}

\begin{rem} \hfill

Instead of starting from the first variation of reparametrization functions  $\partial_{u}f_{\rho(0)}(x)\sim -\Tr(\partial_{u}D_{H(0)} \pi(0)))(x)$, we can take the first variation of reparametrization functions to be $\partial_{v}f_{\rho(0)}(x)\sim -\Tr(\partial_{v}D_{H(0)} \pi(0)))(x)$ by formula (\ref{eq firstVar}) and consider:
\begin{align*}
\partial_{vu}f_{\rho(0)}(x)\sim&-\partial_{u}(\Tr(\partial_{v}D_{H(0)} \pi(0)))(x)\\
=&-\Tr(\frac{\partial^2 D_{H(0)}}{\partial v \partial u}\pi(0))(x) - \Tr(\partial_{v}D_{H(0)}\partial_{u}\pi(0))(x).
\end{align*}
By the same method, we get
\begin{align*}
    &\Tr(\partial_{v}D_{H(0)}\partial_{u}\pi(0))(\Phi_{t}(x))\\
    =&2\Im q_{i}(\Phi_{t}(x))\int_{0}^{t}(e^{s-t}-e^{t-s})\Im q_{\alpha}(\Phi_{s}(x))\mathrm{d}s\\
    +&2\Im q_{i}(\Phi_{t}(x))\int_{0}^{l_{\gamma}}(\frac{e^{s-t}}{e^{l_{\gamma}}-1}-\frac{e^{t-s}}{e^{-l_{\gamma}}-1})\Im q_{\alpha}(\Phi_{s}(x))\mathrm{d}s.
\end{align*}
One can verify by Fubini's theorem,
\begin{equation*}
    \int_{0}^{l_{\gamma}}\Tr(\partial_{u}D_{H(0)}\partial_{v}\pi(0))(\Phi_{t}(x))\mathrm{d}t= \int_{0}^{l_{\gamma}}\Tr(\partial_{v}D_{H(0)}\partial_{u}\pi(0))(\Phi_{t}(x))\mathrm{d}t.
\end{equation*}
This coincides with the fact that $\partial_{v}(\Tr(\partial_{u}D_{H(0)} \pi(0)))(x)$ and $\partial_{u}(\Tr(\partial_{v}D_{H(0)} \pi(0)))(x)$ should be in the same Liv\v{s}ic class by Liv\v{s}ic's Theorem.
\end{rem}

\subsubsection{Evaluation on Poincar\'e disk}

With the computation in last section, we have
\begin{align*}
    \partial_i g_{\alpha \alpha}(\sigma)=&\partial_{v} \Bigg( {\langle \partial_{u}\rho(0,v), \partial_{u}\rho(0,v) \rangle}_{\Sub P}\Bigg) \Bigg|_{v=0}\\
    =&\lim\limits_{r\to \infty}\frac{1}{r}\left[\int_{\Sub UX}(\int_{0}^{r} Re q_{\alpha}\mathrm{d}t)^2\int_{0}^{r}-2Re q_{i}\mathrm{d}t\mathrm{d}m_{0}-2\int_{\Sub UX} \int_{0}^{r}Re q_{\alpha} \mathrm{d}t\int_{0}^{r} \partial_{uv}f_{\rho(0)}^{N}\mathrm{d}t\mathrm{d}m_{0}\right],
\end{align*}
where $\partial_{uv}f_{\rho(0)}^{N}=-\partial_{uv}h(\rho(0))-\partial_{vu}f_{\rho(0)}$ and
\begin{align*}
  \partial_{uv}f_{\rho(0)}(x)\sim \frac{1}{2}\Re y_{21}(x)-2\Im q_{\alpha}(x)\left(\int_{0}^{\infty}\Im q_{i}(\Phi_{s}(x))e^{-s}\mathrm{d}s+\int_{-\infty}^{0}\Im q_{i}(\Phi_{s}(x))e^{s}\mathrm{d}s\right).
\end{align*}
We show in this subsection:

\begin{prop}
\label{prop Case2}
 For $\sigma\in\mathcal{T}(S),\partial_{i} g_{\alpha \alpha}(\sigma)=0$.
\end{prop}

The argument for this proposition boils down to the following lemma. 
\begin{lem} \label{lem 3cubicII}
We have the following holds for any $t,s\in \mathbb{R}$,
\begin{align}
&\int_{\Sub UX}\Re q_{i}(x)\Re q_{\alpha}(\Phi_{t}(x))\Re q_{\alpha}(\Phi_{s}(x))\mathrm{d}m_{0}(x)=0. \label{eq 3cubic3}\\
&\int_{\Sub UX}\Re q_{\alpha}(x)\Re q_{\alpha}(\Phi_{t}(x))\Re q_{i}(\Phi_{s}(x))\mathrm{d}m_{0}(x)=0. \label{eq 3cubic4}\\
 &\int_{\Sub UX}\Re q_{\alpha}(x)\Im q_{\alpha}(\Phi_{t}(x))\Im q_{i}(\Phi_{s}(x))\mathrm{d}m_{0}(x)=0. \label{eq 3cubic5}
\end{align}
\end{lem}

\begin{proof}
The proof of this lemma is basically the same as the proof of Lemma \ref{lem 3cubic} except that flow time $s=\frac{1}{2}t$ tells us nothing in this case. We instead choose flow time to be the following three special cases: $s=t$, $s=2t$ and $s=3t$.
We recall our analytic expansions for $q_{i}$ and $q_{\alpha}$ are:
\begin{align*}
   q_{i}(\Phi_{r}(e^{i\theta}x))&=\Bigg(\sum\limits_{n=0}^{\infty}c_{n}(x)R^{n}(1-R^2)^{2}e^{i(n+2)\theta} \Bigg),\\
    q_{\alpha}(\Phi_{r}(e^{i\theta}x))&=\Bigg(\sum\limits_{n=0}^{\infty}a_{n}(x)R^{n}(1-R^2)^{3}e^{i(n+3)\theta} \Bigg).
\end{align*}

We have when $t,s > 0$
\begin{align*}
&\int_{\Sub UX}\Re q_{i}(x)\Re q_{\alpha}(\Phi_{t}(x))\Re q_{\alpha}(\Phi_{s}(x))\mathrm{d}m_{0}(x)\\
=&\frac{1}{2\pi}\int_{0}^{2\pi}\int_{\Sub UX}\Re q_{i}(e^{i\theta}x)\Re q_{\alpha}(\Phi_{t}(e^{i\theta}x))\Re q_{\alpha}(\Phi_{s}(e^{i\theta}x))\mathrm{d}m_{0}(x)\mathrm{d}\theta\\
=&\frac{1}{4}\sum\limits_{n=0}^{\infty}\int_{UX}\Re(c_{0}a_{n}\bar{a}_{n+2})\mathrm{d}m_{0}T^{n}S^n(1-T^{2})^{3}(1-S^{2})^{3}(S^2+T^2).
\end{align*}

Consider $t=s>0$. Then
\begin{align*}
&\int_{\Sub UX}\Re q_{i}(x)\Re q_{\alpha}(\Phi_{t}(x))\Re q_{\alpha}(\Phi_{t}(x))\mathrm{d}m_{0}(x)\\
=&\int_{\Sub UX}\Re q_{i}(\Phi_{-t}(x))\Re q_{\alpha}(x)\Re q_{\alpha}(x)\mathrm{d}m_{0}(x)\\
=&\int_{\Sub UX}\Re q_{i}(-\Phi_{t}(-x))\Re q_{\alpha}(-x)\Re q_{\alpha}(-x)\mathrm{d}m_{0}(x)\\
=&\int_{\Sub UX}\Re q_{i}(\Phi_{t}(y))\Re q_{\alpha}(y)\Re q_{\alpha}(y)\mathrm{d}m_{0}(y) \hspace{.5in}\text{with $y=-x$}
\end{align*}

The analytic expansions of left and right hand sides of the above equation gives

\begin{align}
&\frac{1}{2}\sum\limits_{n=0}^{\infty}\int_{UX}\Re(c_{0}a_{n}\bar{a}_{n+2})\mathrm{d}m_{0}T^{2n}(1-T^{2})^{6}T^2 \nonumber \\
=&\frac{1}{4}\int_{UX}\Re(a_{0}a_{0}\bar{c}_{4})\mathrm{d}m_{0}(1-T^{2})^{2}T^{4} \label{eq comb2}
\end{align}

We denote $C_{n}=\int_{UX} \Re(c_{0}a_{n}\bar{a}_{n+2})\mathrm{d}m_{0}$ and $D_{n}=\int_{UX} \Re(a_{0}a_{n}\bar{c}_{n+4})\mathrm{d}m_{0}$ for $n \geq 0$. We proceed to prove $C_n=0$ for $n \geq 0$. 

The coefficient of $T^0$ and $T^2$ yield the following equations
\begin{align*}
     C_0&=0,\\
    2C_1-8C_0&=D_0.
\end{align*}

 On the other hand, if we consider $s=2t$ and $s=3t$. They lead to the following two equations

 \begin{align*}
&\int_{\Sub UX}\Re q_{i}(x)\Re q_{\alpha}(\Phi_{t}(x))\Re q_{\alpha}(\Phi_{2t}(x))\mathrm{d}m_{0}(x)\\
=&-\int_{\Sub UX}\Re q_{i}(\Phi_{t}(x))\Re q_{\alpha}(x)\Re q_{\alpha}(\Phi_{t}(-x))\mathrm{d}m_{0}(y)
\end{align*}
 and 
 \begin{align*}
 &\int_{\Sub UX}\Re q_{i}(x)\Re q_{\alpha}(\Phi_{t}(x))\Re q_{\alpha}(\Phi_{3t}(x))\mathrm{d}m_{0}(x)\\
 =&-\int_{\Sub UX}\Re q_{i}(\Phi_{t}(x))\Re q_{\alpha}(x)\Re q_{\alpha}(\Phi_{2t}(-x))\mathrm{d}m_{0}(y).
 \end{align*}

 When $s=2t$, we have $S=\frac{2T}{T^2+1}$ and $S=2T+O(T^3)$.  When $s=3t$, we have $S=\frac{3T+T^3}{3T^2+1}$ and $S=3T+O(T^3)$.

Compare coefficients of $T^4$ of the analytic expansions of above two equations and use the relations $S=2T+O(T^3)$ and $S=3T+O(T^3)$ to obtain $D_0=0$. Therefore from equation \text{\eqref{eq comb2}} we conclude $C_{n}=0$ for $n\geq 0$ and  
equation \text{\eqref{eq 3cubic3}} holds for $t,s>0$.  For $s\leq 0$ or $t \leq 0$, the argument for equation \text{\eqref{eq 3cubic3}} to hold is an analogy of $\partial_{\beta}g_{\alpha \alpha}(\sigma)$ case. We omit it here.

Equation \text{\eqref{eq 3cubic4}} then follows from equation \text{\eqref{eq 3cubic3}} by a $\Phi_t$-invariance argument of $m_{0}$. To prove equation \text{\eqref{eq 3cubic5}}, we just need the following 
\begin{align*}
\int_{\Sub UX}\Re q_{\alpha}(x) q_{i}(\Phi_{t}(x)) q_{i}(\Phi_{s}(x))\mathrm{d}m_{0}(x)=0 \hspace{.5in}\text{$\forall t,s \in \mathbb{R}$}
\end{align*}
The argument is the same as the argument for Lemma \ref{lem 3cubic}. This finishes the proof of Lemma \ref{lem 3cubicII}.

\end{proof}

\begin{proof}[Proof of Proposition \ref{prop Case2}]\hfill

We begin by showing the following is zero by evaluation the integral on Poincar\'e disk.
\begin{align*}
\lim\limits_{r\to \infty}\frac{1}{r}\int_{\Sub UX} 2\int_{0}^{r}\partial_{u}f^{N}_{\rho(0)} \mathrm{d}t\int_{0}^{r}\Tr(\frac{\partial^2 D_{H(0)}}{\partial u \partial v}\pi(0)) \mathrm{d}t\mathrm{d}m_{0}=0.
\end{align*}
 Recall from the last subsection that $y_{21}$ is the solution of $\bar{\partial}\partial y_{21}-2hy_{21}+h^{-1}\partial h \bar{\partial}y_{21}=h^{-2}q_{i}\overline{q_{\alpha}}$. Because $q_{i}$ and $\overline{q_{\alpha}}$ are real analytic and because $h=h(0,0)=\frac{1}{2}\sigma$ is also real analytic, we know $y_{21}$ is real analytic by analytic elliptic regularity theory (\cite{Gilbarg_Elliptic}).

As discussed before, the function $y_{21}$ on $UX$ transfers as  $y_{21}(e^{i\theta}x)=e^{-i\theta}y_{21}(x)$. Similarly to the model case of $g_{\alpha \alpha, \beta}$, we write the real analytic expansion for $y_{21}$ in the coordinates given by Poincar\'e disck model based on $x$,
\begin{equation*}
    y_{21,x}(z)=\sum\limits_{n,m\geq 0}b_{n,m}(x)z^{n}{\bar{z}}^{m}\frac{\partial}{\partial z}.
\end{equation*}
Denote $\tilde{y}_{21,x}(z):=\Re \left(y_{21,x}(z)(dr)\right)$. Recall $r(R)=\frac{1}{2}\log(\frac{1-R}{1+R})$. one has
\begin{align*}
   y_{21}(\Phi_{r}(e^{i\theta}x))= \tilde{y}_{21,x}(Re^{i\theta})=\Re \Bigg(\sum\limits_{n,m\geq0}b_{n,m}(x)R^{n+m}(1-R^2)^{-1}e^{i(n-m-1)\theta} \Bigg).
\end{align*}
Thus
\begin{align*}
&\lim\limits_{r\to \infty}\frac{1}{r}\int_{\Sub UX} 2\int_{0}^{r}\partial_{u}f^{N}_{\rho(0)} \mathrm{d}t\int_{0}^{r}Tr(\frac{\partial^2 D_{H(0)}}{\partial u \partial v}\pi(0)) \mathrm{d}t\mathrm{d}m_{0}\\
=&\lim\limits_{r\to \infty}\frac{1}{r}\int_{\Sub UX} \int_{0}^{r}Re q_{\alpha}(\Phi_{t}(x)) \mathrm{d}t\int_{0}^{r} \Re y_{21}(\Phi_{t}(x)) \mathrm{d}t\mathrm{d}m_{0}\\
=&\lim\limits_{r\to \infty}\frac{1}{r}\int_{0}^{r}\int_{0}^{r}\int_{\Sub UX} \Re q_{\alpha}(\Phi_{t}(x))\Re y_{21}(\Phi_{s}(x))\mathrm{d}m_{0}\mathrm{d}t\mathrm{d}s\\
=&\lim\limits_{r\to \infty}\frac{1}{r}\int_{0}^{r}\int_{0}^{r}\int_{\Sub UX} \Re q_{\alpha}(\Phi_{t-s}(x))\Re y_{21}(x)\mathrm{d}m_{0}\mathrm{d}t\mathrm{d}s.
\end{align*}
When $\mu=t-s \geq0$,
\begin{align*}
    &\int_{\Sub UX} \Re q_{\alpha}(\Phi_{\mu}(x))\Re y_{21}(x)\mathrm{d}m_{0}\\
    =&\frac{1}{2\pi}\int_{\Sub UX} \int_{0}^{2\pi}\Re q_{\alpha}(\Phi_{\mu}(e^{i\theta}x))\Re y_{21}(e^{i\theta}x) \mathrm{d}\theta \mathrm{d}m_{0}.
\end{align*}
However $\int_{0}^{2\pi} \Re(e^{-i\theta}b_{0,0}) \Re(a_{n}e^{i(n+3)\theta})\mathrm{d}\theta=0$ for $\forall n\geq 0$ implies the above is zero.

It also holds for $\mu \leq 0$ by a simple observation that $\Re q_{\alpha}(\Phi_{-\mu}(-x))=-\Re q_{\alpha}(\Phi_{\mu}(x))$. Therefore we conclude 
\begin{align*}
\lim\limits_{r\to \infty}\frac{1}{r}\int_{\Sub UX} 2\int_{0}^{r}\partial_{u}f^{N}_{\rho(0)} \mathrm{d}t\int_{0}^{r}\Tr(\frac{\partial^2 D_{H(0)}}{\partial u \partial v}\pi(0)) \mathrm{d}t\mathrm{d}m_{0}=0.
\end{align*}

Arguments for other terms in $\partial_{i}g_{\alpha \alpha}(\sigma)$ to be equal to zero are analogous to the model case of  $\partial_{\beta}g_{\alpha \alpha}(\sigma)$. They all reduce to Lemma \ref{lem 3cubicII}. We therefore finish the proof of Proposition \ref{prop Case2}. 
\end{proof}

\subsection{The case of $\partial_{j}g_{\alpha i}(\sigma)$}

The proofs for the case of $\partial_{j}g_{\alpha i}(\sigma)$ in this subsection and the case of $\partial_{\beta}g_{\alpha i}(\sigma)$ in the next subsection are basically the same as the cases for $\partial_{\beta}g_{\alpha \alpha}(\sigma)$ and $\partial_{i}g_{\alpha \alpha}(\sigma)$. Although there's no more new ingredients in the proofs, we include them here for completeness.

For $\partial_{j}g_{\alpha i}(\sigma)$, we have three parameters $\{(u,v,w)\}\in \{(-1,1)\}^{3}$. The representations $\{\rho(u,v,w)\}$ in $\mathcal{H}_{3}(S)$ corresponds to $\{(vq_{i}+wq_{j},uq_{\alpha})\}\subset H^{0}(X,K^2)\bigoplus H^{0}(X,K^3)$ by Hitchin parametrization. In particular, we have $\partial_{u}\rho(0,0,0)$ is identified with $\varphi(q_{\alpha})$ and $\partial_{v}\rho(0,0,0)$ is identified with $\varphi(q_{i})$. Also  $\partial_{w}\rho(0,0,0)$ is identified with $\varphi(q_{j})$. The formula for $\partial_{j}g_{\alpha i}(\sigma)$ is
\begin{align*}
\partial_{j}g_{\alpha i}(\sigma)
=&\partial_{w} \Bigg( {\langle \partial_{u}\rho(0,0,w), \partial_{v}\rho(0,0,w) \rangle}_{\Sub P}\Bigg) \Bigg|_{w=0}\\
 =&\lim\limits_{r\to \infty}\frac{1}{r}\Bigg(\int_{\Sub UX}\int_{0}^{r} \partial_{u}f_{\rho(0)}^{N}\mathrm{d}t \int_{0}^{r}\partial_{v}f_{\rho(0)}^{N}\mathrm{d}t\int_{0}^{r} \partial_{w}f_{\rho(0)}^{N}\mathrm{d}t\mathrm{d}m_{0}\\
 +&\int_{\Sub UX} \int_{0}^{r}\partial_{u}f_{\rho(0)}^{N} \mathrm{d}t\int_{0}^{r} \partial_{vw}f_{\rho(0)}^{N}\mathrm{d}t\mathrm{d}m_{0}+\int_{\Sub UX} \int_{0}^{r}\partial_{v}f_{\rho(0)}^{N} \mathrm{d}t\int_{0}^{r} \partial_{uw}f_{\rho(0)}^{N}\mathrm{d}t\mathrm{d}m_{0}\Bigg)
 \end{align*}
where the first and second variations are 
\begin{enumerate}[label=(\roman*)]
    \item $\partial_{u}f_{\rho(0)}^{N}=-\partial_{u}f_{\rho(0)}$.
    \item $\partial_{v}f_{\rho(0)}^{N}=-\partial_{v}f_{\rho(0)}$.
    \item $\partial_{uw}f_{\rho(0)}^{N}=-\partial_{uw}h(\rho(0))-\partial_{uw}f_{\rho(0)}$.
    \item $\partial_{vw}f_{\rho(0)}^{N}=-\partial_{vw}h(\rho(0))-\partial_{vw}f_{\rho(0)}$.
\end{enumerate}

\subsubsection{First and second variations of reparametrization functions}

Our Higgs fields in this case are 

 \[
  \Phi(u,v,w)=
 \begin{bmatrix}
   0 & vq_{i}+wq_{j}  & uq_{\alpha} \\
   1 & 0 & vq_{i}+wq_{j} \\
   0 & 1 & 0 \\
  \end{bmatrix}
\]

Follow the steps and methods from cases $\partial_{\beta}g_{\alpha \alpha}(\sigma)$ and $\partial_{\beta}g_{\alpha \alpha}(\sigma)$, we have

\begin{prop}
\label{prop case1rep3}
The first variations of reparametrization functions $\partial_{u}f_{\rho(0)}:UX \to \mathbb{R}$, $\partial_{v}f_{\rho(0)}:UX \to \mathbb{R}$ and $\partial_{w}f_{\rho(0)}:UX \to \mathbb{R}$  for the case $\partial_{j}g_{\alpha i}(\sigma)$ satisfy 
 \begin{align*}
  \partial_{u}f_{\rho(0)}(x)& \sim -Re q_{\alpha}(x),\\  
  \partial_{v}f_{\rho(0)}(x)&\sim 2Re q_{i}(x),\\
   \partial_{w}f_{\rho(0)}(x)&\sim 2Re q_{j}(x).
\end{align*}
and the second variations of reparametrization functions $\partial_{uw}f_{\rho(0)}:UX \to \mathbb{R}$ and $\partial_{vw}f_{\rho(0)}:UX \to \mathbb{R}$ satisfy
\begin{align*}
&\partial_{uw}f_{\rho(0)}\sim\frac{1}{2}\Re y_{21}(x)-2\Im q_{\alpha}(x)\left(\int_{0}^{\infty}\Im q_{i}(\Phi_{s}(x))e^{-s}\mathrm{d}s+\int_{-\infty}^{0}\Im q_{i}(\Phi_{s}(x))e^{s}\mathrm{d}s\right)\\
 &\partial_{vw}f_{\rho(0)}(x)\sim \frac{1}{2}\phi_{vw}(p(x))+2\Im q_{i}(x)\left(\int_{0}^{\infty}\Im q_{j}(\Phi_{s}(x))e^{-s}\mathrm{d}s+\int_{-\infty}^{0}\Im q_{j}(\Phi_{s}(x))e^{s}\mathrm{d}s\right)
\end{align*}
where $p: UX \to X$ and $y_{21}$ are defined as before.

\end{prop}

\begin{proof}
For the second variations of reparametrization functions, we have computed $\partial_{uw}f_{\rho(0)}$ in the $\partial_{i}g_{\alpha \alpha}(\sigma)$ case:
\begin{align*}
\partial_{uw}f_{\rho(0)}\sim\frac{1}{2}\Re y_{21}(x)-2\Im q_{\alpha}(x)\left(\int_{0}^{\infty}\Im q_{i}(\Phi_{s}(x))e^{-s}\mathrm{d}s+\int_{-\infty}^{0}\Im q_{i}(\Phi_{s}(x))e^{s}\mathrm{d}s\right)
\end{align*}
The computation of $\partial_{vw}f_{\rho(0)}\sim-\Tr(\frac{\partial^2 D_{A(0)}}{\partial w \partial v}\pi(0)) - \Tr(\partial_{v}D_{H(0)}\partial_{w}\pi(0))$ is divided into computation of $\Tr(\frac{\partial^2 D_{A(0)}}{\partial w \partial v}\pi(0))$ and computation of $\Tr(\partial_{v}D_{H(0)}\partial_{w}\pi(0))$.

\begin{itemize}
    \item  Compute $\Tr(\frac{\partial^2 D_{H(0)}}{\partial v \partial w}\pi(0))$

 we set $u=0$, the Higgs field is
\[
  \Phi(v,w)=
  \begin{bmatrix}
   0 & vq_{i}+wq_{j}  & 0 \\
   1 & 0 & vq_{i}+wq_{j} \\
   0 & 1 & 0 \\
  \end{bmatrix}
\]

The harmonic metric $H(v,w)$ is diagonalizable and the computation of $\partial_{vw}D_{H(0)}$ is the same as the model case of $\partial_{\beta}g_{\alpha \alpha}(\sigma)$.

With respect to the notation defined in the model case of $\partial_{\beta}g_{\alpha \alpha}(\sigma)$, one obtains:

\begin{align*}
   \Tr(\frac{\partial^2 D_{H(0)}}{\partial v \partial w}\pi(0))(x)=-\frac{1}{2}\psi_{vw}(z(p(x)))=-\frac{1}{2}\phi_{vw}(p(x))
\end{align*}
Where $p: UX \to X$ is the projection from the unit tangent bundle to our surface and $z$ is the Fermi coordinate we choose evaluating at the point $p(x)\in X$. 

\item  Compute $\Tr(\partial_{v}D_{H(0)}\partial_{w}\pi(0))$

Both $\partial_{v}D_{H(0)}$ and $\partial_{w}\pi(0)$ have been computed in $\partial_{i}g_{\alpha \alpha}(\sigma)$ case. One can check
\begin{align*}
   & \Tr(\partial_{v}D_{H(0)}\partial_{w}\pi(0))(\Phi_{t}(x))\\
   =&q_{i}(\partial_{w}a_{11}(0)e_{12}(0)+a_{11}(0)\partial_{w}e_{12}(0)+\partial_{w}a_{21}(0)e_{13}(0)+a_{21}(0)\partial_{w}e_{13}(0))\\
   +&2\bar{q}_{i}(\partial_{w}a_{21}(0)e_{11}(0)+a_{21}(0)\partial_{w}e_{11}(0)+\partial_{w}a_{31}(0)e_{12}(0)+a_{31}(0)\partial_{w}e_{12}(0))\\
   =&2\Im q_{i}(\Phi_{t}(x))\int_{0}^{t}\Im q_{j}(\Phi_{s}(x))(e^{t-s}-e^{s-t})\mathrm{d}s\\
    +&2\Im q_{i}(\Phi_{t}(x))\int_{0}^{l_{\gamma}}\Im q_{j}(\Phi_{s}(x))(\frac{e^{t-s}}{e^{-l_{\gamma}}-1}-\frac{e^{s-t}}{e^{l_{\gamma}}-1})\mathrm{d}s
\end{align*}
In particular, 
\begin{align*}
    &\Tr(\partial_{v}D_{H(0)}\partial_{w}\pi(0))(x)\\
    =&2\Im q_{i}(x)\int_{0}^{l_{\gamma}}\Im q_{j}(\Phi_{s}(x))(\frac{e^{-s}}{e^{-l_{\gamma}}-1}-\frac{e^{s}}{e^{l_{\gamma}}-1})\mathrm{d}s
\end{align*}
\end{itemize}

Similar to the cases of $\partial_{\beta} g_{\alpha \alpha}(\sigma)$ and $\partial_{i} g_{\alpha \alpha}(\sigma)$, one can then define a function $\eta: UX \to \mathbb{R}$ 
\begin{align*}
    \eta(x)=-2\Im q_{i}(x)\left(\int_{0}^{\infty}\Im q_{j}(\Phi_{s}(x))e^{-s}\mathrm{d}s+\int_{-\infty}^{0}\Im q_{j}(\Phi_{s}(x))e^{s}\mathrm{d}s\right)
\end{align*}
and verify that $\eta(x)$ is H{\"o}lder such that $\Tr(\partial_{v}D_{H(0)}\partial_{w}\pi(0))(x)\equiv \eta(x)$ on $UX$.

We finally obtain
\begin{align*}
 \partial_{vw}f_{\rho(0)}(x)\sim&-\Tr(\frac{\partial^2 D_{H(0)}}{\partial v \partial w}\pi(0))(x) - \Tr(\partial_{v}D_{H(0)}\partial_{w}\pi(0))(x)\\ 
 =&\frac{1}{2}\phi_{vw}(p(x))+2\Im q_{i}(x)\left(\int_{0}^{\infty}\Im q_{j}(\Phi_{s}(x))e^{-s}\mathrm{d}s+\int_{-\infty}^{0}\Im q_{j}(\Phi_{s}(x))e^{s}\mathrm{d}s\right)
\end{align*}
\end{proof}

\subsubsection{Evaluation on Poincar\'e disk}

We show in this subsection
\begin{prop}
\label{prop Case3}
 For $\sigma\in\mathcal{T}(S),\partial_{j} g_{\alpha i}(\sigma)=0$.
\end{prop}
For the same reasoning as before, the proof of the above proposition reduces to the following lemma.

\begin{lem} \label{lem 3cubicIII}
We have the following holds for any $t,s\in \mathbb{R}$,
\begin{align}
&\int_{\Sub UX}\Re q_{i}(x)\Re q_{j}(\Phi_{t}(x))\Re q_{\alpha}(\Phi_{s}(x))\mathrm{d}m_{0}(x)=0 \label{eq 3cubic6}\\
&\int_{\Sub UX}\Re q_{i}(x)\Im q_{j}(\Phi_{t}(x))\Im q_{\alpha}(\Phi_{s}(x))\mathrm{d}m_{0}(x)=0 \label{eq 3cubicMistake+7}\\
&\int_{\Sub UX}\Re q_{\alpha}(x)\Re q_{i}(\Phi_{t}(x))\Re q_{j}(\Phi_{s}(x))\mathrm{d}m_{0}(x)=0 \label{eq 3cubic7}\\
&\int_{\Sub UX}\Re q_{\alpha}(x)\Im q_{i}(\Phi_{t}(x))\Im q_{j}(\Phi_{s}(x))\mathrm{d}m_{0}(x)=0 \label{eq 3cubic8}
\end{align}
\end{lem}

\begin{proof}
We just need to show equation \eqref{eq 3cubic6}. Equations \eqref{eq 3cubicMistake+7}, \eqref{eq 3cubic7} and \eqref{eq 3cubic8} follow easily using methods we developed in the former cases. 

We start from a special case of equation \text{\eqref{eq 3cubic6}} that $q_i=q_j$
\begin{align}
\int_{\Sub UX}\Re q_{i}(x)\Re q_{i}(\Phi_{t}(x))\Re q_{\alpha}(\Phi_{s}(x))\mathrm{d}m_{0}(x)=0 \hspace{.5in}\text{$\forall t,s\in \mathbb{R}$}
\label{eq 3cubicIV}
\end{align}

The proof of this case is an analogy of the case $\partial_{\beta}g_{\alpha \alpha}(\sigma)$ because of the following observations for flow time $s=t$ and $s=\frac{t}{2}$.
\begin{align*}
&\int_{\Sub UX}\Re q_{i}(x)\Re q_{i}(\Phi_{t}(x))\Re q_{\alpha}(\Phi_{t}(x))\mathrm{d}m_{0}(x) \nonumber \\
 =&-\int_{\Sub UX}\Re q_{i}(x)\Re q_{i}(\Phi_{t}(x))\Re q_{\alpha}(x)\mathrm{d}m_{0}(x) 
 \end{align*}

and
\begin{align*}
\int_{\Sub UX}\Re q_{i}(x)\Re q_{i}(\Phi_{t}(x))\Re q_{\alpha}(\Phi_{\frac{t}{2}}(x))\mathrm{d}m_{0}(x)=0
\end{align*}

For $t,s > 0$, recall our analytic expansions given in \eqref{eq: analyticExI} and  \eqref{eq: analyticExIII} lead to

\begin{align*}
&\int_{\Sub UX}\Re q_{i}(x)\Re q_{i}(\Phi_{t}(x))\Re q_{\alpha}(\Phi_{s}(x))\mathrm{d}m_{0}(x)\\
=&\frac{1}{2\pi}\int_{0}^{2\pi}\int_{\Sub UX}\Re q_{i}(e^{i\theta}x)\Re q_{i}(\Phi_{t}(e^{i\theta}x))\Re q_{\alpha}(\Phi_{s}(e^{i\theta}x))\mathrm{d}m_{0}(x)\mathrm{d}\theta\\
=&\frac{1}{4}\sum\limits_{n=0}^{\infty}\bigg(\int_{UX}\Re(c_{0}c_{n}\bar{a}_{n+1})\mathrm{d}m_{0}T^{n}(1-T^{2})^{2}S^{n+1}(1-S^{2})^{3}\\
+&\int_{UX}\Re(c_{0}\bar{c}_{n+3}a_{n})\mathrm{d}m_{0}T^{n+3}(1-T^{2})^{2}S^{n}(1-S^{2})^{3}\bigg)
\end{align*}

Denoting $E_{n}=\int_{UX}\Re(c_{0}c_{n}\bar{a}_{n+1})\mathrm{d}m_{0}$ and $F_{n}=\int_{UX}\Re(c_{0}\bar{c}_{n+3}a_{n})\mathrm{d}m_{0}$. We argue for $n\geq0$,
\begin{equation}
   E_n=F_n=0
   \label{eq coeffIII}
\end{equation}

The case $t=0$ or $s=0$ of equation \eqref{eq 3cubicIV} is included in the $n=0$ case of equation \eqref{eq coeffIII}.

When flow time $s=t$, we have

\begin{align}
&\sum\limits_{n=0}^{\infty}\bigg(E_n T^{2n+1}(1-T^{2})^{5}+F_nT^{2n+3}(1-T^{2})^{5}\bigg)\nonumber\\
&=-F_0T^3(1-T^2)^2 \label{eq relation2}
\end{align}
This implies
\begin{align*}
    E_0&=0\\
    E_1&=-2F_0
\end{align*}

When flow time $s=\frac{t}{2}$, we obtain
\begin{align*}
\sum\limits_{n=0}^{\infty}\bigg(E_n T^{n}(1-T^{2})^{2}S^{n+1}(1-S^{2})^{3}+F_nT^{n+3}(1-T^{2})^{2}S^{n}(1-S^{2})^{3}\bigg)=0
\end{align*}
where $T=\frac{2S}{1+S^2}$.

It simplifies to
\begin{align*}
    \sum_{n=0}^{\infty}(E_n+8F_n\frac{S^2}{(S^2+1)^3})(\frac{2S^2}{S^2+1})^n=0
\end{align*}
Let $W=\frac{S^2}{1+S^2}$, we have
\begin{align*}
    \sum_{n=0}^{\infty}(E_n\sum_{k=0}^{\infty}(k+1)W^k+8F_nW)(2W)^n=0
\end{align*}
This gives relations
\begin{align}
   E_0=&0\nonumber\\
   \sum_{k=0}^{n}2^k(n-k+1)E_k+2^{n+2}F_{n-1}=&0 \hspace{.5 in} \text{$n\geq 1$} \label{eq relation3}
\end{align}
Combing with equation\eqref{eq relation2}, we get $E_1=F_0=0$. Therefore the right hand side of equation\eqref{eq relation2} is zero and we obtain from it $E_{n+1}+F_n=0$ for $n\geq 0$. Combining it with \eqref{eq relation3} and by an induction argument, one concludes $E_n=F_n=0$. This proves equation \eqref{eq coeffIII} for $s,t \geq 0$. The case $s,t < 0$ is similar as before.  

Now we proceed to prove equation (\ref{eq 3cubic6}). The above case implies for $q_i\neq q_j$,
\begin{align*}
    \int_{UX}\Re q_{i}(x)\Re q_{i}(\Phi_{t}(x)) \Re q_{\alpha}(\Phi_{s}(x))\mathrm{d}m_{0}&=0\\
\int_{UX}\Re q_{j}(x)\Re q_{j}(\Phi_{t}(x)) \Re q_{\alpha}(\Phi_{s}(x))\mathrm{d}m_{0}&=0\\
    \int_{UX}\Re (q_{i}+q_{j})(x) \Re (q_{i}+ q_{j})(
   \Phi_{t}(x)) \Re q_{\alpha}(\Phi_{s}(x))\mathrm{d}m_{0}&=0
\end{align*} 

Therefore for all $t,s \in \mathbb{R}$,
\begin{equation}
    \int_{UX}\Re q_{i}(x)\Re q_{j}(\Phi_{t}(x)) \Re q_{\alpha}(\Phi_{s}(x))\mathrm{d}m_{0}+\int_{UX}\Re q_{j}(x)\Re q_{i}(\Phi_{t}(x)) \Re q_{\alpha}(\Phi_{s}(x))\mathrm{d}m_{0}=0
    \label{eq exchange}
\end{equation}

Recall the analytic expansion for $q_j$ is given in \eqref{eq: analyticExIIII}.  Consider $t,s > 0$.
\begin{align*}
&\int_{\Sub UX}\Re q_{i}(x)\Re q_{j}(\Phi_{t}(x))\Re q_{\alpha}(\Phi_{s}(x))\mathrm{d}m_{0}(x)\\
=&\frac{1}{2\pi}\int_{0}^{2\pi}\int_{\Sub UX}\Re q_{i}(e^{i\theta}x)\Re q_{j}(\Phi_{t}(e^{i\theta}x))\Re q_{\alpha}(\Phi_{s}(e^{i\theta}x))\mathrm{d}m_{0}(x)\mathrm{d}\theta\\
=&\frac{1}{4}\sum\limits_{n=0}^{\infty}\bigg(\int_{UX}\Re(c_{0}d_{n}\bar{a}_{n+1})\mathrm{d}m_{0}T^{n}(1-T^{2})^{2}S^{n+1}(1-S^{2})^{3}\\
+&\int_{UX}\Re(c_{0}\bar{d}_{n+3}a_{n})\mathrm{d}m_{0}T^{n+3}(1-T^{2})^{2}S^{n}(1-S^{2})^{3}\bigg)
\end{align*}
Denoting $G_{n}=\int_{UX}\Re(c_{0}d_{n}\bar{a}_{n+1})\mathrm{d}m_{0}$ and $H_{n}=\int_{UX}\Re(c_{0}\bar{d}_{n+3}a_{n})\mathrm{d}m_{0}$. We want to show $G_n=H_n=0$ for $n \geq 0$.

Let $m$ be an integer and $m\geq 2$. Let flow time $s=mt$. Observe we have
\begin{align*}
&\int_{\Sub UX}\Re q_{i}(x)\Re q_{j}(\Phi_{t}(x))\Re q_{\alpha}(\Phi_{mt}(x))\mathrm{d}m_{0}(x)\\
 =&\int_{\Sub UX}\Re q_{i}(\Phi_{-t}(x))\Re q_{j}(x)\Re q_{\alpha}(\Phi_{(m-1)t}(x))\mathrm{d}m_{0}(x)\\
 =&-\int_{\Sub UX}\Re q_{j}(y)\Re q_{i}(\Phi_{t}(y))\Re q_{\alpha}(\Phi_{-(m-1)t}(y))\mathrm{d}m_{0}(y) \hspace{.5in}\text{$y=-x$ and $\Phi_{t}(-y)=-\Phi_{-t}(y)$}\\
 =&\int_{UX}\Re q_{i}(y)\Re q_{j}(\Phi_{t}(y)) \Re q_{\alpha}(\Phi_{-(m-1)t}(y))\mathrm{d}m_{0}(y) \hspace{.5in}\text{by equation \eqref{eq exchange}}\\
=-&\int_{\Sub UX}\Re q_{j}(x)\Re q_{i}(\Phi_{t}(x))\Re q_{\alpha}(\Phi_{mt}(x))\mathrm{d}m_{0}(x)\hspace{.5in}\text{ exchange the roles of $q_i$ and $q_j$}\\
=&-\int_{UX}\Re q_{i}(x)\Re q_{j}(\Phi_{t}(x)) \Re q_{\alpha}(\Phi_{(m-1)t}(-x))\mathrm{d}m_{0}(x)
\end{align*}

When $s=mt$, we have $S=S(m)=\frac{(1+T)^m-(1-T)^m}{(1+T)^m+(1-T)^m}=mT+O(T^3)$. From the analytic expansion
\begin{align*}
&\sum\limits_{n=0}^{\infty}\bigg(G_{n}T^{n}S(m)^{n+1}(1-S(m)^{2})^{3}+G_{n}e^{in\pi}T^{n}S(m-1)^{n+1}(1-S(m-1)^{2})^{3}\bigg)\\
=&-\sum\limits_{n=0}^{\infty}\bigg(-H_{n}e^{in\pi}T^{n+3}S(m-1)^{n}(1-S(m-1)^{2})^{3}+H_{n}T^{n+3}S(m)^{n}(1-S(m)^{2})^{3}\bigg)
\end{align*}

The coefficients of $T^{1}$ and $T^{3}$ and $T^5$ yield the following respectively
\begin{align*}
   G_{0}&=0\\
   (m^2-(m-1)^2)G_{1}&=0\\
   (m^3+(m-1)^3)G_2&=-(2m-1)H_1+(6m-3)H_0
\end{align*}
The cases $m=2$, $m=3$ and $m=4$ together give $H_0=H_1=G_2=0$. By induction, assuming $G_{k}=H_{k-1}=0$ for $1\leq k<n$, the coefficient of $T^{2n+1}$ gives  
\begin{align*}
(m^{n+1}+e^{in\pi}(m-1)^{n+1})G_n=(e^{i(n-1)\pi}(m-1)^{n-1}-m^{n-1})H_{n-1}
\end{align*}
We conclude $G_n=H_n=0$ for $n\geq 0$ by choosing two different $m$. This finishes the proof of equation (\ref{eq 3cubic6}) for $t,s>0$.
Equation (\ref{eq 3cubic6}) for $t\leq 0$ and $s \leq 0$ can be proved similar to the former cases.

\end{proof}

\subsection{The case of $\partial_{\beta}g_{\alpha i}(\sigma)$}

This is the last case. In this case, the representations $\{\rho(u,v,w)\}$ in $\mathcal{H}_{3}(S)$ corresponds to $\{(vq_{i},uq_{\alpha}+wq_{\beta})\}\subset H^{0}(X,K^2)\bigoplus H^{0}(X,K^3)$ by Hitchin parametrization. Our metric tensor is
\begin{align*}
\partial_{\beta}g_{\alpha i}(\sigma)=&\partial_{w} \Bigg( {\langle \partial_{u}\rho(0,0,w), \partial_{v}\rho(0,0,w) \rangle}_{\Sub P}\Bigg) \Bigg|_{w=0}\\
 &=\lim\limits_{r\to \infty}\frac{1}{r}\Bigg(\int_{\Sub UX}\int_{0}^{r} \partial_{u}f_{\rho(0)}^{N}\mathrm{d}t \int_{0}^{r}\partial_{v}f_{\rho(0)}^{N}\mathrm{d}t\int_{0}^{r} \partial_{w}f_{\rho(0)}^{N}\mathrm{d}t\mathrm{d}m_{0}\\
 &+\int_{\Sub UX} \int_{0}^{r}\partial_{u}f_{\rho(0)}^{N} \mathrm{d}t\int_{0}^{r} \partial_{vw}f_{\rho(0)}^{N}\mathrm{d}t\mathrm{d}m_{0}+\int_{\Sub UX} \int_{0}^{r}\partial_{v}f_{\rho(0)}^{N} \mathrm{d}t\int_{0}^{r} \partial_{uw}f_{\rho(0)}^{N}\mathrm{d}t\mathrm{d}m_{0}\Bigg)
 \end{align*}
 where the first and second variations are 
\begin{enumerate}[label=(\roman*)]
    \item $\partial_{u}f_{\rho(0)}^{N}=-\partial_{u}f_{\rho(0)}$;
    \item $\partial_{v}f_{\rho(0)}^{N}=-\partial_{v}f_{\rho(0)}$;
    \item $\partial_{uw}f_{\rho(0)}^{N}=-\partial_{uw}h(\rho(0))-\partial_{uw}f_{\rho(0)}$;
    \item $\partial_{vw}f_{\rho(0)}^{N}=-\partial_{vw}h(\rho(0))-\partial_{vw}f_{\rho(0)}$.
\end{enumerate}

 \subsubsection{First and second variations of reparametrization functions}
 
 Our Higgs fields in this case are
\[
  \Phi(u,v,w)=
  \begin{bmatrix}
   0 & vq_{i} & uq_{\alpha}+wq_{\beta} \\
   1 & 0 & vq_{i} \\
   0 & 1& 0 \\
  \end{bmatrix}
\]

\begin{prop}
\label{prop case1rep4}
The first variations of reparametrization functions $\partial_{u}f_{\rho(0)}:UX \to \mathbb{R}$ and $\partial_{v}f_{\rho(0)}:UX \to \mathbb{R}$ for the case $\partial_{\beta}g_{\alpha i}(\sigma)$ satisfy 
 \begin{align*}
  \partial_{u}f_{\rho(0)}(x)& \sim -Re q_{\alpha}(x),\\  
  \partial_{v}f_{\rho(0)}(x)&\sim 2Re q_{i}(x),\\
   \partial_{w}f_{\rho(0)}(x)&\sim -Re q_{\beta}(x).
\end{align*}
and the second variations of reparametrization functions $\partial_{uw}f_{\rho(0)}:UX \to \mathbb{R}$ and $\partial_{vw}f_{\rho(0)}:UX \to \mathbb{R}$ satisfy
\begin{align*}
\partial_{uw}f_{\rho(0)}(x)\sim &\frac{1}{2}\phi_{uw}(p(x)) \nonumber \\
 +&\Re q_{\alpha}(x)\int_{0}^{\infty}e^{-2s}\Re q_{\beta}(\Phi_{s}(x))\mathrm{d}s+\Re q_{\alpha}(x)\int_{-\infty}^{0}e^{2s}\Re q_{\beta}(\Phi_{s}(x))\mathrm{d}s \nonumber \\
    +&2\Im q_{\alpha}(x)\int_{0}^{\infty}e^{-s}\Im q_{\beta}(\Phi_{s}(x))\mathrm{d}s+2\Im q_{\alpha}(x)\int_{-\infty}^{0}e^{s}\Im q_{\beta}(\Phi_{s}(x))\mathrm{d}s,\\
 \partial_{vw}f_{\rho(0)}(x)=\partial_{wv}f_{\rho(0)}(x)\sim& \frac{1}{2}\Re y_{21}(x)-2\Im q_{\beta}(x)\left(\int_{0}^{\infty}\Im q_{i}(\Phi_{s}(x))e^{-s}\mathrm{d}s+\int_{-\infty}^{0}\Im q_{i}(\Phi_{s}(x))e^{s}\mathrm{d}s\right).
\end{align*}
where $p: UX \to X$ and $y_{21}$ are defined as before.
\end{prop}

\begin{proof}
All of the computations have been done in the former cases.
\end{proof}

\subsubsection{Evaluation on Poincar\'e disk}

We show in this subsection
\begin{prop}
\label{prop Case3}
 For $\sigma\in\mathcal{T}(S),\partial_{\beta} g_{\alpha i}(\sigma)=0$.
\end{prop}
For the same reasoning as before, the proof of the above proposition reduces to the following lemma.

\begin{lem} \label{lem 3cubicIII}
We have the following holds for any $t,s\in \mathbb{R}$,
\begin{align}
&\int_{UX}\Re q_{\alpha}(x) \Re q_{\beta}(\Phi_{t}(x)) \Re q_{i}(\Phi_{s}(x))\mathrm{d}m_{0}(x)=0 \label{eq 3cubic9}\\
&\int_{UX}\Im q_{\alpha}(x) \Im q_{\beta}(\Phi_{t}(x)) \Re q_{i}(\Phi_{s}(x))\mathrm{d}m_{0}(x)=0 \label{eq 3cubic10}\\
&\int_{UX}\Re q_{\alpha}(x) \Im q_{\beta}(\Phi_{t}(x)) \Im q_{i}(\Phi_{s}(x))\mathrm{d}m_{0}(x)=0 \label{eq 3cubic11}
\end{align}
\end{lem}

\begin{proof}
We just need to show equation \eqref{eq 3cubic9}. Equations \text{\eqref{eq 3cubic10}} and \text{\eqref{eq 3cubic11}} follow easily similar to formal cases.

From the computation of $\partial_{i}g_{\alpha \alpha}(\sigma)$, we know    
\begin{align*}
    \int_{UX}\Re q_{\alpha}(x)\Re q_{\alpha}(\Phi_{t}(x)) \Re q_{i}(\Phi_{s}(x))\mathrm{d}m_{0}&=0\\
\int_{UX}\Re q_{\beta}(x)\Re q_{\beta}(\Phi_{t}(x)) \Re q_{i}(\Phi_{s}(x))\mathrm{d}m_{0}&=0\\
    \int_{UX}\Re (q_{\alpha}+q_{\beta})(x) \Re (q_{\alpha}+ q_{\beta})(\Phi_{t}(x)) \Re q_{i}(\Phi_{s}(x))\mathrm{d}m_{0}&=0
\end{align*} 
We deduce 
\begin{align*}
    \int_{UX}\Re q_{\alpha}(x)\Re q_{\beta}(\Phi_{t}(x)) \Re q_{i}(\Phi_{s}(x))\mathrm{d}m_{0}+\int_{UX}\Re q_{\beta}(x)\Re q_{\alpha}(\Phi_{t}(x)) \Re q_{i}(\Phi_{s}(x))\mathrm{d}m_{0}&=0
\end{align*} 

Similar to $\partial_{j}g_{\alpha i}(\sigma)$, we consider $s=mt$ for $m \in \mathbb{N}$ and $m\geq 2$. We observe
\begin{align*}
    &\int_{UX}\Re q_{\alpha}(x)\Re q_{\beta}(\Phi_{t}(x)) \Re q_{i}(\Phi_{mt}(x))\mathrm{d}m_{0}\\
    =&\int_{UX}\Re q_{\alpha}(x)\Re q_{\beta}(\Phi_{t}(x)) \Re q_{i}(\Phi_{(m-1)t}(-x))\mathrm{d}m_{0}\\
\end{align*} 
We recall the Poinc\'are disk model and our analytic expansion for  $q_{\alpha},q_{\beta}, q_{i}$ in \eqref{eq: analyticExI}, \eqref{eq: analyticExII} and \eqref{eq: analyticExIIII}. For $t,s \geq 0$, the analytic expansio
\begin{align*}
&\int_{\Sub UX}\Re q_{\alpha}(x)\Re q_{\beta}(\Phi_{t}(x))\Re q_{i}(\Phi_{s}(x))\mathrm{d}m_{0}(x)\\
=&\frac{1}{2\pi}\int_{0}^{2\pi}\int_{\Sub UX}\Re q_{\alpha}(e^{i\theta}x)\Re q_{\beta}(\Phi_{t}(e^{i\theta}x))\Re q_{i}(\Phi_{s}(e^{i\theta}x))\mathrm{d}m_{0}(x)\mathrm{d}\theta\\
=&\frac{1}{4}\sum\limits_{n=0}^{\infty}\bigg(\int_{UX}\Re(a_{0}b_{n}\bar{c}_{n+4})\mathrm{d}m_{0}T^{n}(1-T^{2})^{3}S^{n+4}(1-S^{2})^{2}\\
+&\int_{UX}\Re(a_{0}\bar{b}_{n+2}c_{n})\mathrm{d}m_{0}S^{n}(1-S^{2})^{2}T^{n+2}(1-T^{2})^{3}\bigg)
\end{align*}
Denoting $I_n=\int_{UX}\Re(a_{0}b_{n}\bar{c}_{n+4})\mathrm{d}m_{0}$ and $J_n=\int_{UX}\Re(a_{0}\bar{b}_{n+2}c_{n})\mathrm{d}m_{0}$ for $n\geq0$,we argue
\begin{align*}
  I_n=J_n=0.
\end{align*}

When $s=mt$, we have $S=S(m)=\frac{(1+T)^m-(1-T)^m}{(1+T)^m+(1-T)^m}=mT+O(T^3)$. The analytic expansions give
\begin{align*}
&\sum\limits_{n=0}^{\infty}\bigg(I_{n}T^{n}S(m)^{n+4}(1-S(m)^{2})^{2}-I_{n}e^{in\pi}T^{n}S(m-1)^{n+4}(1-S(m-1)^{2})^{2}\bigg)\\
=&\sum\limits_{n=0}^{\infty}\bigg(-J_{n}T^{n+2}S(m)^{n}(1-S(m)^{2})^{2}+J_{n}e^{in\pi}T^{n+2}S(m-1)^{n}(1-S(m-1)^{2})^{2}\bigg)
\end{align*}

The coefficients of $T^{4}$  yield the following respectively
\begin{align*}
   (m^4-(m-1)^4)I_{0}&=-(2m-1)J_1+(4m-2)J_0\\
\end{align*}
The cases $m=2$ and $m=3$ and $m=4$ gives $I_0=J_0=J_1=0$. By induction, assuming $I_{k}=J_{k+1}=0$ for $1\leq k<n$, the coefficient of $T^{2n+4}$ gives  
\begin{align*}
(m^{n+4}-e^{in\pi}(m-1)^{n+4})I_n=(e^{i(n+1)\pi}(m-1)^{n+1}-m^{n+1})J_{n+1}
\end{align*}
We conclude $I_n=J_n=0$ for $n\geq 0$ by choosing two different $m$. This finishes the proof of equation (\ref{eq 3cubic9}) for $t,s>0$.
Equation (\ref{eq 3cubic9}) for $t\leq 0$ and $s \leq 0$ can be proved similar to the former cases.
 Lemma \ref{lem 3cubicIII} and also Proposition \ref{prop case1rep3} therefore hold.
 \end{proof}
We have shown (i) $\partial_{\beta}g_{\alpha \alpha}(\sigma)=0$, (ii) $\partial_{i}g_{\alpha \alpha}(\sigma)=0$, (iii) $\partial_{j}g_{\alpha i}(\sigma)=0$ and (iv) $\partial_{\beta}g_{\alpha i}(\sigma)=0$ in consecutive sections. This finishes the proof of our Theorem \ref{thm main}.

\bibliographystyle{abbrv}
	\bibliography{ref}
\end{document}